%% file: arxiv.tex
\documentclass{article}
\usepackage{amssymb}
\usepackage{amsmath}
\usepackage{amsthm}
\usepackage{enumitem}
\usepackage{xcolor}
\usepackage{graphicx}
\usepackage{fullpage}

\title{Deep learning: a statistical viewpoint}
\author{%
Peter L.~Bartlett\thanks{
 Departments of Statistics and EECS,
UC Berkeley}\\
{\tt peter@berkeley.edu}\\
\and
Andrea Montanari\thanks{Departments of EE and Statistics,
Stanford University}\\
{\tt montanar@stanford.edu}\\
\and
Alexander Rakhlin\thanks{
Department of Brain \& Cognitive Sciences and Statistics \& Data Science Center,
MIT}\\
{\tt rakhlin@mit.edu}
}

\usepackage{hyperref}
\hypersetup{
    colorlinks,
    linkcolor={blue!80!black},
    citecolor={green!50!black},
}
\colorlet{linkequation}{blue}

\input{macros}

\input{macros_am}
\newcommand{\Expect}{{\mathbb E}}

\newcommand{\calX}{{\mathcal X}}
\newcommand{\calY}{{\mathcal Y}}
\newcommand{\calZ}{{\mathcal Z}}
\renewcommand{\Re}{{\mathbb R}}
\newcommand{\cas}{\overset{as}{\rightarrow}}    

\newcommand{\ignore}[1]{}

\newtheorem{theorem}{Theorem}[section]
\newtheorem{lemma}[theorem]{Lemma}
\newtheorem{corollary}[theorem]{Corollary}
\newtheorem{example}[theorem]{Example}
\newtheorem{assumption}[theorem]{Assumption}
\newtheorem{remark}[theorem]{Remark}

\newtheoremstyle{myremark} 
    {\topsep}                    
    {\topsep}                    
    {\rm}                        
    {}                           
    {\bf}                        
    {.}                          
    {.5em}                       
    {}  

\theoremstyle{myremark}
\newtheorem{newremark}{Remark}[section]

\begin{document}

\maketitle

\begin{abstract}
The remarkable practical success of deep learning has revealed some
major surprises from a theoretical perspective.  In particular,
simple gradient methods easily find near-optimal solutions to
non-convex optimization problems, and despite giving a near-perfect
fit to training data without any explicit effort to control model
complexity, these methods exhibit excellent predictive accuracy.
We conjecture that specific principles underlie these phenomena: that
overparametrization allows gradient methods to find interpolating
solutions, that these methods implicitly impose regularization,
and that overparametrization leads to benign overfitting, that is,
accurate predictions despite overfitting training data.  In this
article, we survey recent progress in statistical learning theory
that provides examples illustrating these principles in simpler
settings.  We first review classical uniform convergence results
and why they fall short of explaining aspects of the behavior of
deep learning methods.  We give examples of implicit regularization
in simple settings, where gradient methods lead to minimal norm
functions that perfectly fit the training data.  Then we review
prediction methods that exhibit benign overfitting, focusing
on regression problems with quadratic loss.  For these methods,
we can decompose the prediction rule into a simple component that
is useful for prediction and a spiky component that is useful for
overfitting but, in a favorable setting, does not harm prediction
accuracy.  We focus specifically on the linear regime for neural
networks, where the network can be approximated by a linear model.
In this regime, we demonstrate the success of gradient flow, and
we consider benign overfitting with two-layer networks, giving an
exact asymptotic analysis that precisely demonstrates the impact of
overparametrization.  We conclude by highlighting the key challenges
that arise in extending these insights to realistic deep learning
settings.
\end{abstract}

\tableofcontents

\input{intro}

\input{slt}

\input{implicit}

\input{benign}

\input{efficient}

\input{ntk}

\input{future}

\section*{Acknowledgements}

PB, AM and AR acknowledge support from the NSF through award
DMS-2031883 and from the Simons Foundation through Award 814639
for the Collaboration on the
Theoretical Foundations of Deep Learning.
For insightful discussions on these topics, the authors also
thank the other members of that Collaboration
and many other collaborators and colleagues, including
Emmanuel Abbe,
Misha Belkin,
Niladri Chatterji,
Amit Daniely,
Tengyuan Liang,
Philip Long,
G\'abor Lugosi,
Song Mei,
Theodor Misiakiewicz,
Hossein Mobahi,
Elchanan Mossel,
Phan-Minh Nguyen,
Nati Srebro,
Nike Sun,
Alexander Tsigler,
Roman Vershynin,
and
Bin Yu. 
We thank Tengyuan Liang and Song Mei for insightful comments on the draft.
PB acknowledges support from the NSF through grant DMS-2023505.
AM acknowledges support from the ONR through grant N00014-18-1-2729.
AR acknowledges support from the NSF through grant DMS-1953181, and support from the MIT-IBM Watson AI Lab and the NSF AI Institute for Artificial Intelligence and Fundamental Interactions.

\bibliography{refs}

\appendix
\input{appendix_body}

\end{document}

%% file: macros.tex
\newcommand{\biassquared}{{\textsc{bias}^2}}

\newcommand{\biassquaredemp}{{\widehat{\textsc{bias}}^2}}
\newcommand{\varianceemp}{{\widehat{\textsc{var}}}}

\def\deq{:=}
\newcommand{\reals}{\ensuremath{\mathbb R}}
\newcommand{\tr}{\ensuremath{{\scriptscriptstyle\mathsf{\,T}}}} 
\newcommand{\trace}{\mathsf{tr}}

\newcommand{\indicator}[1]{\ensuremath{\boldsymbol{1}\left[#1\right]}}
\newcommand{\inner}[1]{\left\langle #1 \right\rangle}
\newcommand{\sign}{\mathrm{sign}}

\newcommand{\bone}{\ensuremath{\mathbf{1}}}
\newcommand{\norm}[1]{\left\|#1\right\|}

\newcommand{\argmin}[1]{\underset{#1}{\mathrm{argmin}} \ }

\newcommand{\En}{\mathbb{E}}  

\newcommand{\algo}{\widehat{f}}
\newcommand{\prederm}{\widehat{f}_{\scalebox{.4}{erm}}}

\newcommand{\be}{\boldsymbol{e}}
\newcommand{\by}{\boldsymbol{y}}
\newcommand{\br}{\boldsymbol{r}}
\newcommand{\bw}{\boldsymbol{w}}
\newcommand{\bepsilon}{\boldsymbol{\epsilon}}

\def\ddefloop#1{\ifx\ddefloop#1\else\ddef{#1}\expandafter\ddefloop\fi}
\def\ddef#1{\expandafter\def\csname c#1\endcsname{\ensuremath{\mathcal{#1}}}}

\ddefloop ABCDEFGHIJKLMNOPQRSTUVWXYZ\ddefloop
\def\ddef#1{\expandafter\def\csname s#1\endcsname{\ensuremath{\mathsf{#1}}}}
\ddefloop ABCDEFGHIJKLMNOPQRSTUVWXYZ\ddefloop

\ddefloop ABCDEFGHIJKLMNOPQRSTUVWXYZ\ddefloop
\def\ddef#1{\expandafter\def\csname b#1\endcsname{\ensuremath{\mathbf{#1}}}}
\ddefloop ABCDEFGHIJKLMNOPQRSTUVWXYZ\ddefloop

%% file: macros_am.tex
\usepackage[mathscr]{euscript}
\usepackage{accents}
\newcommand{\dbtilde}[1]{\accentset{\approx}{#1}}

\DeclareSymbolFont{rsfs}{U}{rsfs}{m}{n}
\DeclareSymbolFontAlphabet{\mathscrsfs}{rsfs}

\renewcommand{\P}{\mathbb{P}}
\newcommand{\E}{\mathbb{E}}
\newcommand{\Var}{\text{Var}}

\renewcommand{\S}{\mathbb{S}}

\def\reals{{\mathbb R}}

\newcommand{\eps}{\varepsilon} 

\def\id{{\mathbf I}}

\newcommand{\<}{\langle}
\renewcommand{\>}{\rangle}
\newcommand{\diag}{\text{diag}}
\newcommand{\op}{}

\newcommand{\grad}{\nabla}
\def\sT{{\mathsf T}}

\def\tg{\tilde{g}}

\def\bA{{\boldsymbol A}}
\def\bB{{\boldsymbol B}}

\def\bD{{\boldsymbol D}}

\def\bF{{\boldsymbol F}}

\def\bK{{\boldsymbol K}}

\def\bM{{\boldsymbol M}}

\def\bP{{\boldsymbol P}}
\def\bQ{{\boldsymbol Q}}
\def\bR{{\boldsymbol R}}
\def\bS{{\boldsymbol S}}

\def\bU{{\boldsymbol U}}
\def\bV{{\boldsymbol V}}
\def\bW{{\boldsymbol W}}
\def\bX{{\boldsymbol X}}
\def\bY{{\boldsymbol Y}}
\def\bZ{{\boldsymbol Z}}

\def\tZ{\tilde{Z}}

\def\tbZ{\tilde{\boldsymbol Z}}
\def\tbg{\tilde{\boldsymbol g}}

\def\ba{{\boldsymbol a}}
\def\bb{{\boldsymbol b}}
\def\be{{\boldsymbol e}}

\def\bg{{\boldsymbol g}}

\def\bu{{\boldsymbol u}}
\def\bv{{\boldsymbol v}}
\def\bw{{\boldsymbol w}}
\def\bx{{\boldsymbol x}}
\def\by{{\boldsymbol y}}
\def\bz{{\boldsymbol z}}

\def\bbeta{{\boldsymbol \beta}}

\def\bpsi{{\boldsymbol \psi}}
\def\bphi{{\boldsymbol \phi}}

\def\hbbeta{\widehat{\boldsymbol \beta}}
\def\hbeta{\widehat{\beta}}
\def\btheta{{\boldsymbol \theta}}

\def\bDelta{{\boldsymbol \Delta}}

\def\bPhi{{\boldsymbol \Phi}}
\def\bSigma{{\boldsymbol \Sigma}}
\def\bTheta{{\boldsymbol \Theta}}

\def\bfzero{{\boldsymbol 0}}
\def\bfone{{\boldsymbol 1}}

\def\hba{{\widehat {\boldsymbol a}}}
\def\hf{{\widehat f}}

\def\supp{{\rm supp}}

\def\de{{\rm d}}
\def\He{{\rm He}}
\def\Trace{{\trace}}
\def\lin{{\rm lin}}

\def\de{{\rm d}}
\def\Unif{{\rm Unif}}
\def\He{{\rm He}}
\def\lin{{\rm lin}}

\def\RF{{\rm RF}}
\def\NT{{\rm NT}}
\def\NF{{\rm NF}}

\def\cV{{\mathcal V}}
\def\cG{{\mathcal G}}

\def\cF{{\mathcal F}}
\def\cE{{\mathcal E}}
\def\cS{{\mathcal S}}
\def\cI{{\mathcal I}}
\def\cV{{\mathcal V}}
\def\cG{{\mathcal G}}

\def\cH{{\mathcal H}}

\def\Unif{{\sf Unif}}
\def\normal{{\sf N}}
\def\proj{{\mathsf P}}

\def\RF{{\sf RF}}
\def\NT{{\sf NT}}

\def\reals{{\mathbb R}}
\def\integers{{\mathbb Z}}
\def\naturals{{\mathbb N}}

\def\normal{{\sf N}}
\def\proj{{\mathsf P}}

\def\Unif{{\sf Unif}}
\def\normal{{\sf N}}

\def\proj{{\mathsf P}}

\def\RF{{\sf RF}}
\def\NT{{\sf NT}}

\def\reals{{\mathbb R}}
\def\integers{{\mathbb Z}}
\def\naturals{{\mathbb N}}
\def\proj{{\mathsf P}}

\def\tR{\tilde{R}}
\def\ttR{\dbtilde{R}}
\def\tba{\tilde{\boldsymbol a}}
\def\bq{{\boldsymbol q}}

\def\Lip{{\rm Lip}}
\def\sL{\mbox{\rm\tiny L}}
\def\sNL{\mbox{\rm\tiny NL}}
\def\seff{\mbox{\rm\tiny eff}}
\def\smix{\mbox{\rm\tiny mix}}
\def\seq{\mbox{\rm\tiny seq}}
\def\slin{\mbox{\rm\tiny lin}}
\def\sTV{\mbox{\rm\tiny TV}}
\def\sGF{\mbox{\rm\tiny GF}}
\def\samp{\mbox{\rm\tiny s}}
\def\wid{\mbox{\rm\tiny w}}

\def\obtheta{\overline{\boldsymbol \theta}}
\def\stest{\mbox{\tiny\rm test}}

\def\ERM{{\rm ERM}}
\def\Sp{{\mathbb S}}
\def\oalpha{\overline{\alpha}}
\def\hE{\widehat{{\mathbb E}}}
\def\hh{\hat{h}}
\def\tih{\tilde{h}}

\def\oby{\overline{\boldsymbol y}}
\def\br{{\boldsymbol r}}
\def\bfe{{\boldsymbol e}}
\def\bff{{\boldsymbol f}}

\def\oq{\overline{q}}

\def\ratio{\zeta}

\usepackage{color}

\def\cuB{\mathscrsfs{B}}
\def\cuV{\mathscrsfs{V}}

\def\SNR{{\sf SNR}}
\def\rank{{\rm rank}}

\def\hrho{\widehat{\rho}}

\def\eval{s}
\def\Eval{{\boldsymbol S}}

\def\uF{\underline{F}}
\def\ulambda{\underline{\lambda}}
\def\Risk{{L}}
\def\hRisk{\widehat{L}}

%% file: intro.tex
\section{Introduction}\label{sec:intro}

The past decade has witnessed dramatic advances in machine learning
that have led to major breakthroughs in computer vision, speech
recognition, and robotics.  These achievements are based on a
powerful and diverse toolbox of techniques and algorithms that now
bears the name `deep learning'; see, for example, \cite{gbc-dl-16}.
Deep learning has evolved from the decades-old methodology of neural
networks: circuits of parametrized nonlinear functions, trained by
gradient-based methods. Practitioners have made major architectural
and algorithmic innovations, and have exploited technological
advances, such as increased computing power, distributed computing
architectures, and the availability of large amounts of digitized
data. The 2018 Turing Award celebrated these advances, a reflection
of their enormous impact \cite{lbh-dl-15}.

Broadly interpreted, deep learning can be viewed as a family of
highly nonlinear statistical models that are able to encode highly
nontrivial representations of data. A prototypical example is a {\em
feed-forward neural network with $L$ layers}, which is a parametrized
family of functions $\bx\mapsto f(\bx;\btheta)$ defined on $\Re^d$ by
  \begin{equation}\label{eqn:FFNN}
    f(\bx;\btheta) := \sigma_L(\bW_L\sigma_{L-1}(\bW_{L-1}
      \cdots \sigma_1(\bW_1\bx)\cdots)),
  \end{equation}
where the parameters are $\btheta=(\bW_1,\ldots,\bW_L)$
with $\bW_l\in\Re^{d_l\times d_{l-1}}$ and $d_0=d$, and
$\sigma_l: \Re^{d_l}\to\Re^{d_l}$ are fixed nonlinearities,
called {\em activation functions}.  Given a training sample
$(\bx_1,y_1),\ldots,(\bx_n,y_n)\in\Re^{d}\times\Re^{d_L}$, the
parameters $\btheta$ are typically chosen by a gradient method to
minimize the {\em empirical risk},
  \[
    \widehat L(\btheta) := \frac{1}{n}\sum_{i=1}^n \ell(f(\bx_i;\btheta),y_i),
  \]
where $\ell$ is a suitable loss function. The aim is to ensure
that this model generalizes well, in the sense that $f(\bx;\btheta)$
is an accurate prediction of $y$ on a subsequent $(\bx,y)$ pair.
It is important to emphasize that deep learning is a data-driven
approach: these are rich but generic models, and the architecture,
parametrization and nonlinearities are typically chosen without
reference to a specific model for the process generating the data.

While deep learning has been hugely successful in the hands of
practitioners, there are significant gaps in our understanding
of what makes these methods successful. Indeed, deep learning
reveals some major surprises from a theoretical perspective:
deep learning methods can find near-optimal solutions to highly
non-convex empirical risk minimization problems, solutions that
give a near-perfect fit to noisy training data, but despite making
no explicit effort to control model complexity, these methods lead
to excellent prediction performance in practice.

To put these properties in perspective, it is helpful to recall
the three competing goals that statistical prediction methods must
balance: they require expressivity, to allow the richness of real
data to be effectively modelled; they must control statistical
complexity, to make the best use of limited training data; and
they must be computationally efficient.  The classical approach to
managing this trade-off involves a rich, high-dimensional model,
combined with some kind of regularization, which encourages
simple models but allows more complexity if that is warranted
by the data. In particular, complexity is controlled so that
performance on the training data, that is, the empirical risk,
is representative of performance on independent test data,
specifically so that the function class is simple enough that
sample averages $\widehat L(\btheta)$ converge to expectations
$L(\btheta):=\Expect\ell(f(\bx;\btheta),y)$ uniformly across the
function class. And prediction methods are typically formulated
as convex optimization problems---for example with a convex loss
$\ell$ and parameters $\btheta$ that enter linearly---which can be
solved efficiently.

The deep learning revolution built on two surprising empirical
discoveries that are suggestive of radically different ways of
managing these trade-offs.  First, deep learning exploits rich
and expressive models, with many parameters, and the problem
of optimizing the fit to the training data appears to simplify
dramatically when the function class is rich enough, that is, when
it is sufficiently overparametrized.  In this regime, simple,
local optimization approaches, variants of stochastic gradient methods,
are extraordinarily successful at finding near-optimal fits to
training data, even though the nonlinear parametrization---see equation \eqref{eqn:FFNN}---implies that the optimization problems
that these simple methods solve are notoriously non-convex. {\it
A posteriori,\/} the idea that overparametrization could lead to
tractability might seem natural, but it would have seemed completely
foolish from the point of view of classical learning theory: the
resulting models are outside the realm of uniform convergence,
and therefore should not be expected to generalize well.

The second surprising empirical discovery was that these models are
indeed outside the realm of uniform convergence. They are enormously
complex, with many parameters, they are trained with no explicit
regularization to control their statistical complexity, and they
typically exhibit a near-perfect fit to noisy training data, that
is, empirical risk close to zero. Nonetheless this overfitting
is benign, in that they produce excellent prediction performance
in a number of settings. Benign overfitting appears to contradict
accepted statistical wisdom, which insists on a trade-off between
the complexity of a model and its fit to the data. Indeed, the rule
of thumb that models fitting noisy data too well will not generalize
is found in most classical texts on statistics and machine learning
\cite{friedman2001elements,wasserman2013all}.  This viewpoint has
become so prevalent that the word `overfitting' is often taken
to mean both fitting data better than should be expected and also
giving poor predictive accuracy as a consequence. In this paper,
we use the literal meaning of the word `overfitting'; deep learning
practice has demonstrated that poor predictive accuracy is not an
inevitable consequence.

This paper reviews some initial steps towards understanding these two
surprising aspects of the success of deep learning. We have two working
hypotheses:
\begin{description}
\item[Tractability via overparametrization.] Classically,
tractable statistical learning is achieved by restricting
to linearly parametrized classes of functions and convex
objectives. A fundamentally new principle appears to be at work in
deep learning. Although the objective is highly non-convex, we
conjecture that the hardness of the optimization problem depends
on the relationship between the dimension of the parameter space
(the number of optimization variables) and the sample size (which,
when we aim for a near-perfect fit to training data, we can think
of as the number of constraints), that is, tractability is
achieved if and only if we choose a model that is sufficiently
under-constrained or, equivalently, overparametrized.

\item[Generalization via implicit regularization.]  Even if
overparametrized models simplify the optimization task, classically
we would have believed that good generalization properties would
be restricted to either an underparametrized regime or a suitably
regularized regime. Statistical wisdom suggests that a method
that takes advantage of too many degrees of freedom by perfectly
interpolating noisy training data will be poor at predicting new
outcomes. In deep learning, training algorithms appear to induce
a bias that breaks the equivalence among all the models that
interpolate the observed data.  Because these models interpolate
noisy data, the classical statistical perspective would suggest that
this bias cannot provide sufficient regularization to give good
generalization, but in practice it does. We conjecture that deep
learning models can be decomposed into a low-complexity component
for which classical uniform convergence occurs and a high-complexity
component that enables a perfect fit to training data,
and if the model is suitably overparameterized, this perfect
fit does not have a significant impact on prediction accuracy.
\end{description}

As we shall see, both of these hypotheses are supported by results in
specific scenarios, but there are many intriguing open questions in
extending these results to realistic deep learning settings.

It is worth noting that none of the results that we review here make
a case for any optimization or generalization benefits of increasing
depth in deep learning.  Although it is not the focus here, another
important aspect of deep learning concerns how deep neural networks can
effectively and parsimoniously express natural functions that are
well matched to the data that arise in practice. It seems likely
that depth is crucial for these issues of expressivity.

\subsection{Overview}

Section~\ref{sec:slt} starts by reviewing some results from classical
statistical learning theory that are relevant to the problem of
prediction with deep neural networks. It describes an explicit
probabilistic formulation of prediction problems. Consistent with
the data-driven perspective of deep learning, this formulation assumes little
more than that the $(\bx,y)$ pairs are sampled independently from
a fixed probability distribution.  We explain the role played by
uniform bounds on deviations between risk and empirical risk,
  \[
    \sup_{f\in \cF}\left| L(f) - \widehat L(f)\right|,
  \]
in the analysis of the generalization question for functions chosen
from a class $\cF$.
We show how
a partition of a rich function class $\cF$ into a complexity hierarchy
allows regularization methods that balance the statistical complexity
and the empirical risk to enjoy the best bounds on generalization
implied by the uniform convergence results.  We consider consequences
of these results for general pattern classification problems, for
easier ``large margin'' classification problems and for regression
problems, and we give some specific examples of risk bounds for
feed-forward networks.  Finally, we consider the implications of
these results for benign overfitting: If an algorithm chooses an
interpolating function to minimize some notion of complexity, what
do the uniform convergence results imply about its performance? We
see that there are very specific barriers to analysis of this kind
in the overfitting regime; an analysis of benign overfitting must
make stronger assumptions about the process that generates the data.


In Section~\ref{sec:implicit}, we review results on the implicit
regularization that is imposed by the algorithmic approach
ubiquitous in deep learning: gradient methods. We see examples of function
classes and loss functions where gradient methods, suitably
initialized, return the empirical risk minimizers that minimize
certain parameter norms. While all of these examples involve
parameterizations of linear functions with convex losses, we shall see
in Section~\ref{sec:efficient} that this linear/convex viewpoint
can be important for nonconvex optimization problems that arise in
neural network settings.

Section~\ref{sec:benign} reviews analyses of benign overfitting.
We consider extreme cases of overfitting, where the prediction
rule gives a perfect interpolating fit to noisy data. In all the
cases that we review where this gives good predictive accuracy,
we can view the prediction rule as a linear combination of two
components: $\widehat f=\widehat f_0+\Delta$. The first, $\widehat f_0$, is a
simple component that is useful for prediction, and the second,
$\Delta$, is a spiky component that is useful for overfitting.
Classical statistical theory explains the good predictive accuracy
of the simple component. The other component is not useful for
prediction, but equally it is not harmful for prediction.  The first
example we consider is the classical Nadaraya-Watson kernel smoothing
method with somewhat strange, singular kernels, which lead to an
interpolating solution that, for a suitable choice of the kernel
bandwidth, enjoys minimax estimation rates. In this case, we can view
$\widehat f_0$ as the prediction of a standard kernel smoothing method
and $\Delta$ as a spiky component that is harmless for prediction
but allows interpolation.  The other examples we consider are for
high-dimensional linear regression. Here, `linear' means linearly
parameterized, which of course allows for the richness of highly
nonlinear features, for instance the infinite dimensional feature
vectors that arise in reproducing kernel Hilbert spaces (RKHSs).
Motivated by the results of Section~\ref{sec:implicit}, we study
the behavior of the minimum norm interpolating linear function.
We see that it can be decomposed into a prediction component and an
overfitting component, with the split determined by the eigenvalues
of the data covariance matrix. The prediction component corresponds
to a high-variance subspace and the overfitting component to the
orthogonal, low-variance subspace.  For sub-Gaussian features,
benign overfitting occurs if and only if the high-variance
subspace is low-dimensional (that is, the prediction component
is simple enough for the corresponding subspace of functions to
exhibit uniform convergence) and the low-variance subspace has
high effective dimension and suitably low energy. In that case,
we see a {\em self-induced regularization}: the projection of the
data on the low-variance subspace is well-conditioned, just as
it would be if a certain level of statistical regularization were
imposed, so that even though this subspace allows interpolation,
it does not significantly deteriorate the predictive accuracy.
(Notice that this self-induced regularization is a consequence of
the decay of eigenvalues of the covariance matrix, and should not be
confused with the implicit regularization, which is a consequence
of the gradient optimization method and leads to the minimum norm
interpolant.)  Using direct arguments that avoid the sub-Gaussian
assumption, we see similar behavior of the minimum norm interpolant
in certain infinite-dimensional RKHSs, including an example of an
RKHS with fixed input dimension where benign overfitting cannot
occur and examples of RKHSs where it does occur for suitably
increasing input dimension, again corresponding to decompositions
into a simple subspace---in this case, a subspace of polynomials,
with dimension low enough for uniform convergence---and a complex
high-dimensional orthogonal subspace that allows
benign overfitting.

In Section~\ref{sec:efficient}, we consider a specific regime where
overparametrization allows a non-convex empirical risk minimization
problem to be solved efficiently by gradient methods: a linear
regime, in which a parameterized function can be accurately
approximated by its linearization about an initial parameter
vector.  For a suitable parameterization and initialization, we
see that a gradient method remains in the linear regime, enjoys
linear convergence of the empirical risk, and leads to a solution
whose predictions are well approximated by the linearization at
the initialization. In the case of two-layer networks, suitably
large overparametrization and initialization suffice.  On the
other hand, the mean-field limit for wide two-layer networks,
a limit that corresponds to a smaller---and perhaps more
realistic---initialization, exhibits an essentially different
behavior, highlighting the need to extend our understanding beyond
linear models.

Section~\ref{sec:NTK} returns to benign overfitting, focusing
on the linear regime for two specific families of two-layer
networks: a random features model, with randomly initialized
first-layer parameters that remain constant throughout training,
and a neural tangent model, corresponding to the linearization about
a random initialization. Again, we see decompositions into a simple
subspace (of low-degree polynomials) that is useful for prediction
and a complex orthogonal subspace that allows interpolation without
significantly harming prediction accuracy.

Section~\ref{sec:future} outlines future directions. Specifically,
for the two working hypotheses of tractability via
overparametrization and generalization via implicit regularization,
this section summarizes the insights from the examples that we
have reviewed---mechanisms for implicit regularization, the role of
dimension, decompositions into prediction and overfitting components,
data-adaptive choices of these decompositions, and the tractability
benefits of overparameterization. It also speculates on how these
might extend to realistic deep learning settings.

%% file: slt.tex
\section{Generalization and uniform convergence}
\label{sec:slt}

This section reviews  uniform convergence results from
statistical learning theory and their implications for prediction
with rich families of functions, such as those computed by neural
networks.  In classical statistical analyses, it is common to
posit a specific probabilistic model for the process generating
the data and to estimate the parameters of that model; see,
for example,~\cite{bd-ms-07}.  In contrast, the approach in this
section is motivated by viewing neural networks as defining rich,
flexible families of functions that are useful for prediction in
a broad range of settings. We make only weak assumptions about
the process generating the data, for example, that it is sampled
independently from an unknown distribution, and we aim for the best
prediction accuracy.

\subsection{Preliminaries}

Consider a prediction problem in a probabilistic setting, where we
aim to use data to find a function $f$ mapping from an input space
$\calX$ (for example, a representation of images) to an output space
$\calY$ (for example, a finite set of labels for those images). We
measure the quality of the predictions that $f:\calX\to\calY$ makes
on an $(\bx,y)$ pair using the loss $\ell(f(\bx),y)$, which represents
the cost of predicting $f(\bx)$ when the actual outcome is $y$. For
example, if $f(\bx)$ and $y$ are real-valued, we might consider the
square loss, $\ell(f(\bx),y)=(f(\bx)-y)^2$. We assume that we have
access to a training sample of input-output pairs $(\bx_1,y_1),
\ldots,(\bx_n,y_n)\in\calX\times\calY$, chosen independently from
a probability distribution $\P$ on $\calX\times\calY$. These data are
used to choose $\widehat f:\calX\to\calY$, and we would like $\widehat f$ to
give good predictions of the relationship between subsequent $(\bx,y)$
pairs in the sense that the {\em risk of $\widehat f$}, denoted
  \[
    L(\widehat f):= \Expect\ell(\widehat f(\bx),y),
  \]
is small, where $(\bx,y)\sim \P$ and $\Expect$ denotes expectation (and
if $\widehat f$ is random, for instance because it is chosen based on
random training data, we use $L(\widehat f)$ to denote the conditional
expectation given $\widehat f$).  We are interested in ensuring that
the excess risk of $\widehat f$,
  \[
    L(\widehat f) - \inf_{f}L(f),
  \]
is close to zero, where the infimum is over all measurable functions.
 Notice that we assume only that $(\bx,y)$ pairs
are independent and identically distributed; in particular, we do not
assume any functional relationship between $\bx$ and $y$.

Suppose that we choose $\widehat f$ from a set of functions
$\cF\subseteq\calY^\calX$. For instance, $\cF$ might be the set of
functions computed by a deep network with a particular architecture
and with particular constraints on the parameters in the network.
A natural approach to using the sample to choose $\widehat f$ is to
minimize the {\em empirical risk} over the class $\cF$. Define
  \begin{equation}\label{eqn:erm}
    \prederm
    \in\argmin{f\in \cF} \widehat L(f),
  \end{equation}
where the empirical risk, 
$$\widehat L(f) :=
\frac{1}{n}\sum_{i=1}^n \ell(f(\bx_i),y_i),$$ is the expectation of
the loss under the empirical distribution defined by the sample.
Often, we consider classes of functions $\bx\mapsto f(\bx;\btheta)$
parameterized by $\btheta$, and we use $L(\btheta)$ and $\widehat L(\btheta)$
to denote $L(f(\cdot;\btheta))$ and $\widehat L(f(\cdot;\btheta))$,
respectively.

We can split the excess risk of the empirical risk minimizer
$\prederm$ into two components, 
  \begin{align}
	  \label{eq:est_approx_decomp}
    L(\prederm) - \inf_{f}L(f)
      = \left(L(\prederm) - \inf_{f\in \cF}L(f)\right) +
      \left(\inf_{f\in \cF}L(f) - \inf_{f}L(f)\right),
  \end{align}
the second reflecting how well functions in the class $\cF$ can
approximate an optimal prediction rule and the first reflecting
the statistical cost of estimating such a prediction rule from the
finite sample.  For a more complex function class $\cF$, we should
expect the approximation error to decrease and the estimation error
to increase. We focus on the estimation error, and on controlling it
using uniform laws of large numbers.

\subsection{Uniform laws of large numbers}

Without any essential loss of generality,
suppose that a minimizer $f_{\cF}^*\in\arg\min_{f\in \cF} L(f)$ exists.
Then we can split the estimation error of an empirical risk minimizer
$\prederm$ defined in~\eqref{eqn:erm} into three components:
  \begin{align}
    \lefteqn{L(\prederm) - \inf_{f\in \cF}L(f)} & \notag\\
      & = L(\prederm)-L(f_{\cF}^*) \notag \\
      &= \left[ L(\prederm) - \widehat L(\prederm) \right]
        + \left[ \widehat L(\prederm) - \widehat L(f_{\cF}^*) \right]
        + \left[ \widehat L(f_{\cF}^*) - L(f_{\cF}^*) \right].
        \label{eqn:decomp} 
  \end{align}
The second term cannot be positive since $\prederm$ minimizes empirical
risk.  The third term converges to zero by the law of large numbers
(and if the random variable $\ell(f_{\cF}^*(\bx),y)$ is sub-Gaussian, then with
probability exponentially close to $1$ this term is $O(n^{-1/2})$;
see, for example,~\cite[Chapter~2]{blm-ci-13} and~\cite{v-hdp-18} for the
definition of sub-Gaussian and for a review of concentration
inequalities of this kind).  The first
term is more interesting. Since $\prederm$ is chosen using the data,
$\widehat L(\prederm)$ is a biased estimate of $L(\prederm)$,
and so we cannot simply apply a law of large numbers. One approach
is to use the crude upper bound
  \begin{equation}\label{eqn:loosebound}
    L(\prederm) - \widehat L(\prederm)
      \le \sup_{f\in\cF}\left|L(f) - \widehat L(f)\right|,
  \end{equation}
and hence bound the estimation error in terms of this uniform bound.
The following theorem shows that such uniform bounds on deviations
between expectations and sample averages are intimately related to
a notion of complexity of the {\em loss class} $\ell_{\cF}=\{(\bx,y)\mapsto
\ell(f(\bx),y):f\in \cF\}$ known as the {\em Rademacher complexity}.
For a probability distribution $\P$ on a measurable space
$\calZ$, a sample $\bz_1,\ldots,\bz_n\sim \P$, and a function class
$\cG\subset\Re^{\calZ}$, define the Rademacher complexity of $\cG$ as
  \begin{align*}
    R_n(\cG) &:= \Expect\sup_{g\in\cG}
      \left|\frac{1}{n}\sum_{i=1}^n \epsilon_i g(\bz_i)\right|,
  \end{align*}
where $\epsilon_1,\ldots,\epsilon_n\in\{\pm 1\}$ are independent
and uniformly distributed.
\begin{theorem}\label{thm:Rad}
For any $\cG\subset[0,1]^{\calZ}$ and any probability distribution
$\P$
on $\calZ$,
  \[
    \frac{1}{2}R_n(\cG) - \sqrt\frac{\log 2}{2n}
    \le \Expect\sup_{g\in\cG}
      \left|\Expect g - \widehat\Expect g\right|
    \le 2R_n(\cG),
  \]
where $\widehat\Expect g=n^{-1}\sum_{i=1}^n g(\bz_i)$ and
$\bz_1,\ldots,\bz_n$ are chosen i.i.d. according to $\P$.
Furthermore, with probability at least $1-2\exp(-2\epsilon^2 n)$ over
$\bz_1,\ldots,\bz_n$,
  \[
    \Expect\sup_{g\in\cG}
      \left|\Expect g - \widehat\Expect g\right|
        - \epsilon \le
    \sup_{g\in\cG}
      \left|\Expect g - \widehat\Expect g\right|
    \le \Expect\sup_{g\in\cG}
      \left|\Expect g - \widehat\Expect g\right|
        + \epsilon.
  \]
Thus, $R_n(\cG)\to 0$ if and only if 
$\sup_{g\in\cG}\left|\Expect g - \widehat\Expect g\right| \cas 0$.
\end{theorem}
See~\cite{kp-rpbrfl-00,k-rpsrm-01,bbl-msee-02,%
bm-rgcrbsr-02} and \cite{k-lrcaoiirm-06}.
This theorem shows that for bounded losses, a uniform bound
  \[
    \sup_{f\in\cF} \left|L(f) - \widehat L(f)\right|
  \]
on the maximal deviations between risks and empirical risks of any
$f$ in $\cF$ is tightly concentrated around its expectation, which is
close to the Rademacher complexity $R_n(\ell_{\cF})$.  Thus, we can bound
the excess risk of $\prederm$ in terms of the sum of the approximation
error $\inf_{f\in\cF}L(f) - \inf_{f}L(f)$ and this
bound on the estimation error.

\subsection{Faster rates}

Although the approach~\eqref{eqn:loosebound} of bounding the
deviation between the risk and empirical risk of $\prederm$ by the
maximum for any $f\in\cF$ of this deviation appears to be very
coarse, there are many situations where it cannot be improved by
more than a constant factor without stronger assumptions (we will see
examples later in this section). However,
there are situations where it can be significantly improved. As
an illustration, provided $\cF$ contains functions $f$ for which
the variance of $\ell(f(\bx),y)$ is positive, it is easy to see that
$R_n(\ell_{\cF})=\Omega(n^{-1/2})$.
Thus, the best bound on the estimation error implied by
Theorem~\ref{thm:Rad} must go to zero no faster than $n^{-1/2}$, but
it is possible for the risk of the empirical minimizer to converge to
the optimal value $L(f_{\cF}^*)$ faster than this.  For example,
when $\cF$ is suitably simple, this occurs for a nonnegative bounded
loss, $\ell:\calY\times\calY\to[0,1]$, when there is a function
$f_{\cF}^*$ in $\cF$ that gives perfect predictions, in the sense that almost
surely $\ell(f_{\cF}^*(\bx),y)=0$. In that case, the following theorem
is an example that gives a faster rate in terms of the
{\em worst-case empirical Rademacher complexity},
      \[
        \bar R_n(\cF) = \sup_{\bx_1,\ldots,\bx_n\in\calX}
          \Expect\left[\left.\sup_{f\in\cF}
          \left|\frac{1}{n}\sum_{i=1}^n \epsilon_i
          f(\bx_i)\right|\right| \bx_1,\ldots,\bx_n\right].
      \]
Notice that, for any probability distribution on $\calX$,
$R_n(\cF)\le\bar R_n(\cF)$.
  \begin{theorem}\label{theorem:fastRad}
    There is a constant $c>0$ such that for a bounded function class
    $\cF\subset [-1,1]^\calX$, for $\ell(\widehat y,y)=(\widehat
    y-y)^2$, and for any
    distribution $\P$ on $\calX\times[-1,1]$, with probability at least
    $1-\delta$, a sample $(\bx_1,y_1),\ldots,(\bx_n,y_n)$ satisfies for
    all $f\in \cF$,
      \[
        L(f) \le (1+c)\widehat L(f) + c\left(\log n\right)^4
        \bar R_n^2(\cF) + \frac{c\log(1/\delta)}{n}.
      \]
  \end{theorem}

In particular, when $L(f_{\cF}^*)=0$, the empirical minimizer has
$\widehat L(\prederm)=0$, and so with high probability, $L(\prederm) = \tilde
O\left(\bar R_n^2(\cF)\right)$, which can be as small as $\tilde
O(1/n)$ for a suitably simple class $\cF$.

Typically, faster rates like these arise when the variance of the
excess loss is bounded in terms of its expectation, for instance
  \[
    \Expect\left[\ell(f(\bx),y)-\ell(f_{\cF}^*(\bx),y)\right]^2 \le
c\Expect\left[\ell(f(\bx),y)-\ell(f_{\cF}^*(\bx),y)\right].
  \]
For a bounded nonnegative loss with $L(f_{\cF}^*)=0$, this so-called
Bernstein property is immediate, and it has been exploited
 in that case to give
fast rates for prediction with binary-valued~\cite{vc71,vc-tpr-74}
and real-valued~\cite{h-dtgpmnn-92,p-urltep-95,bl-iudsam-99}
function classes.  Theorem~\ref{theorem:fastRad}, which
follows from~\cite[Theorem~1]{srebro2010optimistic} and the
AM-GM inequality\footnote{The exponent on the $\log$ factor in
Theorem~\ref{theorem:fastRad} is larger than the result in the
cited reference; any exponent larger than $3$ suffices. See
\cite[Equation~(1.4)]{rv-crpscb-06}.}, relies
on the smoothness of the quadratic loss to give a bound for that
case in terms of the worst-case empirical Rademacher complexity.
There has been a significant body of related work over
the last thirty years. First, for quadratic loss in this
well-specified setting, that is, when $f^*(\bx)=\Expect[y|\bx]$
belongs to the class $\cF$, faster rates have been obtained even without
$L(f^*)=0$~\cite{vdg-erf-90}.  Second, the Bernstein property can
occur without the minimizer of $L$ being in $\cF$; indeed, it arises
for convex $\cF$ with quadratic loss~\cite{lbw-ealnnbf-96} or more
generally strongly convex losses~\cite{m-iscugd-02}, and this has
been exploited to give fast rates based on several other notions
of complexity~\cite{bbm-lrc-05,k-lrcaoiirm-06,pmlr-v40-Liang15}.
Recent techniques~\cite{mendelson2020extending} eschew concentration
bounds and hence give weaker conditions for convergence of $L(\prederm)$
to $L(f_{\cF}^*)$, without the requirement that the random variables
$\ell(f(\bx),y)$ have light tails.  Finally, while we have defined $\cF$
as the class of functions used by the prediction method, if it is
viewed instead as the benchmark (that is, the aim is to predict
almost as well as the best function in $\cF$, but the prediction
method can choose a prediction rule $\widehat f$ that is not necessarily
in $\cF$), then similar fast rates are possible under even weaker
conditions, but the prediction method must be more complicated than
empirical risk minimization; see~\cite{rakhlin2017empirical}.

\subsection{Complexity regularization}

The results we have seen give bounds on the excess risk of $\prederm$ in
terms of a sum of approximation error and a bound on the estimation
error that depends on the complexity of the function class $\cF$.
Rather than choosing the complexity of the function class $\cF$ in
advance, we could instead split a rich class $\cF$ into a complexity
hierarchy and choose the appropriate complexity based on the data,
with the aim of managing this approximation-estimation tradeoff.
We might define subsets $\cF_r$ of a rich class $\cF$, indexed by a
complexity parameter $r$. We call each $\cF_r$ a {\em complexity
class}, and we say that it has complexity $r$.

There are many classical examples of this approach. For instance,
support vector machines (SVMs)~\cite{cv-svn-95} use a reproducing
kernel Hilbert space (RKHS) $\cH$, and the complexity class $\cF_r$
is the subset of functions in $\cH$ with RKHS norm no more than $r$.
As another example, Lasso~\cite{t-rssvl-96} uses the set $\cF$ of
linear functions on a high-dimensional space, with the complexity
classes $\cF_r$ defined by the $\ell_1$ norm of the parameter vector.
Both SVMs and Lasso manage the approximation-estimation trade-off
by balancing the complexity of the prediction rule and its fit to
the training data: they minimize a combination of empirical risk
and some increasing function of the complexity $r$.

The following theorem gives an illustration of the effectiveness
of this kind of complexity regularization. In the first part of
the theorem, the complexity penalty for a complexity class is a
uniform bound on deviations between expectations and sample averages
for that class.  We have seen that uniform deviation bounds of
this kind imply upper bounds on the excess risk of the empirical
risk minimizer in the class.  In the second part of the theorem, the
complexity penalty appears in the upper bounds on excess risk that
arise in settings where faster rates are possible. In both cases, the
theorem shows that when the bounds hold, choosing the best penalized
empirical risk minimizer in the complexity hierarchy leads to the best
of these upper bounds.
\begin{theorem}\label{thm:complexity-reg}
For each $\cF_r\subseteq \cF$, define an empirical risk minimizer
  \begin{align*}
    \prederm^r & \in\argmin{f\in\cF_r}\widehat L(f).
  \end{align*}
Among these, select the one with complexity $\widehat r$ that gives
an optimal balance between the empirical risk and a complexity
penalty $p_r$:
  \begin{align}\label{eqn:PERM}
    \widehat f&=\prederm^{\widehat r}, &
    \widehat r & \in \argmin{r} \left(\widehat L(\prederm^r) + p_r\right).
  \end{align}
  \begin{enumerate}
    \item In the event that the complexity penalties are uniform deviation
    bounds:
    \begin{equation}\label{eqn:unifdev}
      \text{for all $r$, }
      \sup_{f\in \cF_r}
        \left|L(f)-\widehat L(f)\right|
        \le p_r,
    \end{equation}
    then we have the oracle inequality
      \begin{align}\label{eqn:oracleslow}
         L(\widehat f) - \inf_{f} L(f)
        \le \inf_r \left( \inf_{f\in \cF_r}  L(f)
          - \inf_{f} L(f) + 2p_r\right).
      \end{align}
    \item Suppose that the complexity classes and penalties are
    ordered, that is,
      \[
        r\le s\text{ implies } \cF_r\subseteq \cF_s\text{ and } p_r\le p_s,
      \]
    and fix $f_r^* \in \arg\min_{f\in\cF_r} L(f)$. In the event
    that the complexity penalties satisfy the uniform
    relative deviation bounds
    \begin{align}\label{eqn:unifreldev}
      \text{for all $r$, }
      \sup_{f\in\cF_r}\left(
         L(f)-
         L(f^*_r) -
        2\left(\widehat L(f) -
        \widehat L(f^*_r)\right)\right)
        & \le 2p_r/7 \\
      \text{and }
      \sup_{f\in\cF_r}\left(
        \widehat L(f)-
        \widehat L(f^*_r) -
        2\left( L(f) -
         L(f^*_r)\right)\right)
        & \le 2p_r/7,\notag
    \end{align}
    then we have the oracle inequality
      \begin{align}\label{eqn:oraclefast}
         L(\widehat f) - \inf_{f} L(f)
        \le \inf_r \left( \inf_{f\in\cF_r}  L(f)
          - \inf_{f} L(f) + 3p_r\right).
      \end{align}
  \end{enumerate}
\end{theorem}
These are called {\em oracle inequalities}
because~\eqref{eqn:oracleslow} (respectively~\eqref{eqn:oraclefast})
gives the error bound that follows from the best of the uniform
bounds~\eqref{eqn:unifdev} (respectively~\eqref{eqn:unifreldev}), as
if we have access to an oracle who knows the complexity that gives the
best bound.  The proof of the first part is a straightforward
application of the same decomposition as~\eqref{eqn:decomp}; see, for
example,~\cite{bbl-msee-02}. The proof of the second part, which
allows significantly smaller penalties $p_r$ when faster rates are
possible, is also elementary; see~\cite{b-freeoims-07}.  In both
cases, the broad approach to managing the trade-off between
approximation error and estimation error is qualitatively the
same: having identified a complexity hierarchy $\{\cF_r\}$ with
corresponding excess risk bounds $p_r$, these results show the
effectiveness of choosing from the hierarchy a function $f$ that
balances the complexity penalty $p_r$ with the fit to the training
data $\widehat L(\prederm^r)$.

Later in this section, we will see examples of upper bounds on
estimation error for neural network classes $\cF_r$ indexed by a
complexity parameter $r$ that depends on properties of the network,
such as the size of the parameters. Thus, a prediction method that
trades off the fit to the training data with these measures of
complexity would satisfy an oracle inequality.

\subsection{Computational complexity of empirical risk minimization}
\label{sec:Intractability}

To this point, we have considered the statistical performance of the
empirical risk minimizer $\prederm$ without considering the computational
cost of solving this optimization problem.  The classical cases
where it can be solved efficiently involve linearly parameterized
function classes, convex losses, and convex complexity penalties, so
that penalized empirical risk minimization is a convex optimization
problem.  For instance, SVMs exploit a linear function class (an
RKHS, $\cH$), a convex loss, 
	\[
		\ell(f(\bx),y) := (1-yf(\bx))\vee 0 ~~\text{for}~~ f:\cX\to\Re ~~\text{and}~~ y\in\{\pm 1\}, 
	\] 
and a convex complexity
penalty, 
	\[
		\cF_r=\{f\in\cH:\|f\|_{\cH}\le r\},\, p_r=r/\sqrt n,
	\] 
and choosing $\widehat f$ according to~\eqref{eqn:PERM} corresponds
to solving a quadratic program. Similarly, Lasso involves linear
functions on $\Re^d$, quadratic loss, and a convex penalty,
	\[
		\cF_r=\{\bx\mapsto \langle\bx,\bbeta\rangle:\|\bbeta\|_1\le r\},\,
p_r=r\sqrt{\log(d)/n}.
\] 
Again, minimizing complexity-penalized empirical risk corresponds to solving a
quadratic program.

On the other hand, the optimization problems that arise in a
classification setting, where functions map to a discrete set,
have a combinatorial flavor, and are often computationally hard in
the worst case. For instance, empirical risk minimization over the
set of linear classifiers
  \[
    \cF = \left\{\bx\mapsto\sign\left(\langle \bw, \bx\rangle\right)
        :\bw\in\Re^d\right\}
  \]
is NP-hard~\cite{jp-dhp-78,gj-ci-79}.  In contrast, if there is a
function in this class that classifies all of the training data
correctly, finding an empirical risk minimizer is equivalent
to solving a linear program, which can be solved efficiently.
Another approach to simplifying the algorithmic challenge of
empirical risk minimization is to replace the discrete loss for
this family of thresholded linear functions with a surrogate convex
loss for the family of linear functions. This is the approach used
in SVMs: replacing a nonconvex loss with a convex loss allows for
computational efficiency, even when there is no thresholded linear
function that classifies all of the training data correctly.

However, the corresponding optimization problems for
neural networks appear to be more difficult. Even when
$\widehat L(\prederm)=0$, various natural empirical risk
minimization problems over families of neural networks are
NP-hard~\cite{j-nndcl-90,br-t3nnn-92,dss-ctnncaf-95}, and this is
still true even for convex losses~\cite{v-itnnsse-98,bbd-hrnnap-02}.

In the remainder of this section, we focus on the statistical
complexity of prediction problems with neural network function classes
(we shall return to computational complexity considerations in
Section~\ref{sec:efficient}). We review
estimation error bounds involving these classes, focusing particularly
on the Rademacher complexity. The Rademacher
complexity of a loss class $\ell_{\cF}$ can vary dramatically with
the loss $\ell$.  For this reason, we consider separately discrete
losses, such as those used for classification, convex upper bounds
on these losses, like the SVM loss and other large margin losses
used for classification, and Lipschitz losses used for regression.

\subsection{Classification}

We first consider loss classes for the problem of classification. For
simplicity, consider a two-class classification problem, where
$\calY=\{\pm 1\}$, and define the $\pm 1$ loss, $\ell_{\pm
1}(\widehat y,y) = - y\widehat y$. Then for $\cF\subset\{\pm 1\}^{\calX}$,
$R_n(\ell_{\cF})= R_n(\cF)$, since the distribution of $\epsilon_i\ell_{\pm
1}(f(\bx_i),y_i)=-\epsilon_iy_if(\bx_i)$ is the same as that of
$\epsilon_if(\bx_i)$. The following theorem shows that the Rademacher
complexity depends on a combinatorial dimension of $\cF$, known as
the VC-dimension~\cite{vc71}.

\begin{theorem}\label{thm:vcdim}
For $\cF\subseteq[-1,1]^{\calX}$ and for any distribution on $\calX$,
  \[
    R_n(\cF) \le \sqrt{\frac{2\log(2 \Pi_{\cF}(n))}{n}},
  \]
where
  \[
    \Pi_{\cF}(n) = \max\left\{\left|\left\{(f(\bx_1),\ldots,f(\bx_n)):f\in \cF
    \right\}\right|:\bx_1,\ldots,\bx_n\in\calX\right\}.
  \]
If $\cF\subseteq\{\pm 1\}^{\calX}$ and
$n\ge d=d_{VC}(\cF)$, then
$$ \Pi_{\cF}(n) \le \left(en/d\right)^d,$$
where $d_{VC}(\cF) := \max\left\{d: \Pi_{\cF}(d)=2^d\right\}$.
In that case, for any distribution on $\calX$,
  \[
    R_n(\cF)=O\left( \sqrt{\frac{d\log(n/d)}{n}}\right),
  \]
and conversely, for some probability distribution,
$R_n(\cF) =\Omega\left(\sqrt{d/n}\right)$.
\end{theorem}
These bounds imply that, for the worst case probability
distribution, the uniform deviations between sample averages and
expectations grow like $\tilde\Theta(\sqrt{d_{VC}(\cF)/n})$,
a result of~\cite{vc71}. The log factor in the upper bound can
be removed; see~\cite{t-sbgep-94}.  Classification problems are
an example where the crude upper bound~\eqref{eqn:loosebound}
cannot be improved without stronger assumptions: the {\em
minimax excess risk} is essentially the same as these uniform
deviations. In particular, these results show that empirical
risk minimization leads, for any probability distribution,
to excess risk that is $O(\sqrt{d_{VC}(\cF)/n})$, but conversely,
for every method that predicts a $\widehat f\in \cF$, there is
a probability distribution for which the excess risk is
$\Omega(\sqrt{d_{VC}(\cF)/n})$~\cite{vc-tpr-74}.  When there
is a prediction rule in $\cF$ that predicts perfectly, that is
$L(f_{\cF}^*)=0$, the upper and lower bounds can be improved
to $\tilde\Theta(d_{VC}(\cF)/n)$~\cite{behw-lvd-89,ehkv-glbnenl-89}.

These results show that $d_{VC}(\cF)$ is critical for uniform
convergence of sample averages to probabilities, and more generally
for the statistical complexity of classification with a function
class $\cF$. The following theorem summarizes the known bounds on the
VC-dimension of neural networks with various piecewise-polynomial
nonlinearities. Recall that a
feed-forward neural network with $L$ layers is defined
by a sequence of layer widths $d_1,\ldots,d_L$ and functions
$\sigma_l: \Re^{d_l}\to\Re^{d_l}$ for $l=1,\ldots,L$. It is a
family of $\Re^{d_L}$-valued functions on $\Re^d$ parameterized
by $\btheta=(\bW_1,\ldots,\bW_L)$; see~\eqref{eqn:FFNN}.
We often consider scalar nonlinearities
$\sigma:\Re\to\Re$ applied componentwise, that is,
$\sigma_l(v)_i:=\sigma(v_i)$.  For instance, $\sigma$ might be the
scalar nonlinearity used in the {\em ReLU (rectified linear unit)},
$\sigma(\alpha)=\alpha\vee 0$.  We say that this family has $p$
parameters if there is a total of $p$ entries in the matrices
$\bW_1,\ldots,\bW_L$. We say that $\sigma$ is {\em piecewise polynomial}
if it can be written as a sum of a constant number of polynomials,
  \[
    \sigma(x) = \sum_{i=1}^k \indicator{x\in I_i} p_i(x),
  \]
where the intervals $I_1,\ldots,I_k$ form a partition of $\Re$ and the
$p_i$ are polynomials.

\begin{theorem}\label{thm:VCdimNNs}
Consider feed-forward neural networks $\cF_{L,\sigma}$ with $L$ layers,
scalar output (that is, $d_L=1$), output nonlinearity
$\sigma_L(\alpha)=\sign(\alpha)$, and scalar nonlinearity $\sigma$ at
every other layer. Define
  \[
    d_{L,\sigma,p}=\max\left\{d_{VC}(\cF_{L,\sigma}) :
      \text{$\cF_{L,\sigma}$ has $p$ parameters}
    \right\}.
  \]
\begin{enumerate}
\item\label{part:NNVC:0} For $\sigma$ piecewise constant,
  $d_{L,\sigma,p}= \tilde \Theta\left(p\right)$.
\item\label{part:NNVC:1} For $\sigma$ piecewise linear,
  $d_{L,\sigma,p}= \tilde \Theta\left(pL\right)$.
\item\label{part:NNVC:2} For $\sigma$ piecewise polynomial,
  $d_{L,\sigma,p}= \tilde O\left(pL^2\right)$.
\end{enumerate}
\end{theorem}
Part~\ref{part:NNVC:0} is from~\cite{bh-wsngvg-89}.  The upper bound
in part~\ref{part:NNVC:1} is from~\cite{bhlm-ntvcdpdbfplnn-19}.
The lower bound in part~\ref{part:NNVC:1} and
the bound in part~\ref{part:NNVC:2}
are from~\cite{bmm-alvdbppn-98}.  There are also upper bounds
for the smooth sigmoid $\sigma(\alpha)=1/(1+\exp(-\alpha))$ that are
quadratic in $p$; see~\cite{km-pbvdsgpnn-97}. See Chapter~8
of~\cite{ab-nnltf-99} for a review.

The theorem shows that the VC-dimension of these neural networks
grows at least linearly with the number of parameters in the network,
and hence to achieve small excess risk or uniform convergence of
sample averages to probabilities for discrete losses, the sample size
must be large compared to the number of parameters in these networks.

There is an important caveat to this analysis: it captures
arbitrarily fine-grained properties of real-valued functions,
because the operation of thresholding these functions is very
sensitive to perturbations, as the following example shows.
\begin{example}
  For $\alpha>0$, define the nonlinearity $\tilde r(x):=(x+\alpha\sin
  x)\vee 0$ and the following one-parameter class of functions
  computed by two-layer networks with these nonlinearities:
    \[
      \cF_{\tilde r} := \left\{ x\mapsto \sign(\pi + \tilde r(wx) -
      \tilde r(wx+\pi)):
        w\in\Re\right\}.
    \]
  Then $d_{VC}(\cF_{\tilde r})=\infty$.
\end{example}
Indeed, provided $wx\ge\alpha$, $\tilde r(wx)=wx+\alpha\sin(wx)$, hence $\pi +
\tilde r(wx) - \tilde r(wx+\pi)=2\alpha\sin(wx)$. This shows that the set
of functions in $\cF_{\tilde r}$ restricted to $\mathbb{N}$ contains
    \begin{align*}
      \left\{ x\mapsto\sign(\sin(wx)):w\ge \alpha\right\}
      & =
      \left\{ x\mapsto\sign(\sin(wx)) :w\ge 0\right\},
    \end{align*}
and the VC-dimension of the latter class of functions on $\mathbb{N}$
is infinite; see, for example,~\cite[Lemma~7.2]{ab-nnltf-99}.  Thus,
with an arbitrarily small perturbation of the ReLU nonlinearity, the
VC-dimension of this class changes from a small constant to infinity.
See also~\cite[Theorem~7.1]{ab-nnltf-99}, which gives a similar
result for a slightly perturbed version of a sigmoid nonlinearity.

As we have seen, the requirement that the sample size grows with
the number of parameters is at odds with empirical experience:
deep networks with far more parameters than the number of training
examples routinely give good predictive accuracy.  It is plausible
that the algorithms used to optimize these networks are not
exploiting their full expressive power. In particular, the analysis
based on combinatorial dimensions captures arbitrarily fine-grained
properties of the family of real-valued functions computed by a deep
network, whereas algorithms that minimize a convex loss might not
be significantly affected by such fine-grained properties.  Thus,
we might expect that replacing the discrete loss $\ell_{\pm 1}$
with a convex surrogate, in addition to computational convenience,
could lead to reduced statistical complexity.
 The empirical success of gradient methods
with convex losses for overparameterized thresholded
real-valued classifiers was observed both in neural networks
\cite{mp-rhpld-90}, \cite{lgt-lnnt-97}, \cite{clg-onn-01} and in
related classification methods \cite{dc-bdt-95}, \cite{q-bbc-96},
\cite{b-ac-98}. It was noticed that classification performance
can improve as the number of parameters is increased even after
all training examples are classified correctly \cite{q-bbc-96},
\cite{b-ac-98}.\footnote{Both phenomena were observed more
recently in neural networks; see \cite{zhang2016understanding}
and \cite{neyshabur2017geometry}.} These observations motivated
{\em large margin analyses} \cite{b-scpcnn-98},
\cite{schapire1998boosting}, which reduce classification problems
to regression problems.

\subsection{Large margin classification}

Although the aim of a classification problem is to minimize the
expectation of a discrete loss, if we consider classifiers such
as neural networks that consist of thresholded real-valued functions
obtained by minimizing a surrogate loss---typically a convex
function of the real-valued prediction---then it turns out that we
can obtain bounds on estimation error by considering approximations
of the class of real-valued functions.  This is important because
the statistical complexity of that function class can be considerably
smaller than that of the class of thresholded functions. In effect,
for a well-behaved surrogate loss, fine-grained properties of
the real-valued functions are not important. If the surrogate
loss $\ell$ satisfies a Lipschitz property, we can relate the
Rademacher complexity of the loss class $\ell_{\cF}$ to that of the
function class $\cF$ using the {\em Ledoux-Talagrand contraction
inequality}~\cite[Theorem~4.12]{lt-pbsip-91}.

\begin{theorem}\label{thm:contraction}
Suppose that, for all $y$, $\widehat y\mapsto\ell(\widehat y,y)$
is $c$-Lipschitz and satisfies $\ell(0,y)=0$. Then
$R_n(\ell_{\cF})\le 2cR_n(\cF)$.
\end{theorem}

Notice that the assumption that $\ell(0,y)=0$ is
essentially without loss of generality: adding a fixed
function to $\ell_{\cF}$ by replacing $\ell(\widehat y,y)$
with $\ell(\widehat y,y)-\ell(0,y)$ shifts the Rademacher
complexity by $O\left(1/\sqrt{n}\right)$.

For classification with $y\in\{-1,1\}$, the hinge loss $\ell(\widehat
y,y) = (1-y\widehat y)\vee 0$ used by SVMs and the logistic loss
$\ell(\widehat y,y) = \log\left(1+\exp(-y\widehat y)\right)$ are examples
of convex, $1$-Lipschitz surrogate losses.  The quadratic loss
$\ell(\widehat y,y)=(\widehat y-y)^2$ and the exponential loss
$\ell(\widehat
y,y) := \exp(-y\widehat y)$ used by AdaBoost~\cite{fs-dgolab-97}
(see Section~\ref{sec:implicit}) are also convex, and they are
Lipschitz when functions in $\cF$ have bounded range.

We can write all of these surrogate losses as $\ell_\phi(\widehat y,y)
:= \phi(\widehat y y)$ for some function $\phi:\Re\to[0,\infty)$. The
following theorem relates the excess risk to the excess surrogate
risk. It is simpler to state when $\phi$ is convex and when,
rather than $\ell_{\pm 1}$, we consider a shifted, scaled version,
defined as $\ell_{01}(\widehat y,y): = \indicator{\widehat y\not=y}$.
We use
$L_{01}(f)$ and $L_{\phi}(f)$ to denote $\Expect\ell_{01}(f(\bx),y)$
and $\Expect\ell_\phi(f(\bx),y)$ respectively.

\begin{theorem}\label{thm:psicomp}
For a convex function $\phi:\Re\to[0,\infty)$, define $\ell_\phi(\widehat
y,y) := \phi(\widehat y y)$ and
$C_\theta(\alpha):=(1+\theta)\phi(\alpha)/2+(1-\theta)\phi(-\alpha)/2$,
and define $\psi_\phi:[0,1]\to[0,\infty)$ as
  $  \psi_\phi(\theta)
  :=
      \inf\left\{ C_\theta(\alpha)
        : \alpha\le 0\right\}
      -\inf\left\{ C_\theta(\alpha)
        : \alpha\in\Re\right\}$.
Then we have the following.
\begin{enumerate}
\item For any measurable $\widehat f:\calX\to\Re$ and any probability
distribution $\P$ on $\calX\times\calY$,
  \[
    \psi_\phi\left(L_{01}(\widehat f) - \inf_f L_{01}(f)\right) \le
    L_{\phi}(\widehat f) - \inf_f L_{\phi}(f),
  \]
where the infima are over measurable functions $f$.
\item For $|\calX|\ge 2$, this inequality cannot hold if
$\psi_\phi$ is replaced by any larger function:
  \begin{align*}
    &\sup_\theta \inf\left\{
    L_{\phi}(\widehat f) - \inf_f L_{\phi}(f) - \psi_\phi(\theta):
    \right.  \\*
    &\qquad\qquad\qquad\qquad\left.
    \text{$\P$, $\widehat f$ satisfy }
    L_{01}(\widehat f) - \inf_f L_{01}(f) = \theta\right\} = 0.
  \end{align*}
\item $\psi_\phi(\theta_i)\to 0$ implies
$\theta_i\to 0$ if and only if both $\phi$ is differentiable at zero and
$\phi'(0)<0$.
\end{enumerate}
\end{theorem}
For example, for the hinge loss $\phi(\alpha)= (1-\alpha)\vee 0$,
the relationship between excess risk and excess $\phi$-risk is
given by $\psi_\phi(\theta)=|\theta|$, for the quadratic loss
$\phi(\alpha)=(1-\alpha)^2$, $\psi_\phi(\theta)=\theta^2$,
and for the exponential loss $\phi(\alpha) =
\exp(-\alpha)$, $\psi_\phi(\theta)=1-\sqrt{1-\theta^2}$.
Theorem~\ref{thm:psicomp} is from~\cite{bjm-ccrb-05}; see
also \cite{l-nmblfc-04,lv-brcrbm-04} and \cite{z-sbccbcrm-04}.

Using~\eqref{eqn:decomp},~\eqref{eqn:loosebound},
and Theorems~\ref{thm:Rad} and~\ref{thm:contraction} to bound
$\Expect\ell_{\phi,\widehat f}-\inf_f\Expect\ell_{\phi,f}$ in terms of
$R_n(\cF)$ and combining with Theorem~\ref{thm:psicomp} shows that, if
$\phi$ is $1$-Lipschitz then with high probability,
\begin{equation}\label{eqn:phiriskbound}
  \psi_\phi\left(L_{01}(\widehat f)-\inf_fL_{01}(f)\right)
  \le 4R_n(\cF) + O\left(\frac{1}{\sqrt n}\right) +
  \inf_{f\in\cF}L_{\phi}(f)-\inf_fL_{\phi}(f).
\end{equation}
Notice that in addition to the Rademacher complexity of the
real-valued class $\cF$, this bound includes an approximation error
term defined in terms of the surrogate loss; the binary-valued
prediction problem has been reduced to a real-valued problem.

Alternatively, we could consider more naive bounds: If
a loss satisfies the pointwise inequality $\ell_{01}(\widehat
y,y)\le\ell_{\phi}(\widehat y,y)$, then we have an upper bound on risk in
terms of surrogate risk: $L_{01}(\widehat f)\le L_{\phi}(\widehat f)$.  In fact,
Theorem~\ref{thm:psicomp} implies that pointwise inequalities
like this are inevitable for any reasonable convex loss. Define a
surrogate loss $\phi$ as {\em classification-calibrated} if any $f$
that minimizes the surrogate risk $L_\phi(f)$ will also minimize the
classification risk $L_{01}(f)$.  Then part~3 of the theorem shows
that if a convex surrogate loss $\phi$ is classification-calibrated
then it satisfies
  \[
    \text{for all $\widehat y, y$, } \frac{\ell_\phi(\widehat y,y)}{\phi(0)}
    = \frac{\phi(\widehat y y)}{\phi(0)}
    \ge 1[\widehat y y \le 0] = \ell_{01}(\widehat y,y).
  \]
Thus, every classification-calibrated convex surrogate loss, suitably
scaled so that $\phi(0)=1$, is an upper bound on the discrete loss
$\ell_{01}$, and hence immediately gives an upper bound on risk in
terms of surrogate risk: $L_{01}(\widehat f)\le L_{\phi}(\widehat f)$.  Combining
this with Theorems~\ref{thm:Rad} and~\ref{thm:contraction} shows
that, if $\phi$ is also $1$-Lipschitz then with high probability,
\begin{equation}\label{eqn:phiupper}
  L_{01}(\widehat f)\le \widehat L_{\phi}(\widehat f) + 4R_n(\cF)
    + O\left(\frac{1}{\sqrt n}\right).
\end{equation}

\subsection{Real prediction}

For a real-valued function class $\cF$, there is an analog of
Theorem~\ref{thm:vcdim} with the VC-dimension of $\cF$ replaced by the
{\em pseudodimension} of $\cF$, which is the VC-dimension of
$\left\{(\bx,y)\mapsto \indicator{f(\bx)\ge y}: f\in \cF\right\}$;
see~\cite{p-epta-90}.  Theorem~\ref{thm:VCdimNNs} is true with the
output nonlinearity $\sigma_L$ of $\cF_{L,\sigma}$ replaced by any
Lipschitz nonlinearity and with $d_{VC}$ replaced by the
pseudodimension. However, using this result to obtain bounds on the
excess risk of an empirical risk minimizer would again require the
sample size to be large compared to the number of parameters.

Instead, we can bound $R_n(\ell_{\cF})$ more directly in many
cases. With a bound
on $R_n(\cF)$ for a class $\cF$ of real-valued functions computed by
neural networks, we can then apply Theorem~\ref{thm:contraction} to
relate $R_n(\ell_{\cF})$ to $R_n(\cF)$, provided the loss is a Lipschitz
function of its first argument.  This is the case, for example,
for absolute loss $\ell(\widehat y,y)=|\widehat y-y|$, or for quadratic
loss $\ell(\widehat y,y)=(\widehat y-y)^2$ when $\cY$ and the range of
functions  in $\cF$ are bounded.

The following result gives a bound on Rademacher complexity for neural
networks that use a bounded, Lipschitz nonlinearity, such as the
sigmoid function
  \[
    \sigma(x) = \frac{1-\exp(-x)}{1+\exp(-x)}.
  \]

\begin{theorem}\label{thm:sigmoidRn}
For two-layer neural networks defined on $\calX=[-1,1]^d$,
  \[
    \cF_B = \left\{\bx\mapsto \sum_{i=1}^k
        b_i\sigma\left(\langle \bw_i,\bx\rangle\right)
        : \|\bb\|_1\le 1, \,
        \|\bw_i\|_1\le B,\, k\ge 1 \right\},
  \]
where the nonlinearity $\sigma:\Re\to[-1,1]$ is $1$-Lipschitz and has
$\sigma(0)=0$,
  \[
    R_n(\cF_B) \le B \sqrt{\frac{2\log 2d}{n}}.
  \]
\end{theorem}
Thus, for example, applying~\eqref{eqn:phiriskbound} in this
case with a Lipschitz convex loss $\ell_\phi$ and the corresponding
$\psi_\phi$ defined by Theorem~\ref{thm:psicomp}, shows that with high
probability the minimizer $\prederm$ in $\cF_B$ of
$\widehat\Expect\ell_{f}$ satisfies
  \[
    \psi_\phi\left(L_{01}(\prederm)-\inf_fL_{01}(f)\right)
      \le O\left(B \sqrt{\frac{\log d}{n}}\right) +
      \inf_{f\in  \cF_B}L_{\phi}(f)-\inf_fL_{\phi}(f).
  \]
If, in addition, $\ell_\phi$ is scaled so that it is an upper bound on
$\ell_{01}$, applying~\eqref{eqn:phiupper} shows that with high
probability every $f\in \cF_B$ satisfies
  \[
    L_{01}(f) \le  \widehat L_{\phi}(f)
            + O\left(B\sqrt{\frac{\log d}{n}}\right).
  \]
Theorem~\ref{thm:sigmoidRn} is from~\cite{bm-rgcrbsr-02}. The proof
uses the contraction inequality (Theorem~\ref{thm:contraction})
and elementary properties of Rademacher complexity.

The following theorem gives similar error bounds for networks with
Lipschitz nonlinearities that, like the ReLU nonlinearity, do not
necessarily have a bounded range. The definition of the function class
includes deviations of the parameter matrices $\bW_i$ from fixed
`centers' $\bM_i$.

\begin{theorem}\label{thm:spectralN}
Consider a feed-forward network with $L$ layers, fixed vector
nonlinearities $\sigma_i:\Re^{d_i}\to\Re^{d_i}$ and parameter
$\btheta=(\bW_1,\ldots,\bW_L)$ with
$\bW_i\in\Re^{d_i\times d_{i-1}}$, for $i=1,\ldots,L$,
which computes functions
  \[
    f(\bx;\btheta) =
      \sigma_L(\bW_L\sigma_{L-1}(\bW_{L-1} \cdots \sigma_1(\bW_1\bx)\cdots)),
  \]
where $d_0=d$ and $d_L=1$. Define $\bar d =d_0\vee\cdots\vee d_L$.
Fix matrices $\bM_i\in\Re^{d_i\times d_{i-1}}$, for $i=1,\ldots,L$, and
define the class of functions on the unit Euclidean ball in $\Re^d$,
  \[
    \cF_r = \left\{f(\cdot,\btheta): \prod_{i=1}^L \|\bW_i\| \left(\sum_{i=1}^L
    \frac{\|\bW_i^\top-\bM_i^\top\|_{2,1}^{2/3}}{\|\bW_i\|^{2/3}}
    \right)^{3/2}\le r \right\},
  \]
where $\|\bA\|$ denotes the spectral norm of the matrix $\bA$ and
$\|\bA\|_{2,1}$ denotes the sum of the $2$-norms of its columns.
If the $\sigma_i$ are all $1$-Lipschitz and the surrogate
loss $\ell_\phi$ is a $b$-Lipschitz upper bound on the classification
loss $\ell_{01}$, then with probability at least $1-\delta$, every
$f\in \cF_r$ has
  \[
    L_{01}(f) \le \widehat L_{\phi}(f)
      + \tilde O\left(\frac{rb\log\bar d
      +\sqrt{\log(1/\delta)}}{\sqrt{n}}\right).
  \]
\end{theorem}

Theorem~\ref{thm:spectralN} is from \cite{snmbnn-manng-17}.
The proof uses different techniques (covering numbers rather than
the Rademacher complexity) to address the key technical difficulty,
which is controlling the scale of vectors that appear throughout
the network.

When the nonlinearity has a {\em $1$-homogeneity property}, the
following result gives a simple direct bound on the Rademacher
complexity in terms of the Frobenius norms of the weight matrices
(although it is worse than Theorem~\ref{thm:spectralN}, even with
$\bM_i=0$, unless the ratios $\|\bW_i\|_F/\|\bW_i\|$ are close to $1$).
We say that $\sigma:\Re\to\Re$ is $1$-homogeneous
if $\sigma(\alpha x) = \alpha\sigma(x)$ for all $x\in\Re$ and
$\alpha\ge 0$.  Notice that the ReLU nonlinearity $\sigma(x)=x\vee 0$
has this property.

\begin{theorem}\label{thm:FrobeniusRn}
Let $\bar\sigma:\Re\to\Re$ be a fixed $1$-homogeneous
nonlinearity, and define the componentwise version
$\sigma_i:\Re^{d_i}\to\Re^{d_i}$ via
$\sigma_i(\bx)_j=\bar\sigma(\bx_j)$.
Consider a network with $L$ layers of these nonlinearities
and parameters $\btheta=(\bW_1,\ldots,\bW_L)$,
which computes functions
  \[
    f(\bx;\btheta) =
      \sigma_L(\bW_L\sigma_{L-1}(\bW_{L-1} \cdots \sigma_1(\bW_1\bx)\cdots)).
  \]
Define the class of functions on the unit Euclidean ball in $\Re^d$,
  \[
    \cF_B = \left\{f(\cdot;\btheta): \|\bW_i\|_F\le B \right\},
  \]
where $\|\bW_i\|_F$ denotes the Frobenius norm of $\bW_i$.
Then we have
  \[
    R_n(\cF_B) \lesssim \frac{\sqrt{L}B^L}{\sqrt n}.
  \]
\end{theorem}

This result is from \cite{pmlr-v75-golowich18a}, which also shows that
it is possible to remove the $\sqrt{L}$ factor at the cost of a worse
dependence on $n$. See also~\cite{pmlr-v40-Neyshabur15}.

\subsection{The mismatch between benign overfitting and uniform
convergence}
\label{sec:slt_mismatch}

It is instructive to consider the implications of the generalization
bounds we have reviewed in this section for the phenomenon of
benign overfitting, which has been observed in deep learning. For
concreteness, suppose that $\ell$ is the quadratic loss.  Consider a
neural network function $\widehat f\in \cF$ chosen so that $\widehat
L(\widehat
f)=0$. For an appropriate complexity hierarchy $\cF=\bigcup_r \cF_r$,
suppose that $\widehat f$ is chosen to minimize the complexity
$r(\widehat
f)$, defined as the smallest $r$ for which $\widehat f\in \cF_r$, subject
to the interpolation constraint $\widehat L(f)=0$. What do the bounds
based on uniform convergence imply about the excess risk $L(\widehat
f)-\inf_{f\in\cF}L(f)$ of this {\em minimum-complexity interpolant}?

Theorems~\ref{thm:sigmoidRn},~\ref{thm:spectralN},
and~\ref{thm:FrobeniusRn} imply upper bounds on risk in terms of
various notions of scale of network parameters. For these bounds
to be meaningful for a given probability distribution, there must
be an interpolating $\widehat f$ for which the complexity $r(\widehat f)$
grows suitably slowly with the sample size $n$ so that the excess
risk bounds converge to zero.

An easy example is when there is an $f^*\in \cF_r$ with $L(f^*)=0$,
where $r$ is a fixed complexity. Notice that this implies not
just that the conditional expectation is in $\cF_r$, but that there
is no noise, that is, almost surely $y=f^*(\bx)$. In that case,
if we choose $\widehat f$ as the interpolant $\widehat L(\widehat f)=0$ with
minimum complexity, then its complexity will certainly satisfy
$r(\widehat f)\le r(f^*)=r$. And then as the sample size $n$ increases,
$L(\widehat f)$ will approach zero.  In fact, since $\widehat L(\widehat f)=0$,
Theorem~\ref{theorem:fastRad} implies a faster rate in this case:
$L(\widehat f) = O((\log n)^4\bar R_n^2(\cF_r))$.

Theorem~\ref{thm:complexity-reg} shows that if we were to balance
the complexity with the fit to the training data, then we can hope
to enjoy excess risk as good as the best bound for any $\cF_r$ in
the complexity hierarchy.  If we always choose a perfect fit to the
data, there is no trade-off between complexity and empirical risk,
but when there is a prediction rule $f^*$ with finite complexity
and zero risk, then once the sample size is sufficiently large,
the best trade-off does correspond to a perfect fit to the data.
To summarize: when there is no noise, that is, when $y=f^*(x)$,
and $f^*\in \cF$, classical theory shows that a minimum-complexity
interpolant $\widehat f\in \cF$ will have risk $L(\widehat f)$ converging to
zero as the sample size increases.

But what if there is noise, that is, there is no deterministic
relationship between $\bx$ and $y$? Then it turns out that the bounds
on the excess risk $L(\widehat f)-L(f_{\cF}^*)$ presented in this section
must become vacuous: they can never decrease below a constant, no
matter how large the sample size.  This is because these bounds
do not rely on any properties of the distribution on $\calX$,
and hence are also true in a {\em fixed design} setting, where the
excess risk is at least the noise level. 

To make this precise, fix $\bx_1,\ldots,\bx_n\in\calX$ and define
the {\em fixed design risk}
$$ L_{|\bx}(f) := \frac{1}{n}\sum_{i=1}^n
      \Expect\left[\left.\ell(f(\bx_i),y)\right|\bx=\bx_i\right].$$
Then the decomposition~\eqref{eqn:decomp} extends to this risk:
for any $\widehat f$ and $f^*$,
  \begin{align*}
    \lefteqn{L_{|\bx}(\widehat f) - L_{|\bx}(f^*)} & \\*
      &= \left[ L_{|\bx}(\widehat f) - \widehat L(\widehat f) \right]
        + \left[ \widehat L(\widehat f) - \widehat L(f^*) \right]
        + \left[ \widehat L(f^*) - L_{|\bx}(f^*) \right].
  \end{align*}
For a nonnegative loss, the second term is nonpositive when $\widehat
L(\widehat f)=0$, and the last term is small for any fixed $f^*$.
Fix $f^*(\bx)=\Expect[y|\bx]$, and suppose we choose $\widehat f$ from a
class $\cF_r$. The same proof as that of Theorem~\ref{thm:Rad}
gives a Rademacher complexity bound on the first term above, and
\cite[Theorem~4.12]{lt-pbsip-91} implies the same contraction
inequality as in Theorem~\ref{thm:contraction} when $\widehat
y\mapsto\ell(\widehat y,y)$ is $c$-Lipschitz:
  \begin{align*}
    \Expect\sup_{f\in\cF_r} \left|L_{|\bx}(f) - \widehat L(f)\right|
    &\le 2 \Expect\left[\left.
      \sup_{f\in\cF_r}
      \left|\frac{1}{n}\sum_{i=1}^n \epsilon_i \ell(f(\bx_i),y_i)\right|
      \right|\bx_1,\ldots,\bx_n\right] \\
    &\le 4c \bar R_n(\cF_r).
  \end{align*}
Finally, although Theorems~\ref{thm:sigmoidRn} and
Theorem~\ref{thm:FrobeniusRn} are stated as bounds on the Rademacher
complexity of $\cF_r$, they are in fact bounds on $\bar R_n(\cF_r)$,
the worst-case empirical Rademacher complexity of $\cF$.

Consider the complexity hierarchy defined in Theorem~\ref{thm:sigmoidRn} 
or Theorem~\ref{thm:FrobeniusRn}.
For the minimum-complexity interpolant $\widehat f$, these theorems give
bounds that depend on the complexity $r(\widehat f)$, that is, bounds
of the form $L(\widehat f)-L(f^*)\le B(r(\widehat f))$ (ignoring
the fact that that the minimum complexity $r(\widehat f)$ is random;
making the bounds uniform over $r$ would give a worse bound).
Then these observations imply that
  \[
    \Expect \left[L_{|\bx}(\widehat f) - L_{|\bx}(f^*)\right]
    = \Expect L_{|\bx}(\widehat f) - L(f^*) \le \Expect B(r(\widehat f)).
  \]
But  then
  \begin{align*}
    \Expect B(r(\widehat f))
    & \ge \Expect\left[L_{|\bx}(\widehat f) - L(f^*)\right] 
    = \frac{1}{n}\sum_{i=1}^n\Expect\left(\widehat f(\bx_i)-f^*(\bx_i)\right)^2
    = L(f^*).
  \end{align*}
Thus, unless there is no noise, the upper bound on excess risk must be
at least as big as a constant.

\cite{bartlett2020failures} use a similar comparison between
prediction problems in random design and fixed design settings
to demonstrate situations where benign overfitting occurs but a
general family of excess risk bounds---those that depend only on
properties of $\widehat f$ and do not increase too quickly with
sample size---must sometimes be very loose. \cite{nk-ucmuegdl-19}
present a scenario where, with high probability, a classification
method gives good predictive accuracy but uniform convergence
bounds must fail for any function class that contains the
algorithm's output.  Algorithmic stability approaches---see
\cite{devroye1979distribution} and \cite{bousquet2002stability}---also aim to
identify sufficient conditions for closeness of risk and empirical
risk, and appear to be inapplicable in the interpolation
regime.  These examples illustrate that to understand benign
overfitting, new analysis approaches are necessary that exploit
additional information.  We shall review results of this kind in
Section~\ref{sec:benign}, for minimum-complexity interpolants in
regression settings. The notion of complexity that is minimized
is obviously of crucial importance here; this is the topic of the
next section.

%% file: implicit.tex
\section{Implicit regularization}
\label{sec:implicit}

When the model $\cF$ is complex enough to ensure zero empirical error, such as in the case of overparametrized neural networks, the set of empirical minimizers may be large. Therefore, it may very well be the case that some empirical minimizers generalize well while others do not. Optimization algorithms introduce a bias in this choice: an iterative method may converge to a solution with certain properties. Since this bias is a by-product rather than an explicitly enforced property, we follow the recent literature and call it \emph{implicit regularization}. In subsequent sections, we shall investigate statistical consequences of such implicit regularization.

Perhaps the simplest example of implicit regularization is gradient descent on the square-loss objective with linear functions:
\begin{align}
	\label{eq:grad_descent_square_loss_linear}
	\btheta_{t+1} = \btheta_t - \eta_t \grad \widehat{L}(\btheta_t),~~~~ \widehat{L}(\btheta)= \frac{1}{n}\norm{\bX \btheta-\by}^2_2,~~ \btheta_0=\boldsymbol{0}\in\reals^d,
\end{align}
where $\bX = [\bx_1,\ldots,\bx_n]^\tr\in \reals^{n\times d}$ and
$\by=[y_1,\ldots,y_n]^\tr$ are the training data, and $\eta_t>0$ is
the step size.
While the set of minimizers of the square-loss objective in the
overparametrized ($d>n$) regime  is an affine subspace of dimension at least $d-n$, gradient descent (with any choice of step size that ensures convergence) converges to a very specific element of this subspace: the minimum-norm solution
\begin{align}
	\label{eq:min_norm_linear}
	\widehat{\btheta} = \argmin{\btheta} \Big\{ \| \btheta \|_2: \inner{\btheta,\bx_i} = y_i ~\text{for all}~ i \leq n \Big\}.
\end{align}
This minimum-norm interpolant can be written in closed form as
\begin{align}
	\widehat{\btheta} = \bX^\dagger \by,
\end{align}	
where $\bX^\dagger$ denotes the pseudoinverse.
It can also be seen as a limit of \emph{ridge regression}
\begin{align}
	\label{eq:reg_least_squares}
	\btheta_{\lambda} = \argmin{\btheta} \frac{1}{n}\norm{\bX\theta-\by}^2_2 + \lambda\norm{\btheta}^2_2
\end{align}
as $\lambda\to 0^+$. The connection between minimum-norm interpolation \eqref{eq:min_norm_linear} and the ``ridgeless'' limit of ridge regression will be fruitful in the following sections when statistical properties of these methods are analyzed and compared.

To see that the iterations in
\eqref{eq:grad_descent_square_loss_linear} converge to the
minimum-norm solution, observe that the Karush-Kuhn-Tucker (KKT) conditions for the
constrained optimization problem \eqref{eq:min_norm_linear} are
$\bX \btheta = \by$ and $\btheta+\bX^\tr \boldsymbol{\mu} = 0$
for Lagrange multipliers $\boldsymbol{\mu}\in\reals^n$. Both
conditions are satisfied (in finite time or in the limit) by any
procedure that interpolates the data while staying in the span of the
rows of $\bX$, including \eqref{eq:grad_descent_square_loss_linear}.
It should be clear that a similar statement holds for more general
objectives $\widehat{L}(\btheta) = n^{-1}\sum_i
\ell(\inner{\btheta,\bx_i},y_i)$ under appropriate assumptions on
$\ell$. Furthermore, if started from an arbitrary $\btheta_0$,
gradient descent (if it converges) selects a solution that is closest to
the initialization with respect to $\norm{\cdot}_2$.

Boosting is another notable example of implicit regularization arising from the choice of the optimization algorithm, this time for the problem of classification. Consider the linear classification objective
\begin{align}
	\label{eq:boosting_classification_loss}
	\widehat{L}_{01}(\btheta) = \frac{1}{n}\sum_{i=1}^n
    \indicator{-y_i \inner{\btheta, \bx_i} \ge 0}
\end{align}
where $y_1,\ldots,y_n\in\{\pm1\}$. In the classical formulation of the boosting problem, the coordinates of vectors $\bx_i$ correspond to features computed by functions in some class of base classifiers.  
Boosting was initially proposed as a method for minimizing empirical classification loss \eqref{eq:boosting_classification_loss} by iteratively updating $\btheta$. In particular, AdaBoost \cite{fs-dgolab-97} corresponds to \emph{coordinate descent} on the exponential loss function 
\begin{align}
	\label{eq:explosssurrogate}
	\btheta\mapsto\frac{1}{n}\sum_{i=1}^n \exp\{-y_i \inner{\btheta,\bx_i}\}
\end{align} 
\cite{b-ac-98,friedman2001}. Notably, the minimizer of this
surrogate loss does not exist in the general separable case, and there
are multiple directions along which the objective decreases to $0$ as
$\norm{\btheta}\to\infty$. The AdaBoost optimization procedure and its
variants were observed empirically to shift the distribution of
margins (the values $y_i\inner{\btheta_t,\bx_i}$, $i=1,\ldots,n$)
during the optimization process in the positive direction even after
empirical classification error becomes zero, which in part motivated
the theory of large margin classification \cite{schapire1998boosting}. In the separable case, convergence to the direction of the maximizing $\ell_1$ margin solution 
\begin{align}
	\label{eq:min_norm_linear_l1}
	\widehat{\btheta} = \argmin{\btheta} \Big\{\| \btheta \|_1: y_i\inner{\btheta,\bx_i} \geq 1 ~\text{for all}~ i \leq n \Big\}
\end{align}
was shown in \cite{zhang2005boosting} and \cite{telgarsky2013margins} assuming small enough step size, where separability means positivity of the margin
\begin{align}
	\label{eq:max_l1_margin}
	\max_{\norm{\btheta}_1 = 1} \min_{i\in[n]}~ y_i\inner{\btheta,\bx_i}.
\end{align}

More recently, \cite{soudry2018implicit} and \cite{ji2018risk} have shown that gradient (rather than coordinate) descent on \eqref{eq:explosssurrogate} and separable data lead to a solution with direction approaching that of the maximum $\ell_2$ (rather than $\ell_1$) margin separator
\begin{align}
	\label{eq:min_norm_linear_l2}
	\widehat{\btheta} = \argmin{\btheta} \Big\{ \| \btheta \|_2 : y_i\inner{\btheta,\bx_i} \geq 1 ~\text{for all}~ i \leq n \Big\}.
\end{align}
We state the next theorem from \cite{soudry2018implicit} for the case of logistic loss, although essentially the same statement---up to a slightly modified step size upper bound---holds for any smooth loss function that has appropriate exponential-like tail behavior, including $\ell(u)=e^{-u}$  \cite{soudry2018implicit,ji2018risk}. 
\begin{theorem}
	Assume the data $\bX$, $\by$ are linearly separable. For logistic
    loss $\ell(u)=\log(1+\exp\{-u\})$, any step size $\eta\leq 8
    \lambda^{-1}_{\text{max}}(n^{-1}\boldsymbol{\bX}^\tr \boldsymbol{\bX})$,
    and any initialization $\btheta_0$, the gradient descent
    iterations
	$$\btheta_{t+1} = \btheta_t - \eta \nabla \widehat{L}(\btheta_t), ~~~\widehat{L}(\btheta) = \frac{1}{n}\sum_{i=1}^n \ell (y_i \inner{\bx_i, \btheta})$$
	satisfy $\btheta_t = \widehat{\btheta} \cdot \log t + \rho_t$
	where $\widehat{\btheta}$ is the $\ell_2$ max-margin solution in \eqref{eq:min_norm_linear_l2}.
	Furthermore, the residual grows at most as $\norm{\rho_t}= O(\log\log t)$, and thus
	$$\lim_{t\to\infty} \frac{\btheta_t}{\norm{\btheta_t}_2} = \frac{\widehat{\btheta}}{\|\widehat{\btheta}\|_2}.$$
\end{theorem}	

These results have been extended to multi-layer fully connected neural
networks and convolutional neural networks (without nonlinearities)
in  \cite{gunasekar2018implicit}. On the other hand,
\cite{gunasekar18a} considered the implicit bias arising from other
optimization procedures, including mirror descent, 
steepest descent, and AdaGrad, both in the case when the global minimum is attained (as for the square loss) and when the global minimizers are at infinity (as in the classification case with exponential-like tails of the loss function). We refer to \cite{ji2019refined} and \cite{nacson2019convergence} and references therein for further studies on faster rates of convergence to the direction of the max margin solution (with more aggressive time-varying step sizes) and on milder assumptions on the loss function. 

In addition to the particular optimization algorithm being employed, implicit regularization arises from the choice of model parametrization. Consider re-parametrizing the least-squares objective in \eqref{eq:grad_descent_square_loss_linear} as
\begin{align}
	\label{eq:sqrt_reparam}
	\min_{\bu\in\reals^d} \norm{\bX \btheta(\bu) - \by}^2_2,
\end{align}
where $\btheta(\bu)_i = \bu_i^2$ is the coordinate-wise square. \cite{gunasekar2017implicit} show that if $\btheta_\infty(\alpha)$ is the limit point of gradient flow on \eqref{eq:sqrt_reparam} with initialization $\alpha \boldsymbol{1}$ and the limit $\widehat{\btheta}=\lim_{\alpha\to 0} \btheta_\infty(\alpha)$ exists and satisfies $\bX \widehat{\btheta} = \by$, then it must be that
\begin{align}
	\widehat{\btheta} \in \argmin{\btheta\in \reals^d_+} \Big\{ \norm{\btheta}_1 : \inner{\btheta,\bx_i} = y_i ~\text{for all}~ i \leq n \Big\}.
\end{align}
In other words, in that case, gradient descent on the reparametrized
problem with infinitesimally small step sizes and infinitesimally
small initialization converges to the minimum $\ell_1$ norm solution
in the original space. More generally,
\cite{gunasekar2017implicit} and \cite{li2018algorithmic} proved an analogue of
this statement for matrix-valued $\btheta$ and $\bx_i$, establishing
convergence to the minimum nuclear-norm solution under additional assumptions on the $\bx_i$. The matrix version of the problem can be written as
$$\min_{\bU,\bV} \sum_{i=1}^n \ell (\inner{\bU\bV^\tr, \bx_i}, y_i),$$
which can be viewed, in turn, as an empirical risk minimization objective for a two-layer neural network with linear activation functions. 

In summary, in overparametrized problems that admit multiple minimizers of the empirical objective, the choice of the optimization method and the choice of parametrization both play crucial roles in selecting a minimizer with certain properties. As we show in the next section, these properties of the solution can ensure good generalization properties through novel mechanisms that go beyond the realm of uniform convergence.

%% file: benign.tex
\section{Benign overfitting}
\label{sec:benign}

We now turn our attention to generalization properties of specific
solutions that interpolate training data. As emphasized in
Section~\ref{sec:slt}, mechanisms of uniform convergence alone cannot
explain good statistical performance of such methods, at least in the
presence of noise.

For convenience, in this section we focus our attention on regression problems with square loss $\ell (f(\bx),y) = (f(\bx)-y)^2$. In this case, the regression function $f^*=\En[y|\bx]$ is a minimizer of $L(f)$, and excess loss can be written as
$$L(f) - L(f^*) = \En(f(\bx)-f^*(\bx))^2 =
\norm{f-f^*}^2_{L^2(\P)}.$$
We assume that for any $\bx$, conditional variance of the noise $\xi= y-f^*(\bx)$ is at most $\sigma_\xi^2$, and we write $\xi_i=y_i-f^*(\bx_i)$.

As in the previous section, we say that a solution $\algo$ is \emph{interpolating} if 
\begin{align}
	\algo(\bx_i) = y_i, ~~~ i=1,\ldots,n.
\end{align}
For learning rules $\algo$ expressed in closed form---such as local
methods and linear and kernel regression---it is convenient to employ a bias-variance decomposition that is different from the approximation-estimation error decomposition \eqref{eq:est_approx_decomp} in Section~\ref{sec:slt}.
First, for $\bX = [\bx_1,\ldots,\bx_n]^\tr\in \reals^{n\times d}$ and
$\by=[y_1,\ldots,y_n]^\tr$, conditionally on $\bX$, define
\begin{align}
	\label{eq:emp_bias_var_def}
	\biassquaredemp = \En_\bx\left(f^*(\bx) - \En_{\by} \algo(\bx)\right)^2,~~~~~ \varianceemp = \En_{\bx,\by}\left(\algo(\bx) - \En_{\by} \algo(\bx)\right)^2.
\end{align}
It is easy to check that
\begin{align}
	\En\|\algo-f^*\|^2_{L^2(\P)} = \En_{\bX} \left[ \biassquaredemp \right] + \En_{\bX} \left[\varianceemp\right].
\end{align}
In this section we consider linear (in $\by$) estimators of the form
$\algo(\bx) = \sum_{i=1}^n y_i \omega_i(\bx).$
For such estimators we have 
\begin{equation}\label{eqn:bias}
\biassquaredemp = \En_\bx\left(f^*(\bx)- \sum_{i=1}^n f^*(\bx_i) \omega_i(\bx)\right)^2
\end{equation}
and
\begin{equation}\label{eqn:var}
\varianceemp = \En_{\bx,\boldsymbol{\xi}}\left(\sum_{i=1}^n \xi_i
\omega_i(\bx)\right)^2 \leq \sigma^2_\xi \sum_{i=1}^n
\En_\bx\left(\omega_i(\bx)\right)^2,
\end{equation}
with equality if conditional noise variances are equal to $\sigma_\xi^2$ at each $\bx$.

In classical statistics, the balance between bias and variance is achieved by tuning an explicit parameter. Before diving into the more unexpected interpolation results, where the behavior of bias and variance are driven by novel self-regularization phenomena, we discuss the bias-variance tradeoff in the context of one of the oldest statistical methods.

\subsection{Local methods: Nadaraya-Watson}
\label{sec:local_methods}

Consider arguably the simplest nontrivial interpolation procedure, the
$1$-nearest neighbour (1-NN) $\algo(\bx)=y_{\mathsf{nn}(\bx)}$, where
$\mathsf{nn}(\bx)$ is the index of the datapoint closest to $\bx$ in
Euclidean distance. While we could view $\algo$ as an empirical minimizer in some effective class $\cF$ of possible functions (as a union for all possible $\{\bx_1,\ldots,\bx_n\}$), this set is large and growing with $n$. Exploiting the particular form of $1$-NN is, obviously, crucial. Since typical distances to the nearest neighbor in $\reals^d$ decay as $n^{-1/d}$ for i.i.d. data, in the noiseless case ($\sigma_\xi=0$) one can guarantee consistency and nonparametric rates of convergence of this interpolation procedure under continuity and smoothness assumptions on $f^*$ and the underlying measure.
 Perhaps more interesting is the case when the $\xi_i$ have non-vanishing variance. 
Here $1$-NN is no longer consistent in general (as can be easily seen
by taking $f^*=0$ and independent Rademacher $\xi_i$ at random
$\bx_i\in[0,1]$), although its asymptotic risk is at most $2L(f^*)$
\cite{cover1967nearest}. The reason for inconsistency is insufficient
averaging of the $y$-values, and this deficiency can be addressed by
averaging over the $k$ nearest neighbors with $k$ growing with $n$.
Classical smoothing methods generalize this idea of local averaging;
however, averaging forgoes empirical fit to data in favor of estimating the regression function under smoothness assumptions. While this has been the classical view, estimation is not necessarily at odds with fitting the training data for these local methods, as we show next.

The Nadaraya-Watson (NW) smoothing estimator
\cite{nadaraya1964estimating,watson1964smooth} is defined as
\begin{align}
	\label{eq:NW}
	\algo(\bx) = \sum_{i=1}^n y_i \omega_i(\bx),~~~~~~~ \omega_i(\bx) = \frac{K((\bx-\bx_i)/h)}{\sum_{j=1}^n K((\bx-\bx_j)/h)},
\end{align}	
where $K(u):\reals^d\to\reals_{\geq 0}$ is a kernel and $h> 0$ is a bandwidth parameter. For standard kernels used in practice---such as the Gaussian, uniform, or Epanechnikov kernels---the method averages the $y$-values in a local neighborhood around $\bx$, and, in general, does not interpolate. However, as noted by \cite{devroye1998hilbert}, a kernel that is singular at $0$ does interpolate the data. While the Hilbert kernel $K(u)=\norm{u}_2^{-d}$, suggested in \cite{devroye1998hilbert}, does not enjoy non-asymptotic rates of convergence, its truncated version
\begin{align}
	\label{eq:ker_nw}
	K(u) = \norm{u}^{-a}_2 \indicator{\norm{u}_2\leq 1},~ u\in\reals^d
\end{align}
 with a smaller power $0<a<d/2$ was shown in \cite{belkin2019does} to
 lead to minimax optimal rates of estimation under the corresponding smoothness assumptions. Notably, the NW estimator with the kernel in \eqref{eq:ker_nw} is necessarily interpolating the training data for any choice of $h$.

Before stating the formal result, define the H\"older class $H(\beta,L)$, for $\beta\in(0,1]$, as the class of functions $f:\reals^d\to \reals$ satisfying
$$\forall \bx,\bx'\in\reals^d,~~~~ |f(\bx)-f(\bx')|\leq L\norm{\bx-\bx'}_2^\beta.$$
The following result appears in \cite{belkin2019does}; see also \cite{belkin2018overfitting}:
\begin{theorem}
	\label{thm:NW}
	Let $f^*\in H(\beta,L)$ for $\beta\in(0,1]$ and $L>0$. Suppose the marginal density $p$ of $\bx$ satisfies $0<p_{\text{min}}\leq p(\bx) \leq p_{\text{max}}$ for all $\bx$ in its support. Then the estimator \eqref{eq:NW} with kernel \eqref{eq:ker_nw} satisfies\footnote{In the remainder of this paper, the symbol $\lesssim$ denotes inequality up to a multiplicative constant.}
		\begin{align}
			\En_{\bX} \left[\biassquaredemp\right] \lesssim
            h^{2\beta},~~~~\En_{\bX} \left[\varianceemp\right]
            \lesssim \sigma^2_\xi (nh^d)^{-1} \, .
		\end{align}
\end{theorem}
The result can be extended to smoothness parameters $\beta>1$
\cite{belkin2019does}. The choice of $h=n^{-1/(2\beta+d)}$
balances the two terms and leads to minimax optimal rates for H\"older classes \cite{tsybakov2008introduction}. 

In retrospect, Theorem~\ref{thm:NW} should not be surprising, and we mention
it here for pedagogical purposes. It should be clear from the
definition \eqref{eq:NW} that the behavior of the kernel at $0$,
and in particular the presence of a singularity, determines
whether the estimator fits the training data exactly. This is,
however, decoupled from the level of smoothing, as given by the
bandwidth parameter $h$.  In particular, it is the choice of $h$
alone that determines the bias-variance tradeoff, and the value
of the empirical loss cannot inform us whether the estimator is
over-smoothing or under-smoothing the data.

The NW estimator with the singular kernel can be also viewed as
adding small ``spikes'' at the datapoints on top of the general
smooth estimate that arises from averaging the data in a neighborhood
of radius $h$.  This suggests a rather obvious scheme for changing
any estimator $\algo_0$ into an interpolating one by adding small
deviations around the datapoints: $\algo(\bx) \deq \algo_0(\bx)
+ \Delta(\bx)$, where $\Delta(\bx_j)=y_i-\algo_0(\bx_j)$ but
$\|\Delta\|_{L^2(\P)}=o(1)$.  The component $\algo_0$ is useful
for prediction because it is smooth, whereas the spiky
component $\Delta$ is useful for interpolation but does not harm
the predictions of $\algo$. Such combinations have been observed
experimentally in other settings and described as ``spiked-smooth''
estimates \cite{wyner2017explaining}.  The examples that we see
below suggest that interpolation may be easier to achieve with
high-dimensional data than with low-dimensional data, and this is
consistent with the requirement that the overfitting component
$\Delta$ is benign: it need not be too ``irregular'' in high
dimensions, since typical distances between datapoints in $\reals^d$
scale at least as $n^{-1/d}$.

\subsection{Linear regression in the interpolating regime}
\label{sec:lin_regression}

In the previous section, we observed that the spiky part of
the NW estimator, which is responsible for interpolation,
does not hurt the out-of-sample performance when measured in
$L^2(\P)$.  The story for minimum-norm interpolating linear and
kernel regression is significantly more subtle: there is also
a decomposition into a prediction component and an overfitting
component, but there is no explicit parameter that trades off bias
and variance. The decomposition depends on the distribution of the
data, and the overfitting component provides a \emph{self-induced
regularization}\footnote{ This is not to be confused with {\em
implicit regularization}, discussed in Section~\ref{sec:implicit},
which describes the properties of the particular empirical
risk minimizer that results from the choice of an optimization
algorithm.  Self-induced regularization is a statistical property
that also depends on the data-generating mechanism.}, similar to the
regularization term in ridge regression~\eqref{eq:reg_least_squares},
and this determines the bias-variance trade-off.

Consider the problem of linear regression in the over-parametrized
regime. We assume that the regression function $f^*(\bx) =
f(\bx;\btheta^*)=\inner{\btheta^*,\bx}$ with $\btheta^*,\bx\in
\reals^d$.  We also assume $\Expect\bx=0$.  (While we present the
results for finite $d>n$, all the statements in this section hold
for separable Hilbert spaces of infinite dimension.)

It is easy to see that the excess square loss can be written as
  $$
    L(\widehat{\btheta})-L(\btheta^*)
    = \Expect\left(f(\widehat\btheta)-f(\btheta^*)\right)^2
    = \|\widehat{\btheta}-\btheta^*\|^2_{\bSigma},
  $$
where we write $\|\bv\|_\Sigma^2:=\bv^\tr\bSigma\bv$ and
$\bSigma = \En \bx\bx^\tr$.
Since $d>n$, there is not enough data to learn all the $d$ directions
of $\btheta^*$ reliably, unless $\bSigma$ has favorable spectral
properties. To take advantage of such properties, classical
methods---as described in Section~\ref{sec:slt}---resort to explicit
regularization (shrinkage) or model complexity control, which
inevitably comes at the expense of not fitting the noisy data exactly.
In contrast, we are interested in estimates that interpolate the data.
Motivated by the properties of the gradient descent method~\eqref{eq:grad_descent_square_loss_linear}, we consider the minimal norm linear
function that fits the data $\bX$, $\by$ exactly:
\begin{align}\label{eq:minnormlin}
	\widehat{\btheta} = \argmin{\btheta} \Big\{ \| \btheta \|_2: \inner{\btheta,\bx_i} = y_i ~\text{for all}~ i \leq n \Big\}.
\end{align}
The solution has a closed form and yields the estimator
\begin{align}
	\label{eq:def_linear_reg}
	\algo(\bx) = \langle \widehat{\btheta}, \bx \rangle = \langle \bX^\dagger \by, \bx \rangle
	= (\bX \bx)^\tr (\bX\bX^\tr)^{-1} \by,
\end{align}	
which can also be written as $\algo(\bx) = \sum_{i=1}^n y_i
\omega_i(\bx)$, with
\begin{align}
	\omega_i(\bx) = (\bx^\tr \bX^\dagger)_i = (\bX \bx)^\tr (\bX\bX^\tr)^{-1} \be_i.
\end{align}	
Thus, from \eqref{eqn:bias}, the bias term can be written as
\begin{align}
	\label{eq:bias_linear_emp}
	\biassquaredemp = \En_\bx \inner{P^\perp \bx,\btheta^*}^2 = \norm{ \bSigma^{1/2} P^{\perp} \btheta^*}^2_2,
\end{align}
where $P^\perp = \bI_d - \bX^\tr (\bX\bX^\tr)^{-1}\bX$, and
from \eqref{eqn:var}, the variance term is
\begin{align}
	\label{eq:var_linear}
	\varianceemp \leq \sigma^2_\xi \cdot \En_\bx\norm{(\bX\bX^\tr)^{-1} (\bX \bx)}^2_2 = \sigma^2_\xi \cdot \trace\left((\bX\bX^\tr)^{-2} \bX \bSigma \bX^\tr \right).
\end{align}

We now state our assumptions.
\begin{assumption}
	\label{assmpt:linear_reg}
	Suppose $\bz=\bSigma^{-1/2} \bx$ is $1$-sub-Gaussian. Without loss of generality, assume
    $\bSigma=\diag(\lambda_1, \ldots,\lambda_d)$
    with $\lambda_1\geq  \cdots \ge\lambda_d$.
\end{assumption}

The central question now is: Are there mechanisms that
can ensure small bias and variance of the minimum-norm
interpolant? Surprisingly, we shall see that the answer is yes. To
this end, choose an index $k\in\{1,\ldots,d\}$ and consider
the subspace spanned by the top $k$ eigenvectors corresponding
to $\lambda_1,\ldots,\lambda_k$. Write $\bx^\tr=[\bx_{\leq
k}^\tr,\bx_{> k}^\tr]$.  For an appropriate choice of $k$,
it turns out the decomposition of the minimum-norm interpolant as
$\langle\widehat\btheta,\bx\rangle = \langle\widehat\btheta_{\leq k},\bx_{\leq
k}\rangle + \langle\widehat\btheta_{> k},\bx_{> k}\rangle$ corresponds
to a decomposition into a prediction component and an interpolation
component. Write the data matrix as
$\bX = [\bX_{\leq k}, \bX_{>k}]$ and
\begin{align}
	\label{eq:splitting_two_parts}
	\bX\bX^\tr = \bX_{\leq k}\bX_{\leq k}^\tr + \bX_{> k}\bX_{> k}^\tr.
\end{align}
Observe that if the eigenvalues of the second part were to be
contained in an interval $[\gamma/c,c \gamma]$ for some $\gamma$ and a
constant $c$, we could write
\begin{align}
	\label{eq:replace_tail_reg}
	\bX_{\leq k}\bX_{\leq k}^\tr + \gamma \bM,
\end{align}
where $c^{-1}\bI_n\preceq\bM\preceq c\bI_n$.  If we replace
$\bM$ with the approximation $\bI_n$ and substitute this expression
into \eqref{eq:def_linear_reg}, we see that $\gamma$ would have
an effect similar to \emph{explicit} regularization through a
ridge penalty: if that approximation were precise, the first $k$
components of $\widehat\btheta$ would correspond to
\begin{align}
	\widehat{\btheta}_{\leq k} = \argmin{\btheta \in\reals^k} \norm{\bX_{\leq k} ~ \btheta - \by}^2_2 + \gamma\norm{\btheta}^2_2,
\end{align}
since this has the closed-form solution $\bX_{\leq k}^\tr (\bX_{\leq
k}\bX_{\leq k}^\tr + \gamma\bI_n)^{-1} \by$. Thus, if $\gamma$ is not too
large, we might expect this  approximation to have a minimal impact on the bias and
variance of the prediction component.

It is, therefore, natural to ask when to expect such a near-isotropic
behavior arising from the ``tail'' features. The following lemma
provides an answer to this question \cite{bartlett2020benign}:
\begin{lemma} 
	\label{eq:near_isometry_indep}
	Suppose coordinates of $\bSigma^{-1/2}\bx$ are independent. Then there exists a constant $c>0$ such that, with probability at least $1-2\exp\{-n/c\}$,
\begin{align*}
  \frac{1}{c} \sum_{i > k} \lambda_i - c\lambda_{k+1} n
  & \leq \lambda_{\text{min}}(\bX_{> k}\bX_{> k}^\tr) \\*
  &\leq \lambda_{\text{max}}(\bX_{> k}\bX_{> k}^\tr) 
  \leq c\left( \sum_{i > k} \lambda_i + \lambda_{k+1} n\right).
\end{align*}
\end{lemma}
The condition of independence of coordinates in
Lemma~\ref{eq:near_isometry_indep} is satisfied for Gaussian
$\bx$. It can be relaxed to the following small-ball assumption:
\begin{align}
	\label{eq:small_ball}
	\exists c>0:~ \P(c\norm{\bx}_2^2\geq \En\norm{\bx}_2^2)\geq 1-\delta.
\end{align}
Under this assumption, the conclusion of Lemma~\ref{eq:near_isometry_indep} still holds with probability at least $1-2\exp\{-n/c\}-n\delta$ \cite{tsigler2020benign}.

An appealing consequence of Lemma~\ref{eq:near_isometry_indep}
is the small condition number of $\bX_{> k}\bX_{> k}^\tr$ for any
$k$ such that $\sum_{i>k}\lambda_i \gtrsim \lambda_{k+1}n$. Define
the \emph{effective rank} for a given index $k$ by $$r_k(\bSigma)
= \frac{\sum_{i>k} \lambda_i}{\lambda_{k+1}}.$$ We see that
$r_k(\bSigma) \ge b n$ for some constant $b$ implies that the set
of eigenvalues of $\bX_{> k}\bX_{> k}^\tr$ lies in the interval
$[\gamma/c,c\gamma]$ for
  \begin{equation*}
    \gamma = \sum_{i>k} \lambda_i,
  \end{equation*}
and thus the scale of the self-induced regularization
in~\eqref{eq:replace_tail_reg} is
the sum of the tail eigenvalues of the covariance operator.
Interestingly, the reverse implication also holds: if for some $k$ the
condition number of $\bX_{> k}\bX_{> k}^\tr$ is at most $\kappa$ with
probability at least $1-\delta$, then effective rank $r_k(\bSigma)$ is
at least $c_\kappa n$ with probability at least
$1-\delta-c\exp\{-n/c\}$ for some constants $c,c_\kappa$. Therefore,
the condition $r_k(\bSigma)\gtrsim n$ characterizes the indices $k$
such that $\bX_{> k}\bX_{> k}^\tr$ behaves as a scaling of $\bI_d$,
and the scaling is proportional to $\sum_{i>k}\lambda_i$. We may call
the smallest such index $k$ \emph{the effective dimension}, for
reasons that will be clear in a bit.

How do the estimates on tail eigenvalues help in controlling the
variance of the minimum-norm interpolant? Define
$$\bSigma_{\leq k}=\diag(\lambda_1,\ldots, \lambda_k),~~~ \bSigma_{> k}=\diag(\lambda_{k+1},\ldots, \lambda_d).$$
Then, omitting $\sigma_\xi^2$ for the moment, the variance upper bound in \eqref{eq:var_linear} can be estimated by
\begin{align}
	\trace\left((\bX\bX^\tr)^{-2} \bX \bSigma \bX^\tr \right) 
		&\lesssim \trace\left((\bX\bX^\tr)^{-2} \bX_{\leq k}
        \bSigma_{\leq k} \bX^\tr_{\leq k} \right) \notag\\
		&\qquad {}
        + \trace\left((\bX\bX^\tr)^{-2} \bX_{> k} \bSigma_{> k} \bX^\tr_{> k} \right). \label{eq:lin_var_decomp2}
\end{align}
The first term is further upper-bounded by 
\begin{align}
	\trace\left((\bX_{\leq k}\bX_{\leq k}^\tr)^{-2} \bX_{\leq k} \bSigma_{\leq k} \bX^\tr_{\leq k} \right),
\end{align}
and its expectation corresponds to the variance of $k$-dimensional
regression, which is of the order of $k/n$. On the other hand, by Bernstein's inequality, with probability at least $1-2\exp^{-cn}$,
\begin{align}
	\trace(\bX_{> k} \bSigma_{> k} \bX^\tr_{> k}) \lesssim n\sum_{i>k}
    \lambda_i^2,
\end{align}
so we have that the second term in~\eqref{eq:lin_var_decomp2} is, with high probability, of order at most
$$\frac{n\sum_{i>k} \lambda_i^2}{(\sum_{i>k} \lambda_i)^2}.$$
Putting these results together, we have the following theorem \cite{tsigler2020benign}:
\begin{theorem}
	\label{thm:var_linear}
    Fix $\delta<1/2$.
    Under Assumption~\ref{assmpt:linear_reg}, 
	suppose for some $k$ the condition number of $\bX_{> k}\bX_{>
    k}^\tr$ is at most $\kappa$ with probability at least $1-\delta$.
    Then
	\begin{align}
		\label{eq:var_linear_condition}
		\varianceemp \lesssim \sigma^2_\xi \kappa^2
        \log\left(\frac{1}{\delta}\right) 
        \left( \frac{k}{n} + \frac{ n \sum_{i>k} \lambda_i^2}{(\sum_{i>k} \lambda_i)^2}\right) 
	\end{align}	
	with probability at least $1-2\delta$.
\end{theorem}

We now turn to the analysis of the bias term. Since the projection operator in \eqref{eq:bias_linear_emp} annihilates any vector in the span of the rows of $\bX$, we can write 
\begin{align}
	\label{eq:simple_bias_calc}
	\biassquaredemp = \norm{\bSigma^{1/2} P^\perp  \btheta^*}^2_2 
	&= \norm{(\bSigma - \widehat{\bSigma})^{1/2}P^\perp  \btheta^*}^2_2,
\end{align}
where $\widehat{\bSigma} = n^{-1}\bX^\tr \bX$ is the sample covariance operator. Since projection contracts distances, we obtain an upper bound 
\begin{align}
	\norm{(\bSigma - \widehat{\bSigma})^{1/2} \btheta^*}^2_2 \leq
    \norm{\btheta^*}^2_2 \times \norm{\bSigma-\widehat{\bSigma}}.
\end{align}
The rate of approximation of the covariance operator by its sample-based counterpart has been studied in \cite{koltchinskii2017concentration}, and we conclude 
	\begin{align}
		\label{eq:bias_linear_koltchinskii}
		\biassquared \lesssim \norm{\btheta^*}^2_{\bSigma}  \max\left\{\sqrt{\frac{r_0(\bSigma)}{n}}, \frac{r_0(\bSigma)}{n}\right\}
	\end{align}	
(see \cite{bartlett2020benign} for details).

The upper bound in \eqref{eq:bias_linear_koltchinskii} can be sharpened significantly by analyzing the bias in the two subspaces, as proved in \cite{tsigler2020benign}:
\begin{theorem}
	\label{thm:bias_split_linear}
	Under the assumptions of Theorem~\ref{thm:var_linear}, for
    $n\gtrsim\log(1/\delta)$, with probability at least $1-2\delta$,
	\begin{align}
		\label{eq:bias_split_linear}
		\biassquaredemp \lesssim \kappa^4\left[\norm{\btheta^*_{\leq k}}_{\bSigma_{\leq k}^{-1}}^2 \left( \frac{\sum_{i>k} \lambda_i}{n}\right)^2 + \norm{\btheta^*_{>k}}^2_{\bSigma_{>k}}\right].
	\end{align}	
\end{theorem}

The following result shows that without further assumptions, the
bounds on variance and bias given in Theorems~\ref{thm:var_linear}
and~\ref{thm:bias_split_linear} cannot be improved by more than
constant factors; see \cite{bartlett2020benign} and
\cite{tsigler2020benign}.

\begin{theorem}\label{thm:lower_linear}
    There are absolute constants $b$ and $c$ such that for
    Gaussian $\bx\sim\normal(0,\bSigma)$, where $\bSigma$ has
    eigenvalues $\lambda_1\ge\lambda_2\ge\cdots$,
    with probability at least $1-\exp(-n/c)$,
	\begin{align*}
		\varianceemp \gtrsim
        1\wedge\left(
		\sigma^2_\xi \left( \frac{k}{n}
		+ \frac{ n \sum_{i>k} \lambda_i^2}{(\sum_{i>k}
		\lambda_i)^2}\right)\right),
	\end{align*}
    where $k$ is the effective dimension,
    $ k = \min\left\{l: r_l(\bSigma)\ge bn\right\}$.
    Furthermore, for any $\btheta\in\Re^d$, if the regression function
    $f^*(\cdot)=\langle\cdot,\btheta^*\rangle$, where
    $\theta^*_i=\epsilon_i\theta_i$ and
    $\bepsilon=(\epsilon_1,\ldots,\epsilon_d)\sim\Unif\left(\{\pm 1\}^d\right)$,
    then with probability at least $1-\exp(-n/c)$,
	\begin{align*}
		\Expect_{\bepsilon}\biassquaredemp \gtrsim
		\left[\norm{\btheta^*_{\leq k}}_{\bSigma_{\leq k}^{-1}}^2
		\left( \frac{\sum_{i>k} \lambda_i}{n}\right)^2 +
		\norm{\btheta^*_{>k}}^2_{\bSigma_{>k}}\right].
	\end{align*}
\end{theorem}

A discussion of Theorems~\ref{thm:var_linear},
\ref{thm:bias_split_linear} and~\ref{thm:lower_linear} is in order.
First, the upper and lower bounds match up to constants, and
in particular both involve the decomposition  of $\widehat f$
into a prediction component $\widehat
f_0(\bx):=\langle\widehat\btheta_{\leq k},\bx_{\leq k}\rangle$
and an interpolation component
$\Delta(\bx):=\langle\widehat\btheta_{> k},\bx_{>
k}\rangle$ with distinct bias and variance contributions, so
this decomposition is not an artifact of our analysis.  Second, the
$\norm{\btheta^*_{>k}}^2_{\bSigma_{>k}}$ term in the bias and the
$k/n$ term in the variance for the prediction component
$\widehat f_0$
correspond to the terms we would get by performing ordinary
least-squares (OLS) restricted to the first $k$ coordinates of
$\btheta$. Provided $k$ is small compared to $n$, there is enough
data to estimate the signal in this $k$-dimensional component, and
the bias contribution is the approximation error due to truncation
at $k$.  The other aspect of the interpolating component
$\Delta$ that could
harm prediction accuracy is its variance term. The definition of
the effective dimension $k$ implies that this is no more than a
constant, and it is small if the tail eigenvalues decay slowly and
$d-k\gg n$, for in that case, the ratio of the squared $\ell_1$
norm to the squared $\ell_2$ norm of these eigenvalues is large
compared to $n$;
overparametrization is important.
Finally, the bias and variance terms are similar to
those that arise in ridge regression~\eqref{eq:reg_least_squares},
with the regularization coefficient determined by the self-induced
regularization. Indeed, define
  \begin{align}
    \lambda = \frac{b}{n}\sum_{i>k}\lambda_i
  \end{align}
for the constant $b$ in the definition of the effective dimension
$k$.  That definition implies that $\lambda_k\ge \lambda\ge
\lambda_{k+1}$, so we can write the bias and variance terms, within
constant factors, as
  \begin{align*}
    \biassquaredemp & \approx \sum_{i=1}^d
      {\theta^*_i}^2
      \frac{\lambda_i}{\left(1+\lambda_i/\lambda\right)^2}, &
    \varianceemp & \approx \frac{\sigma_\xi^2}{n}\sum_{i=1}^d
      \left(\frac{\lambda_i}{\lambda+\lambda_i}\right)^2.
  \end{align*}
These are reminiscent of the bias and variance terms that arise
in ridge regression~\eqref{eq:reg_least_squares}. Indeed,
a ridge regression estimate in a fixed design setting with
$\bX^\tr\bX=\diag(s_1,\ldots,s_d)$ has precisely these bias and
variance terms with $\lambda_i$ replaced by $s_i$; see, for example,
\cite[Lemma~1]{JMLR:v14:dhillon13a}. In Section~\ref{sec:dpropn},
we shall see the same bias-variance decomposition arise in a related
setting, but with the dimension growing with sample size.

\subsection{Linear regression in Reproducing Kernel Hilbert Spaces}
\label{sec:ker_regression}

Kernel methods are among the core algorithms in machine learning and
statistics. These methods were introduced to machine learning in the
pioneering work of \cite{aiserman1964theoretical} as a generalization
of the Perceptron algorithm to nonlinear functions by lifting the
$\bx$-variable to a high- or infinite-dimensional feature space. Our
interest in studying kernel methods here is two-fold: on the one hand,
as discussed in detail in Sections~\ref{sec:efficient} and
\ref{sec:NTK}, sufficiently wide neural networks with random
initialization stay close to a certain kernel-based solution during
optimization and are essentially equivalent to a minimum-norm
interpolant; on the other hand, it has been noted that kernel methods
exhibit similar surprising behavior of benign interpolation to neural networks \cite{belkin2018understand}. 

A kernel method in the regression setting amounts to choosing a
feature map $\bx\mapsto \phi(\bx)$ and computing a (regularized)
linear regression solution in the feature space. While
Section~\ref{sec:lin_regression} already addressed the question of
overparametrized linear regression, the non-linear feature map
$\phi(\bx)$ might not satisfy Assumption~\ref{assmpt:linear_reg}.
In this section, we study interpolating RKHS regression estimates
using a more detailed analysis of certain random kernel matrices.

Since the linear regression solution involves inner products of
$\phi(\bx)$ and $\phi(\bx')$, the feature maps do not need to be
computed explicitly. Instead, kernel methods rely on a kernel function $k:\cX\times\cX\to \reals$ that, in turn, corresponds to an RKHS $\cH$. A classical method is kernel ridge regression (KRR)
\begin{align}
	\algo = \argmin{f\in\cH} \frac{1}{n}\sum_{i=1} (f(\bx_i)-y_i)^2 + \lambda\norm{f}^2_{\cH},
\end{align}
which has been extensively analyzed through the lens of bias-variance tradeoff with an appropriately tuned parameter $\lambda>0$ \cite{caponnetto2007optimal}. As $\lambda\to 0^+$, we obtain a minimum-norm interpolant
\begin{align}
	\label{eq:kernel_min_norm_interpolant}
	\algo = \argmin{f\in\cH} \Big\{ \| f \|_{\cH}: f(\bx_i) = y_i ~\text{for all}~ i \leq n \Big\},
\end{align}
which has the closed-form solution
\begin{align}
	\algo(\bx) = 
    K(\bx,\bX)^\tr K(\bX,\bX)^{-1} \by,
\end{align}
assuming $K(\bX,\bX)$ is invertible; see \eqref{eq:minnormlin}
and~\eqref{eq:def_linear_reg}. Here $K(\bX,\bX)\in\reals^{n\times n}$ is the kernel matrix with $$[K(\bX,\bX)]_{i,j} = k(\bx_i,\bx_j) ~~\text{and}~~ K(\bx,\bX)=[k(\bx,\bx_1),\ldots,k(\bx,\bx_n)]^\tr.$$
Alternatively, we can write the solution as 
$$\algo(\bx) = \sum_{i=1}^n y_i \omega_i(\bx) ~~\text{with}~~ \omega_i(\bx) = K(\bx,\bX) K(\bX,\bX)^{-1}\be_i,$$ which makes it clear that $\omega_i(\bx_j)=\indicator{i=j}$.
We first describe a setting where this approach does not lead to
benign overfitting.

\subsubsection{The Laplace kernel with constant dimension}

	We consider the Laplace (exponential) kernel on $\Re^d$
    with parameter $\sigma>0$:
	$$k_\sigma(\bx,\bx') = \sigma^{-d} \exp\{-\norm{\bx-\bx'}_2/\sigma \}.$$
The RKHS norm corresponding to this kernel can be related to a
Sobolev norm, and its RKHS has been shown
\cite{bach2017breaking,geifman2020similarity,chen2020deep}
to be closely related to the RKHS
corresponding to the Neural Tangent Kernel (NTK), which we study in
Section~\ref{sec:NTK}.

To motivate the lower bound, consider $d=1$. In this case, the
minimum-norm solution with the Laplace kernel corresponds to a rope
hanging from nails at heights $y_i$ and locations $\bx_i\in\reals$. If
points are ordered $\bx_{(1)}\leq \bx_{(2)}\leq \ldots \leq
\bx_{(n)}$, the form of the minimum-norm solution between two adjacent
points $\bx_{(i)},\bx_{(i+1)}$ is only affected by the values
$y_{(i)},y_{(i+1)}$ at these locations. As $\sigma\to\infty$, the
interpolant becomes piece-wise linear, while for $\sigma\to 0$, the
solution is a sum of spikes at the datapoints and zero everywhere
else. In both cases, the interpolant is not consistent: the error
$\En\|\algo-f^*\|^2_{L^2(\P)}$ does not converge to $0$ as $n$ increases. Somewhat surprisingly, there is no choice of $\sigma$ that can remedy the problem, even if $\sigma$ is chosen in a data-dependent manner. 

The intuition carries over to the more general case, as long as $d$ is a constant. The following theorem appears in \cite{rakhlin2019consistency}:

\begin{theorem}
    Suppose $f^*$ is a smooth function defined on a unit ball in $\reals^d$. Assume the probability distribution of $\bx$ has density that is bounded above and away from $0$. Suppose the noise random variables $\xi_i$ are independent Rademacher.\footnote{$\mathbb{P}(\xi_i=\pm1)=1/2$.} For fixed $n$ and odd $d$, with probability at least $1-O(n^{-1/2})$, for any choice $\sigma>0$,
    $$\|\algo-f^*\|^2_{L^2(\P)}  =  \Omega_d(1).$$
\end{theorem}

Informally, 
the minimum-norm interpolant with the Laplace kernel does not have the
flexibility to both estimate the regression function and generate
interpolating spikes with small $L^2(\P)$ norm if the dimension $d$ is
small. For high-dimensional data, however, minimum-norm
interpolation with the same kernel can be more benign, as we see
in the next section.

\subsubsection{Kernels on $\Re^d$ with $d\asymp n^{\alpha}$}

Since $d=O(1)$ may lead to inconsistency of the minimum-norm interpolator, we consider here a scaling $d\asymp n^{\alpha}$ for $\alpha\in(0,1]$. 
Some assumption on the independence of coordinates is needed to circumvent the lower bound of the previous section, and we assume the simplest possible scenario: each coordinate of $\bx\in\reals^d$ is independent.
\begin{assumption} 
	\label{asmpt:polynomials}
	Assume that $\bx\sim \P = p^{\otimes d}$ such that $z\sim p$ is
    mean-zero, that for some $C>0$ and $\nu>1$, $\P(|z|\geq t)\leq
    C(1+t)^{-\nu}$ for all $t\geq 0$, and that $p$ does not contain atoms.
\end{assumption}

We only state the results for the inner-product kernel
$$k(\bx,\bx') = h\left(\frac{\inner{\bx,\bx'}}{d}\right),~~~~ h(t) = \sum_{i=0}^\infty \alpha_i t^i,~~~~ \alpha_i\geq 0$$
and remark that more general rotationally invariant kernels (including
NTK: see Section~\ref{sec:NTK}) exhibit the same behavior under the
independent-coordinate assumption \cite{liang2020multiple}.

For brevity, define $\boldsymbol{K} = n^{-1} K(\bX,\bX)$.
Let $\br=(r_1,\cdots,r_d) \geq 0$ be a multi-index, and write
$\norm{\br}=\sum_{i=1}^d r_i$. With this notation, each entry of the kernel matrix can be expanded as
	\begin{align*}
		n \boldsymbol{K}_{ij} & = \sum_{\iota=0}^\infty \alpha_\iota
        \left( \frac{\langle \bx_i, \bx_j\rangle}{d} \right)^{\iota} =
        \sum_{\br} ~ c_{\br} \alpha_{\norm{\br}} p_{\br}(\bx_i)
        p_{\br}(\bx_j)/ d^{\norm{\br}}
	\end{align*}
	with
		$$c_{\br} = \frac{(r_1+\cdots+r_d)!}{r_1!\cdots r_d!},$$ and
        the monomials are $p_{\br}(\bx_i) = (\bx_i[1])^{r_1}\cdots
        (\bx_i[d])^{r_d}$ . If $h$ has infinitely many positive
        coefficients $\alpha$, each $\bx$ is lifted to an
        infinite-dimensional space. However, the
        resulting feature map $\phi(\bx)$ is not (in
        general) sub-Gaussian. Therefore, results from Section~\ref{sec:lin_regression} are not immediately applicable and a more detailed analysis that takes advantage of the structure of the feature map is needed.
		
As before, we separate the high-dimensional feature map into two parts, one corresponding to the prediction component, and the other corresponding to the overfitting part of the minimum-norm interpolant. More precisely, the truncated function $h^{\leq \iota}(t) = \sum_{i=0}^\iota \alpha_i t^i$ leads to the degree-bounded component of the empirical kernel:
	\begin{align*}
		n \boldsymbol{K}_{ij}^{[\leq \iota]} :=  \sum_{\norm{\br}\leq \iota}~ c_{\br} \alpha_{\norm{\br}} p_{\br}(\bx_i) p_{\br}(\bx_j)/ d^{\norm{\br}},~~~ n \boldsymbol{K}^{[\leq \iota]} =\Phi \Phi^\top
	\end{align*}
	with data $\bX\in\reals^{n\times d}$ transformed into polynomial features $\Phi \in \reals^{n \times \binom{\iota+d}{\iota}}$ defined as
	\begin{align*}
		\Phi_{i, \br} = \left(c_{\br} \alpha_{\norm{\br}}\right)^{1/2}  p_{\br}(\bx_i) / d^{\norm{\br}/2} \enspace.
	\end{align*}
	
	The following theorem reveals the staircase structure of the
    eigenvalues of the kernel, with $\Theta(d^{\iota})$ eigenvalues of
    order $\Omega(d^{-\iota})$, as long as $n$ is large enough to
    sketch these directions; see
    \cite{liang2020multiple} and \cite{ghorbani2020linearized}.
	\begin{theorem}
		\label{thm:structure_eigenvalues_polynomials}
		Suppose $\alpha_0,\ldots,\alpha_{\iota_0}>0$ and
        $d^{\iota_0}\log d = o(n)$. Under
        Assumption~\ref{asmpt:polynomials}, with probability at least
        $1-\exp^{-\Omega(n/d^{\iota_0})}$, for any $\iota\leq
        \iota_0$, $\boldsymbol{K}^{[\leq \iota]}$ has ${\iota +
        d}\choose \iota$ nonzero eigenvalues, all of them larger than
        $C d^{-\iota}$ and the range of $\boldsymbol{K}^{[\leq
        \iota]}$ is the span of
        $$\left\{(p(\bx_1),\ldots,p(\bx_n)): p~~ \text{multivariable polynomial of degree at most } \iota \right\}.$$
	\end{theorem}

	The component $\boldsymbol{K}^{[\leq \iota]}$ of the kernel matrix sketches the low-frequency component of the signal in much the same way as the corresponding $\bX_{\leq k}\bX_{\leq k}^\tr$ in linear regression sketches the top $k$ directions of the population distribution (see Section~\ref{sec:lin_regression}).
	
Let us explain the key ideas behind the proof of Theorem~\ref{thm:structure_eigenvalues_polynomials}.
	In correspondence with the sample covariance operator $n^{-1}\bX_{\leq k}^\tr \bX_{\leq k}$ in the linear case, we define the sample covariance operator 
	$\Theta^{[\leq \iota]} := n^{-1} \Phi^\top \Phi.$ 
	If the monomials $p_{\br}(\bx)$ were orthogonal in
    $L^2(\P)$,
    then we would have:
		\begin{align*}
			\mathbb{E}\left[ \Theta^{[\leq \iota]} \right] = {\rm diag}(C(0),~ \cdots, ~C(\iota')d^{-\iota'} ,~  \cdots,~ \underbrace{C(\iota) d^{-\iota}}_{\binom{d+\iota-1}{d-1}~\text{such entries}} )
		\end{align*}
	where $C(\iota)$  denotes constants that depend on $\iota$. Since under our general assumptions on the distribution this orthogonality does not necessarily hold, we employ the Gram-Schmidt process on the basis $\{1,t,t^2,\ldots\}$ with respect to $L^2(p)$ to produce an orthogonal polynomial basis $q_0,q_1,\ldots$. This yields new features
	\begin{align*}
		\Psi_{i,\br} = \left(c_{\br} \alpha_{\norm{\br}}\right)^{1/2} q_{\br}(\bx_i) / d^{\norm{\br}/2},~~~ q_{\br}(\bx)=\prod_{j \in [d]} q_{r_j}(\bx[j]).
	\end{align*}
	As shown in \cite{liang2020multiple}, these features are weakly dependent and the orthogonalization process does not distort the eigenvalues of the covariance matrix by more than a multiplicative constant. A small-ball method \cite{koltchinskii2015bounding} can then be used to prove the lower bound for the eigenvalues of $\Psi\Psi^\tr$ and thus establish Theorem~\ref{thm:structure_eigenvalues_polynomials}.

We now turn to variance and bias calculations. The analogue of \eqref{eq:var_linear} becomes
\begin{align}
	\label{eq:var_upper_kernel}
	\varianceemp \leq \sigma^2_\xi \cdot \En_\bx\norm{K(\bX,\bX)^{-1} K(\bX,\bx)}^2_2 
\end{align}
and, similarly to \eqref{eq:splitting_two_parts}, we split the kernel matrix into two parts, according to the degree $\iota$.

The following theorem establishes an upper bound on \eqref{eq:var_upper_kernel} \cite{liang2020multiple}:
\begin{theorem}
	\label{thm:var_multiple_descent}
	Under Assumption~\ref{asmpt:polynomials} and the additional
    assumption of sub-Gaussianity of the distribution $p$ for the coordinates of $\bx$,
if $\alpha_1,\ldots,\alpha_\iota>0$, there exists $\iota'\geq 2\iota+3$
with $\alpha_{\iota'}>0$, and $d^\iota\log d \lesssim n\lesssim
d^{\iota+1}$, then with probability at least $1-\exp^{-\Omega(n/d^\iota)}$,
	\begin{align}
		\label{eq:multiple_descent}
		\varianceemp \lesssim \sigma_\xi^2 \cdot \left(\frac{d^\iota}{n} + \frac{n}{d^{\iota+1}}\right).
	\end{align}
\end{theorem}

Notice that the behavior of the upper bound changes as $n$ increases
from $d^\iota$ to $d^{\iota+1}$. At $d\asymp n^\iota$, variance is
large since there is not enough data to reliably estimate all the
$d^\iota$ directions in the feature space. As $n$ increases, variance
in the first $d^\iota$ directions decreases; new directions in
the data appear (those corresponding to monomials of degree $\iota+1$,
with smaller population eigenvalues) but cannot be reliably estimated.
This second part of \eqref{eq:multiple_descent} grows linearly with
$n$, similarly to the second term in \eqref{eq:var_linear_condition}.
The split between these two terms occurs at the effective dimension
defined in Section~\ref{sec:lin_regression}.

Two aspects of the \emph{multiple-descent} behavior of the upper bound
\eqref{eq:multiple_descent} should be noted. First, variance is small
when $d^\iota \ll n \ll d^{\iota+1}$, between the peaks; second, the
valleys become deeper as $d$ becomes larger, with variance at most $d^{-1/2}$ at $n=d^{\iota+1/2}$. 

We complete the discussion of this section by exhibiting one possible upper bound on the bias term \cite{liang2020multiple}:
\begin{theorem}
	Assume the regression function can be written as
    $$f^*(\bx)=\int
    k(\bx,\bz)\rho_*(\bz)\P(d\bz) ~~\text{with}~~ \int
    \rho_*^4(\bz)\P(d\bz)\leq c.$$ Let Assumption~\ref{asmpt:polynomials} hold, and suppose $\sup_\bx k(\bx,\bx)\lesssim 1$. Then
	\begin{align}
		\biassquaredemp \lesssim \delta^{-1/2} \left( \En_\bx\norm{K(\bX,\bX)^{-1}K(\bX,\bx)}^2_2 + \frac{1}{n}\right)
	\end{align} 
	with probability at least $1-\delta$. The above expectation is precisely $\varianceemp/\sigma^2_\xi$ and can be bounded as in Theorem~\ref{thm:var_multiple_descent}.
\end{theorem}

\input{proportional.tex}

To conclude this section, we summarize the insights gained from
the analyses of several models in the interpolation regime.
First, in all cases, the interpolating solution $\algo$ can
be decomposed into a prediction (or simple) component and an
overfitting (or spiky) component. The latter ensures interpolation
without hurting prediction accuracy.  In the next section,
we show, under appropriate conditions on the parameterization
and the initialization, that gradient methods can be accurately
approximated by their linearization, and hence can be viewed as
converging to a minimum-norm linear interpolating solution despite
their non-convexity.
In Section~\ref{sec:NTK}, we return to the question of generalization,
focusing specifically on two-layer neural networks in linear regimes.

%% file: proportional.tex
\subsubsection{Kernels on $\Re^d$ with $d\asymp n$}
\label{sec:dpropn}

We now turn our attention to the regime  $d\asymp n$ and investigate
the behavior of minimum norm interpolants in the RKHS in this
high-dimensional setting. Random kernel matrices in the $d\asymp n$
regime have been extensively studied in the last ten years. As shown
in \cite{el2010spectrum}, under assumptions specified below,
the kernel matrix can be approximated in operator norm by
$$K(\bX,\bX) \approx c_1\frac{\bX\bX^\tr}{d} + c_2 \bI_n,$$
that is, a linear kernel plus a scaling of the identity. While this
equivalence can be viewed as a negative result about the utility of
kernels in the $d\asymp n$ regime, the term $c_2 \bI_n$ provides
implicit regularization for the minimum-norm interpolant in the RKHS
\cite{liang2020just}.

We make the following assumptions.
\begin{assumption}
	\label{asmpt:kernel_nd_regime}
        We assume that coordinates of $\bz=\bSigma^{-1/2}\bx$ are
        independent, with zero mean and unit variance, so that
        $\bSigma=\En \bx\bx^\tr$.
        Further assume there are constants
        $0<\eta,M<\infty$, such that the following hold.
        \begin{enumerate}
        \item[$(a)$] For all $i\le d$,
          $\E[|\bz_i|^{8+\eta}]\le M$.
        \item[$(b)$] $\|\bSigma\| \le M$, $d^{-1}\sum_{i=1}^d\lambda_i^{-1}\le M$, where $\lambda_1,\dots,\lambda_d$ are the eigenvalues of $\bSigma$.
        \end{enumerate}
\end{assumption}
	
Note that, for $i\neq j$, the rescaled scalar products $\inner{\bx_i,\bx_j}/d$ are typically of order $1/\sqrt{d}$.  We can therefore approximate the kernel function by its Taylor expansion around $0$. To this end, define
\begin{align*}
  \alpha &:= h(0)+ h''(0)\frac{\trace(\bSigma^2)}{2d^2},~ \beta := h'(0),\\
  \gamma &:= \frac{1}{h'(0)}\big[h(\trace(\bSigma)/d) - h(0)-h'(0)\trace(\bSigma/d)\big].
\end{align*}

Under Assumption~\ref{asmpt:kernel_nd_regime}, a variant of a result of \cite{el2010spectrum} implies that for  some $c_0\in (0,1/2)$, the following holds with high probability 
\begin{align}
	\norm{K(\bX,\bX)-K^{\text{lin}}(\bX,\bX)} \lesssim d^{-c_0} \, 
\end{align}
where
\begin{align}
	K^{\text{lin}}(\bX,\bX) = \beta \frac{\bX\bX^\tr}{d} + \beta \gamma \bI_n + \alpha \bone\bone^\tr.
\end{align}
To make the self-induced regularization due to the ridge apparent, we
develop an upper bound on the variance of the minimum-norm interpolant in \eqref{eq:var_upper_kernel}. Up to an additive diminishing factor, this expression can be replaced by 
\begin{align}
	\sigma_\xi^2 \cdot \trace \left((\bX\bX^\tr + d\gamma \bI_n)^{-2}\bX\bSigma\bX^\tr \right),
\end{align}
where we assumed without loss of generality that $\alpha=0$. Comparing to \eqref{eq:lin_var_decomp2},
we observe that here implicit regularization arises due to the `curvature' of the kernel, in addition to any favorable tail behavior in the spectrum of $\bX\bX^\tr$. Furthermore, this regularization arises under rather weak assumptions on the random variables even if Assumption~\ref{assmpt:linear_reg} is not satisfied. A variant of the development in \cite{liang2020just} yields a more interpretable upper bound of
\begin{align}
	\varianceemp \lesssim \sigma^2_\xi \cdot \frac{1}{\gamma} \left(\frac{k}{n} + \lambda_{k+1}\right)
\end{align}
for any $k\geq 1$ \cite{liang2020amend}; the proof is in the Supplementary
Material. Furthermore, a high probability bound on the bias
	\begin{align}
		\label{eq:bias_kernel_ridgeless}
		\biassquaredemp \lesssim \norm{f^*}^2_{\cH} \cdot \inf_{0\leq k\leq n} \left\{ \frac{1}{n}\sum_{j>k} \lambda_j(\frac{1}{d}\bX\bX^\tr) + \gamma + \sqrt{\frac{k}{n}}\right\}
	\end{align}
can be established with basic tools from empirical process theory under boundedness assumptions on $\sup_{\bx} k(\bx,\bx)$ \cite{liang2020just}. 

With more recent developments on the bias and variance of linear
interpolants in \cite{hastie2019surprises}, a significantly more
precise statement can be derived for the $d\asymp n$ regime.
The proof of the following theorem is in the Supplementary
Material.
\begin{theorem}
  \label{thm:npropd}
  Let $0< M,\eta<\infty$ be fixed constants and suppose that
  Assumption~\ref{asmpt:kernel_nd_regime} holds with $M^{-1}\le d/n\le
  M$. Further assume that  $h$ is continuous on $\reals$ and smooth in a
  neighborhood of $0$ with $h(0), h'(0)>0$,
  that $\|f^*\|_{L^{4+\eta}(\P)}\le M$ and that the $z_i$ are $M$-subgaussian.
  Let $y_i=f^*(\bx_i)+\xi_i$, $\E(\xi_i^2)= \sigma_{\xi}^2$,
  and $\bbeta_0 := \bSigma^{-1}\E[\bx f^*(\bx)]$.
  Let $\lambda_*>0$ be the unique positive solution of
        \begin{align}
          n \Big(1-\frac{\gamma}{\lambda_*}\Big) = \Trace\Big( \bSigma(\bSigma+\lambda_*\id)^{-1}\Big)\, .\label{eq:FixedPointProportional}
        \end{align}
        Define $\cuB(\bSigma,\bbeta_0)$ and $\cuV(\bSigma)$ by
        \begin{align}
          \cuV(\bSigma) &:=\frac{\Trace\big( \bSigma^2(\bSigma+\lambda_*\id)^{-2}\big)}{n-\Trace\big( \bSigma^2(\bSigma+\lambda_*\id)^{-2}\big)}
                          \, , \label{eq:VarianceProportional}\\
          \cuB(\bSigma,\bbeta_0) &:= \frac{\lambda_*^2\<\bbeta_0,(\bSigma+\lambda_*\id)^{-2}\bSigma\bbeta_0\>}
           {1-n^{-1} \Trace\big(
           \bSigma^2(\bSigma+\lambda_*\id)^{-2}\big)}\,
           .\label{eq:BiasProportional}
        \end{align}
        Finally, let $\biassquaredemp$ and  $\varianceemp$ denote the
        squared bias and variance for the minimum-norm interpolant \eqref{eq:kernel_min_norm_interpolant}. 
	Then there exist  $C,c_0>0$  (depending also on the constants in Assumption~\ref{asmpt:kernel_nd_regime})
        such that the following holds with probability at least $1-Cn^{-1/4}$
        (here $\proj_{>1}$ denotes the projector orthogonal to affine
        functions in $L^2(\P)$):
	\begin{align}
          \big|\biassquaredemp -
          \cuB(\bSigma,\bbeta_0)-\|\proj_{>1}f^*\|_{L^2}^2
           (1+\cuV(\Sigma)) \big|&\le C n^{-c_0}\, , \label{eq:bias_kernel_linear_precise}\\
           \big|	\varianceemp-  \sigma_{\xi}^2\cuV(\bSigma) \big|&\le C n^{-c_0}\, .\label{eq:var_kernel_linear_precise}
	\end{align}
      \end{theorem}
	  
A few remarks are in order. First, note that the left hand side of \eqref{eq:FixedPointProportional} is strictly increasing in $\lambda_*$, while the right
hand side is strictly decreasing. By considering the limits as $\lambda_*\to 0$ and $\lambda_*\to\infty$, it is easy to
see that this equation indeed admits a unique solution.
Second, the bias estimate in \eqref{eq:bias_kernel_ridgeless} requires
$f^*\in\cH$, while the bias calculation in
\eqref{eq:bias_kernel_linear_precise} does not make this assumption, but instead incurs an approximation error for non-linear components of $f^*$.

We now remark that the minimum-norm interpolant with kernel $K^{\text{lin}}$ is simply ridge regression with respect to the plain covariates $\bX$ and ridge penalty proportional to $\gamma$:
\begin{align}
  (\widehat{\theta_0},\widehat{\btheta}) &:=\argmin{\theta_0,\btheta}
  \frac{1}{d}\big\|\by-\theta_0-\bX\btheta\big\|_2^2+\gamma \|\btheta\|^2_2 \, .\label{eq:VanillaRidge}
\end{align}
The intuition is that the minimum-norm interpolant for the original kernel takes the form $\algo(\bx)  = \widehat{\theta}_0+\<\widehat{\btheta},\bx\> +\Delta(\bx)$.
Here $\widehat{\theta}_0+\<\widehat{\btheta},\bx\>$ is a simple
component, and $\Delta(\bx)$ is an overfitting component: a function
that is small in $L^2(\P)$ but allows interpolation of the data.

The characterization in \eqref{eq:FixedPointProportional},
\eqref{eq:VarianceProportional}, and \eqref{eq:BiasProportional}
can be shown to imply upper bounds that are related to the analysis in
Section~\ref{sec:lin_regression}.
\begin{corollary}
  Under the assumptions of Theorem \ref{thm:npropd}, further assume
   that $f^*(\bx) = \<\bbeta_0,\bx\>$ is linear and
  that there is an integer $k\in\naturals$, and
  a constant $c_*>0$  such that  $r_k(\bSigma)+(n\gamma/c_*\lambda_{k+1})\ge  (1+c_*)n$.
   Then there exists $c_0\in(0,1/2)$ such that, with high probability,
  the following hold as long as the right-hand side is less than one:
  \begin{align}
    \biassquaredemp  &\le
    4\Big(\gamma+\frac{1}{n}\sum_{i=k+1}^d\lambda_i\Big)^2
    \|\bbeta_{0,\le
    k}\|_{\bSigma^{-1}}^2+\|\bbeta_{0,>k}\|_{\bSigma}^2+n^{-c_0}\,
    ,\label{eq:BiasProportional2}\\
    \varianceemp& \le \frac{2k\sigma_\xi^2}{n} +
    \frac{4n \sigma_\xi^2}{c_*}
                  \frac{\sum_{i=k+1}^d \lambda_i^2}{(n\gamma/c_*+\sum_{i=k+1}^{d}\lambda_i)^2}+n^{-c_0}\, . \label{eq:VarProportional}
  \end{align}
  Further, under the same assumptions, the effective regularization $\lambda_*$
  (that is, the unique solution of \eqref{eq:FixedPointProportional}), satisfies
\begin{align}
  \gamma+\frac{c_*}{1+c_*}\frac{1}{n}\sum_{i=k+1}^{d}\lambda_i \le \lambda_*\le
  2\gamma+\frac{2}{n}\sum_{i=k+1}^d\lambda_i\, .
\end{align}
\end{corollary}
Note that apart from the $n^{-c_0}$ term,
\eqref{eq:BiasProportional2}
recovers the result of Theorem \ref{thm:bias_split_linear},
while \eqref{eq:VarProportional} recovers Theorem \ref{thm:var_linear}
(setting $\gamma=0$), both with improved constants
but  limited to the proportional regime.
We remark that analogues of Theorems~\ref{thm:var_linear}, \ref{thm:bias_split_linear}, and \ref{thm:lower_linear} for ridge regression with $\gamma\neq 0$ can be found in \cite{tsigler2020benign}.

The formulas \eqref{eq:FixedPointProportional},
\eqref{eq:VarianceProportional}, and \eqref{eq:BiasProportional}
might seem somewhat mysterious. However, they have an appealing interpretation in
terms of a simpler model that we will refer to as a `sequence model' (this terminology comes from classical
statistical estimation theory \cite{johnstone2019book}).
As stated precisely in the remark below, the sequence model
is a linear regression model in which the design matrix is deterministic (and diagonal),
and the noise and regularization levels are  determined via a fixed point equation. 
\begin{remark}
Assume without loss of generality $\bSigma = \diag(\lambda_1,\dots,\lambda_d)$.
In the sequence model we observe $\by^{\seq}\in\reals^d$ distributed according to
\begin{align}
  \by_i^{\seq} & = \lambda_i^{1/2}\beta_{0,i}+\frac{\tau}{\sqrt{n}}g_i \, ,\;\;\;   (g_i)_{i\le d}\sim_{iid}\normal(0,1)\, ,
\end{align}
where $\tau$ is a parameter given below. We then perform ridge regression with regularization $\lambda_*$:
\begin{align}
  \hbbeta^{\seq}(\lambda_*)& :=\argmin{\bbeta}  \big\|\by^{\seq}-\bSigma^{1/2}\bbeta\big\|_2^2+\lambda_*\|\bbeta\|^2_2\, ,
\end{align}
which can be written in closed form as
\begin{align} 
  \hbeta^{\seq}_i(\lambda_*)& = \frac{\lambda_i^{1/2} y^{\seq}_i}{\lambda_*+\lambda_i}\, .
\end{align}
The noise level $\tau^2$ is then fixed via the condition
$\tau^2 = \sigma_\xi^2+\E \|
\hbbeta^{\seq}(\lambda_*)-\bbeta_0\|_2^2$. Then under the assumption
that $f^*$ is linear, Theorem \ref{thm:npropd} states that
\begin{align}
  \E\{(f^*(\bx)-\hf(\bx))^2|\bX\} =  \E \|  \hbbeta^{\seq}(\lambda_*)-\bbeta_0\|_2^2 +O(n^{-c_0})
\end{align}
with high probability.
\end{remark}

%% file: efficient.tex
\section{Efficient optimization}
\label{sec:efficient}

The empirical risk minimization (ERM) problem is, in general, intractable even in simple cases.
Section \ref{sec:Intractability} gives examples of such hardness results.
The classical approach to address this conundrum is to construct convex surrogates of the non-convex
ERM problem. The problem of learning a linear classifier provides an
easy-to-state---and yet subtle---example.
Considering the $0$-$1$ loss, ERM reads
\begin{align}
  \mbox{minimize}\;\;\;\; \hRisk_{01}(\btheta) :=\frac{1}{n}\sum_{i=1}^n\indicator{y_i\<\btheta,\bx_i\> \le 0}\, .\label{eq:ERM01}
\end{align}
Note however that the original problem \eqref{eq:ERM01} is not always intractable. If there exists $\btheta\in\reals^p$
such that $\hRisk(\btheta)=0$, then finding $\btheta$  amounts to
solving a set of $n$ linear inequalities. This  can be
done in polynomial time.  In other words, when the model is sufficiently rich to interpolate the data, an interpolator can be
constructed efficiently.

In the case of linear classifiers, tractability arises because of the specific structure of the function
class (which is linear in the parameters $\btheta$), but one might wonder whether it is instead a more general phenomenon.
The problem of finding an interpolator can be phrased as a constraint optimization problem. Write the empirical
risk as 
$$\hRisk(\btheta) = \frac{1}{n}\sum_{i=1}^n\ell(\btheta;y_i,\bx_i).$$  Then we are seeking $\btheta\in \Theta$ such that
\begin{align}
  \ell(\btheta;y_i,\bx_i) = 0\;\;\; \text{for all}~ i\le n\, .
\end{align}
Random constraint satisfaction problems have been studied in depth over the last twenty years, although under different distributions
from those arising from neural network theory.
Nevertheless, a recurring observation is that, when the number of free parameters is sufficiently
large compared to the number of constraints, these problems (which
are NP-hard in the worst case) become tractable;
see, for example, \cite{frieze1996analysis,achlioptas1997analysis} and \cite{coja2010better}.

These remarks motivate a fascinating working hypothesis:
 modern neural networks are tractable \emph{because} they are overparametrized.

Unfortunately, a satisfactory theory of this phenomenon is still lacking, with an important exception:
the linear regime. This is a training regime in which the network can be approximated by a linear model, with a random featurization map
associated with the training initialization. We discuss these results in Section~\ref{sec:LinearRegime}.

While the linear theory can explain a number of phenomena observed in practical neural networks, it also
misses some important properties. We will discuss these points, and results beyond the linear regime, in Section~\ref{sec:BeyondLinear}.

\subsection{The linear regime}
\label{sec:LinearRegime}

Consider a neural network with parameters $\btheta\in\reals^p$: for an input $\bx\in\reals^d$
the network outputs $f(\bx;\btheta)\in\reals$.
We consider training using the square loss
\begin{align}
  \hRisk(\btheta) := \frac{1}{2n}\sum_{i=1}^n
  \big(y_i-f(\bx_i;\btheta)\big)^2
  = \frac{1}{2n}\big\|\by-f_n(\btheta)\big\|_{2}^2\, .
\end{align}
%
Here $\by=(y_1,\ldots,y_n)$ and $f_n:\reals^p\to\reals^n$ maps the
parameter vector $\btheta$ to the evaluation of $f$ at the $n$ data
points, $f_n:\btheta\mapsto (f(\bx_1;\btheta),\dots,f(\bx_n;\btheta))$.
We minimize this empirical risk using gradient flow, with initialization $\btheta_0$:
\begin{align}
  \frac{\de \btheta_t}{\de t} = \frac{1}{n}\bD f_n(\btheta_t)^{\sT}(\by-f_n(\btheta_t))\, .\label{eq:GradFlow}
\end{align}
Here $\bD f_n(\btheta)\in\reals^{n\times p}$ is the Jacobian matrix of the map
$f_n$.
Our focus on the square loss and continuous time is for simplicity of exposition.
Results of the type presented below have been proved for more general loss functions
and for
discrete-time and stochastic gradient methods.

As first argued in \cite{jacot2018neural}, in a highly overparametrized regime it can happen that  $\btheta$ changes only slightly with
respect to the initialization $\btheta_0$. This suggests comparing the original gradient flow with the one
obtained by linearizing the right-hand side of \eqref{eq:GradFlow} around the initialization $\btheta_0$:
\begin{align}
  \frac{\de \obtheta_t}{\de t} = \frac{1}{n}\bD f_n(\btheta_0)^{\sT}\big(\by-f_n(\btheta_0)-\bD  f_n(\btheta_0)(\obtheta_t-\btheta_0)\big)\, .
  \label{eq:LinearizedEvolution}
\end{align}
More precisely, this is  the gradient flow for the risk function
\begin{align} 
  \hRisk_{\slin}(\obtheta) := \frac{1}{2n}\| \by-f_n(\btheta_0)-\bD
  f_n(\btheta_0) (\obtheta-\btheta_0)\|_{2}^2,
  \end{align}
which is obtained by replacing $f_n(\btheta)$ with its first-order Taylor expansion at $\btheta_0$. 
  Of course,   $\hRisk_{\slin}(\obtheta)$ is quadratic in $\obtheta$. 
  In particular, if the Jacobian $\bD f_n(\btheta_0)$ has full row rank, the set of global minimizers $\ERM_0:=\{\obtheta: \;  \hRisk_{\slin}(\obtheta) = 0 \}$ forms an affine space of dimension $p-n$. 
In this case, gradient flow converges to $\obtheta_{\infty}\in\ERM_0$,
which---as discussed in Section~\ref{sec:implicit}---minimizes the $\ell_2$ distance from the initialization:
\begin{align}
  \obtheta_{\infty}:= \argmin{}\Big\{\|\obtheta-\btheta_0\|_2:\;\;\;
  \bD f_n (\btheta_0) (\obtheta-\btheta_0) = \by-f_n(\btheta_0)\Big\}\, . \label{eq:MinNormTheta}
\end{align}

The linear (or `lazy' ) regime is a training regime in which $\btheta_t$ is well approximated by $\obtheta_t$ at all times. 
Of course if $f_n(\btheta)$ is an affine function of $\btheta$, that is, if $\bD f_n(\btheta)$ is constant, then
we have $\btheta_t=\obtheta_t$ for all times $t$. It is therefore natural to quantify deviations
from linearity by defining  the Lipschitz constant
\begin{align}
  \Lip(\bD f_n) := \sup_{\btheta_1\neq\btheta_2}\frac{\|\bD f_n(\btheta_1)-\bD f_n(\btheta_2)\|_{\op}}{\|\btheta_1-\btheta_2\|_2}\, .
\end{align}
(For a matrix $\bA\in \reals^{n\times p}$, we define
$\|\bA\|_{\op}:=\sup_{\bx\neq 0}\|\bA\bx\|_2/\|\bx\|_2$.)
It is also useful to define a population version of the last quantity.
For this, we assume as usual that samples are i.i.d. draws
$(\bx_i)_{i\le n}\sim_{iid} \P$,
and with a slight abuse of notation, we
view $f:\btheta\mapsto f(\btheta)$ as a map from $\reals^p$ to
$L^2(\P):=L^2(\reals^d;\P)$.
We let $\bD f(\btheta)$ denote the differential of this map at $\btheta$,
which is a linear operator, $\bD f(\btheta):\reals^p\to
L^2(\P)$. The corresponding operator
norm and Lipschitz constant are given by
\begin{align}
  \|\bD f(\btheta)\|_{\op}& :=
  \sup_{\bv\in\reals^p\setminus\{0\}}\frac{\|\bD
  f(\btheta)
  \bv\|_{L^2(\P)}}{\|\bv\|_2}\, ,\\
  \Lip(\bD f) &:= \sup_{\btheta_1\neq\btheta_2}\frac{\|\bD f(\btheta_1) -\bD f(\btheta_2)\|_{\op}}{\|\btheta_1-\btheta_2\|_2}\, .
\end{align}

The next theorem establishes sufficient conditions for $\btheta_t$ to
remain in the linear regime in terms of the singular values
and Lipschitz constant of the Jacobian. Statements of this type were proved in several papers, starting with
\cite{du2018gradient}; see, for example,  \cite{allen2019convergence,du2019gradient,zou2020gradient,oymak2020towards} and \cite{liu2020linearity}.
We follow the abstract point of view  in \cite{oymak2019overparameterized} and \cite{chizat2019lazy}. 
\begin{theorem}\label{thm:Linearization}
    Assume
  \begin{align}
    \Lip(\bD f_n)\, \|\by-f_n(\btheta_0)\|_{2} < \frac{1}{4} \sigma^2_{\min}(\bD f_n(\btheta_0))\, .
    \label{eq:ConditionLinearization}
  \end{align}
  Further define 
  $$\sigma_{\max}:=\sigma_{\max}(\bD f_n(\btheta_0)), \sigma_{\min}:=\sigma_{\min}(\bD  f_n(\btheta_0)). $$
  Then the following hold for all $t>0$:
  \begin{enumerate}
  \item The empirical risk decreases exponentially fast to $0$, with  rate $\lambda_0 = \sigma^2_{\min}/(2n)$:
    \begin{align}
      \hRisk(\btheta_t)\le \hRisk(\btheta_0) \, e^{-\lambda_0 t}\, .
      \label{eq:R-Linearization}
    \end{align}
  \item The parameters stay close to the initialization and are closely tracked by those of the linearized flow.
    Specifically, letting $L_n:=\Lip(\bD f_n)$,
  \begin{align}
    \|\btheta_t-\btheta_0\|_2& \le \frac{2}{\sigma_{\min}}\, \|\by-f_n(\btheta_0)\|_2\, ,\label{eq:CloseToInit}\\
    \|\btheta _t-\obtheta_t\|_2 &\le
      \Big\{\frac{32\sigma_{\max}}{\sigma^2_{\min}}
      \|\by-f_n(\btheta_0)\|_{2}+ \frac{16L_n}{\sigma^3_{\min}}
      \|\by-f_n(\btheta_0)\|^2_{2}\Big\} \notag\\*
    & \qquad\qquad {}
      \wedge
      \frac{180L_n\sigma_{\max}^2}{\sigma^5_{\min}}
      \|\by-f_n(\btheta_0)\|^2_{2}
\,.                                  \label{eq:Coupling}
  \end{align}
\item The models constructed by gradient flow and by the linearized flow are similar on test data.
  Specifically, writing $f^{\slin}(\btheta) = f(\btheta_0)+\bD f(\btheta_0) (\btheta-\btheta_0)$, we have  
  \begin{align}
      \lefteqn{\|f(\btheta_t)
      -f^{\slin}(\obtheta_t)\|_{L^2(\P)}}
      & \notag\\*
      & \le \Big\{4\, \Lip(\bD f)\frac{1}{\sigma_{\min}^2}+
    180\|\bD
    f(\btheta_0)\|_{\op}\frac{L_n\sigma_{\max}^2}{\sigma_{\min}^5}\Big\}\|\by-f_n(\btheta_0)\|_2^2\, .  \label{eq:ModelDeviationLinearized}
  \end{align}
\end{enumerate}
\end{theorem}
The bounds in \eqref{eq:R-Linearization} and \eqref{eq:CloseToInit} follow from the main result of \cite{oymak2019overparameterized}.
The coupling bounds in \eqref{eq:Coupling} and
\eqref{eq:ModelDeviationLinearized} are 
proved in the Supplementary Material.

A key role in this theorem is played by the singular values of the
Jacobian at initialization, $\bD f_n(\btheta_0)$. These can also be encoded in the kernel matrix $\bK_{m,0}:=
\bD f_n(\btheta_0) \bD f_n(\btheta_0)^{\sT}\in\reals^{n\times n}$. The importance of this matrix can be easily understood by writing the evolution
of the predicted values $f^{\slin}_{n}(\obtheta_t):= f_n(\btheta_0)+\bD  f_n(\btheta_0)(\obtheta_t-\btheta_0)$. Equation  \eqref{eq:LinearizedEvolution}
implies
\begin{align}
  \frac{\de f^{\slin}_{n}(\obtheta_t)}{\de t} = \frac{1}{n}\bK_{m,0}\big(\by-f^{\slin}_{n}(\obtheta_t)\big)\, .\label{eq:EvolutionPred}
\end{align}
Equivalently, the residuals $\br_t:=\by-f^{\slin}_{n}(\obtheta_t)$ are driven to zero according to
$(\de/\de t)\br_t=-\bK_{m,0}\br_t/n$.

Applying  Theorem \ref{thm:Linearization}  requires the evaluation of
the minimum and maximum singular values of the
Jacobian, as well as its Lipschitz constant. As an example, we consider the case of
two-layer neural networks:
\begin{align}
  f(\bx;\btheta) := \frac{\alpha}{\sqrt{m}}\sum_{j=1}^m b_j\sigma(\<\bw_j,\bx\>),~~
  \btheta= (\bw_1,\dots,\bw_m)\, .\label{eq:Two-Layers-Net}
\end{align}
To simplify our task, we assume the second layer weights $\bb =(b_1,\dots,b_m)\in \{+1,-1\}^m$
to be fixed with an equal number of $+1$s and $-1$s.
Without loss of generality we can assume  that
$b_1=\cdots=b_{m/2}=+1$ and $b_{m/2+1}=\cdots=b_{m}=-1$. 
We train the weights $\bw_1,\dots, \bw_m$ via gradient flow.
The number of parameters is $p=md$. The scaling factor $\alpha$ allows
tuning between different regimes.
We consider two initializations,
denoted by $\btheta_0^{(1)}$ and $\btheta_0^{(2)}$:
  \begin{align}
    \btheta_0^{(1)}:&&  (\bw_i)_{i\le m}&\sim_{i.i.d.} \Unif(\Sp^{d-1});\\
    \btheta_0^{(2)}:&&  (\bw_i)_{i\le m/2}&\sim_{i.i.d.}
    \Unif(\Sp^{d-1}),\, \bw_{m/2+i}=\bw_{i},\, i\le m/2,
  \end{align}
where $\Sp^{d-1}$ denotes the unit sphere in $d$ dimensions.
The important difference between these initializations is that (by the
central limit theorem)
$|f(\bx;\btheta_0^{(1)})| = \Theta(\alpha)$,
while $f(\bx;\btheta_0^{(2)}) = 0$.

It is easy to compute
the Jacobian $\bD f_n(\bx;\btheta)\in\reals^{n\times md}$:
\begin{align}
  [\bD f_n(\bx;\btheta)]_{i,(j,a)} = \frac{\alpha}{\sqrt{m}} b_j\sigma'(\<\bw_j,\bx_i\>)\, x_{ia}\, ,\;\;\;\;
  i\in [n], (j,a)\in [m]\times [d]\, .\label{eq:TwoLayersJacobian}
\end{align}

\begin{assumption}\label{ass:Sigma}
Let $\sigma:\reals\to\reals$ be a fixed activation function which we assume differentiable
  with bounded first and second order derivatives.
  Let $$\sigma = \sum_{\ell\ge 0}\mu_{\ell}(\sigma)h_{\ell}$$ denote its decomposition
  into orthonormal Hermite polynomials. Assume 
  $\mu_{\ell}(\sigma)\neq 0$ for all $\ell\le \ell_0$ for some constant $\ell_0$.
\end{assumption}
\begin{lemma}\label{lemma:TwoLayers}
  Under Assumption \ref{ass:Sigma},  further assume
 $\{(\bx_i,y_i)\}_{i\le n}$ to be i.i.d.\ with $\bx_i
 \sim_{i.i.d.}\normal(0,\id_d)$,  and  $y_i$ $B^2$-sub-Gaussian.
 Then there exist constants $C_i$,
  depending uniquely on $\sigma$, such that the following hold with probability at least $1-2\exp\{-n/C_0\}$,
  provided $md\ge C_0 n\log n$ and $n\le d^{\ell_0}$ (whenever not specified, these hold for both  initializations
  $\btheta_0\in\{\btheta_0^{(1)},\btheta_0^{(2)}\}$):
  \begin{align}
    \|\by-f_n(\btheta^{(1)}_0)\|_2 & \le C_1\big(B+\alpha)\sqrt{n}\, \label{eq:LemmaLin1}\\
                         \|\by-f_n(\btheta^{(2)}_0)\|_2 & \le C_1B\sqrt{n}\, , \label{eq:LemmaLin2}\\ 
    \sigma_{\min}(\bD f_n(\btheta_0)) &\ge C_2\alpha \sqrt{d}\, , \label{eq:LemmaLin3}\\
    \sigma_{\max}(\bD f_n(\btheta_0)) &\le C_3\alpha \big(\sqrt{n}+\sqrt{d}\big)\, , \label{eq:LemmaLin4}\\
    \Lip(\bD f_n) & \le C_4\alpha\sqrt{\frac{d}{m}}\big(\sqrt{n}+\sqrt{d}\big)\, . \label{eq:LemmaLin5}
  \end{align}
  Further
  \begin{align}
    \|\bD f(\btheta_0)\|_{\op}& \le C_1'\alpha\, , \label{eq:LemmaLin6}\\
    \Lip(\bD f) & \le    C_4' \alpha\sqrt{\frac{d}{m}}\, . \label{eq:LemmaLin7}
  \end{align}
  \end{lemma}
Equations \eqref{eq:LemmaLin1}, \eqref{eq:LemmaLin2} are
straightforward~\cite{oymak2019overparameterized}.
The remaining inequalities are proved in the
Supplementary Material.
  Using these estimates in Theorem \ref{thm:Linearization}, we get the
  following approximation theorem for two-layer neural nets.
  \begin{theorem}\label{thm:Two-Layers-Linear}
    Consider the two layer neural network of \eqref{eq:Two-Layers-Net}
  under the assumptions of Lemma~\ref{lemma:TwoLayers}.
  Further let $\oalpha :=\alpha/(1+\alpha)$ for initialization $\btheta_0=\btheta_0^{(1)}$ and
  $\oalpha :=\alpha$ for $\btheta_0=\btheta_0^{(2)}$.
  Then there exist constants $C_i$,
  depending uniquely on $\sigma$, such that
 if $md\ge C_0 n\log n$, $d\le n\le d^{\ell_0}$ and
  \begin{align}
    \oalpha\ge C_0 \sqrt{\frac{n^2}{md}}\,
    ,\label{eq:ConditionLinearization2}
  \end{align}
  then,
  with probability at least $1-2\exp\{-n/C_0\}$,
  the following hold for all $t\ge 0$.
  \begin{enumerate}
  \item Gradient flow converges exponentially fast to a global
  minimizer. Specifically, letting
    $\lambda_* = C_1\alpha^2d/n$, we have
    \begin{align}
      \hRisk(\btheta_t)\le \hRisk(\btheta_0) \, e^{-\lambda_* t}\, .\label{eq:TrainingConvergence}
    \end{align}
  \item The model constructed by gradient flow and linearized flow are similar on test data,
    namely
    \begin{align}
      \|f(\btheta_t) -f_{\slin}(\obtheta_t)\|_{L^2(\P)}\le C_1\left\{\frac{\alpha}{\oalpha^2}\sqrt{\frac{n^2}{md}} +\frac{1}{\oalpha^2}
      \sqrt{\frac{n^5}{md^4}}\right\}\, .\label{eq:ModelDistanceLin}
      \end{align}
  \end{enumerate}
\end{theorem}

It is instructive to consider Theorem \ref{thm:Two-Layers-Linear} for two different
choices of $\alpha$ (a third one will be considered in
Section~\ref{sec:BeyondLinear}).

For $\alpha =\Theta(1)$, we have $\oalpha =\Theta(1)$  and therefore the two initializations $\{\btheta^{(1)},\btheta_0^{(2)}\}$
behave similarly. In particular, condition \eqref{eq:ConditionLinearization2}  requires $m d\gg n^2$: the number of network parameters must be
quadratic in the sample size. This is significantly stronger than the
simple condition that the network is overparametrized, namely $md\gg n$. 
Under the condition $m d\gg n^2$ we have exponential convergence to vanishing training error, and  the difference between the neural network
and its linearization   is bounded as in \eqref{eq:ModelDistanceLin}.  This bound vanishes for  $m\gg n^5/d^4$. While we
do not expect this condition to be tight, it implies that, \emph{under the choice $\alpha=\Theta(1)$}, sufficiently wide networks behave as linearly
parametrized  models.

For $\alpha \to\infty$, we have $\oalpha\to 1$ for initialization
$\btheta_0^{(1)}$  and therefore Theorem~\ref{thm:Two-Layers-Linear} yields the same
bounds as in the previous paragraph for this initialization. However, for the  initialization $\btheta_0=\btheta_0^{(2)}$
(which is constructed so that $f(\btheta^{(2)}_0) =0$)  we have $\oalpha = \alpha$ and condition \eqref{eq:ConditionLinearization2} is always verified
as $\alpha\to \infty$.
Therefore the conclusions of Theorem~\ref{thm:Two-Layers-Linear} apply under nearly minimal overparametrization,
namely if $md\gg n\log n$. In that case,
the linear model is an arbitrarily good approximation of
the neural net as $\alpha$ grows:  $\|f(\btheta_t)
-f_{\slin}(\obtheta_t)\|_{L^2(\P)}= O(1/\alpha)$.
In other words, an overparametrized neural network can be trained in
the linearized regime by choosing suitable initializations and
suitable scaling of the parameters.

Recall   that, as $t\to\infty$,  $\obtheta_t$ converges to the
min-norm interpolant $\obtheta_{\infty}$; see \eqref{eq:MinNormTheta}.
Therefore, as long as condition \eqref{eq:ConditionLinearization2}  holds and the right-hand side of \eqref{eq:ModelDistanceLin} is negligible,
the generalization properties of the neural network are well approximated by those of min-norm interpolation in a linear model
with featurization map $\bx\mapsto \bD f(\bx;\btheta_0)$. We will study the latter in Section~\ref{sec:NTK}.

In the next subsection we will see that the linear theory outlined here fails to
capture different training schemes in which the network weights genuinely change.

\subsection{Beyond the linear regime?}
\label{sec:BeyondLinear}

For a given dimension $d$ and sample size $n$, we can distinguish two ways to violate the
conditions for the linear regime, as stated for instance in Theorem~\ref{thm:Two-Layers-Linear}.
First, we can reduce the network size $m$. While Theorem \ref{thm:Two-Layers-Linear} does not specify
the minimum $m$ under which the conclusions of the theorem cease to hold,
it is clear that $md\ge n$ is necessary in order for the training
error to vanish as in \eqref{eq:TrainingConvergence}.

However, even if the model is overparametrized, the same condition is violated if $\alpha$ is sufficiently
small. In particular, the limit $m\to\infty$ with $\alpha= \alpha_0/\sqrt{m}$ has attracted considerable attention and is known as the mean field limit.
In order to motivate the mean field analysis, we can suggestively rewrite \eqref{eq:Two-Layers-Net} as
\begin{align}
  f(\bx;\btheta) := \alpha_0\int \! b\,\sigma(\<\bw,\bx\>)\, \hrho(\de\bw,\de b) \, ,
\end{align}
where $\hrho:= m^{-1}\sum_{j=1}^{m}\delta_{\bw_j,b_j}$ is the
empirical distribution
of neuron weights. If the weights are drawn i.i.d.\ from a common distribution $(\bw_j,b_j)\sim\rho$,
we can asymptotically replace $\hrho$ with $\rho$ in the above expression, by the law of large numbers.

The gradient flow \eqref{eq:GradFlow} defines an evolution over the
space of neuron weights,  and hence an evolution in the space of empirical
distributions $\hrho$. It is natural to ask whether this evolution admits a simple characterization. This question was first addressed
by \cite{nitanda2017stochastic,mei2018mean,rotskoff2018neural,sirignano2020mean} and \cite{chizat2018global}. 
\begin{theorem}\label{thm:PDE}
  Initialize the weights so that $\{(\bw_j,b_j)\}_{j\le m}\sim_{i.i.d.} \rho_0$ with $\rho_0$ a probability measure on $\reals^{d+1}$.
  Further, assume the activation function $u\mapsto \sigma(u)$ to be differentiable with $\sigma'$ bounded and Lipschitz continuous,
  and assume $|b_j|\le C$ almost surely under the initialization $\rho_0$, for some constant $C$.
 Then, for any fixed $T\ge 0$, the
  following limit holds  in $L^2(\P)$, uniformly over $t\in [0,T]$:
  \begin{align}
\lim_{m\to\infty}f(\btheta_{mt}) = F(\rho_t) :=\alpha_0\int \! b\,\sigma(\<\bw,\, \cdot\,\>)\, \rho_t(\de\bw,\de b)\, ,\label{eq:MF-Limit}
  \end{align}
  where $\rho_t$ is a probability measure on $\reals^{d+1}$
  that solves the following partial differential equation (to be
  interpreted in the weak sense):
  \begin{align}
    \partial_t\rho_t(\bw,b) & = \alpha_0\nabla (\rho_t(\bw,b)\nabla \Psi(\bw,b;\rho_t))\, ,\label{eq:PDE}\\
    \Psi(\bw,b;\rho) & := \hE\big\{b\sigma(\<\bw,\bx\>)\big(F(\bx;\rho_t)-y\big)\big\}\, .
     \end{align}
     Here the gradient $\nabla$ is with respect to $(\bw,b)$ (gradient in $d+1$ dimensions) if both first- and second-layer
     weights are trained, and only with respect to $\bw$ (gradient in $d$ dimensions) if only first-layer weights are trained.
  \end{theorem}
This statement can be obtained by checking the conditions of
\cite[Theorem 2.6]{chizat2018global}. A quantitative version can be obtained for
bounded $\sigma$ using Theorem 1 of \cite{mei2019mean}.
  
  A few remarks are in order. First, the limit in \eqref{eq:MF-Limit} requires time to be accelerated by a factor
  $m$. This is to compensate for
  the fact that the function value is scaled by a factor $1/m$. Second, while we stated this theorem as an asymptotic result, for large $m$,
  the evolution described by the PDE~\eqref{eq:PDE} holds at any finite $m$ for the empirical measure $\hrho_t$. In that case, the gradient of
  $\rho_t$ is not well defined, and it is important to interpret this
  equation in the weak sense \cite{ambrosio2008gradient,santambrogio2015optimal}.
  The advantage of working with the average measure $\rho_t$ instead of the empirical one $\hrho_t$ is that the former is deterministic and
  has a positive density (this has important connections to global convergence). Third, quantitative versions of this theorem were proved
  in \cite{mei2018mean,mei2019mean}, and generalizations to multi-layer networks in \cite{nguyen2020rigorous}.

Mean-field theory can be used to prove global convergence results.
Before discussing these results,
  let us emphasize that ---in this regime---  the weights move in a non-trivial way during training,
  despite the fact that the network is infinitely wide. For the sake of simplicity, we will focus on the case
  already treated in the previous section in which the weights
  $b_j\in\{+1,-1\}$ are initialized with signs in equal proportions, and are not changed during training.
  Let us first consider the evolution of the predicted values $F_n(\rho_t) :=(F(\bx_1;\rho_t),\dots ,F(\bx_n;\rho_t))$.
  Manipulating \eqref{eq:PDE}, we get 
  \begin{align}
    \frac{\de\phantom{t}}{\de t}F_n(\rho_t) &= -\frac{1}{n}\bK_t \big(F_n(\rho_t)-\by\big)\, ,\;\;\;\;\; \bK_t = (K_t(\bx_i,\bx_j))_{i,j\le n}\,\\
    K_t(\bx_1,\bx_2)& := \int \<\bx_1,\bx_2\>\sigma'(\<\bw,\bx_1\>) \sigma'(\<\bw,\bx_2\>)\, \rho_t(\de b,\de\bw)\, ,
    \label{eq:EvolutionPredictedMinfty}
  \end{align}
  In the short-time limit we recover the linearized evolution of~\eqref{eq:EvolutionPred}
  \cite{mei2019mean}, but the kernel $K_t$ is now changing with training (with a
  factor $m$ acceleration in time).
  
  It also follows from the same characterization of Theorem
  \ref{thm:PDE} that the weight $\bw_j$ of a neuron with weight
  $(\bw_j,b_j) = (\bw,b)$
  moves at a speed $\hE\{b\,\bx\,\sigma'(\<\bw,\bx\>(F(\bx;\rho_t)-y)\}$.
  This implies
  \begin{align}
    \lim_{m\to\infty}\frac{1}{m}&\big\|\bW_{t+s}-\bW_t\|_F^2 = v_2(\rho_t)\, s^2+o(s^2)\, ,\label{eq:WeightChangeMF}\\
    v_2(\rho_t)&:= \frac{1}{n^2} \<\by-F_n(\rho_t),\bK_{t}(\by-F_n(\rho_t))\>\, .
  \end{align}
  This expression implies that the first-layer weights change
  significantly more than in the linear  regime studied in
  Section~\ref{sec:LinearRegime}.
  As an example, consider the setting of Lemma \ref{lemma:TwoLayers}, namely data $(\bx_i)_{i\le n}\sim_{i.i.d.}\normal(0,\id_d)$,
  an activation function satisfying Assumption \ref{ass:Sigma} and dimension parameters such that $md\ge C n\log n$, $n\le d^{\ell_0}$.
  We further initialize $\rho_0=\Unif(\Sp^{d-1})\otimes \Unif(\{+1,-1\})$ (that is, the vectors $\bw_j$ are uniform on the unit sphere and
  the weights $b_j$ are uniform in $\{+1,-1\}$). Under this
  initialization $\|\bW_0\|_F^2=m$ and hence~\eqref{eq:WeightChangeMF}
  at $t=0$ can be interpreted as describing the initial relative
  change of the first-layer weights.

 Theorem \ref{thm:Linearization} (see \eqref{eq:CloseToInit}) and Lemma \ref{lemma:TwoLayers}
  (see \eqref{eq:LemmaLin1}--\eqref{eq:LemmaLin3})
  imply that, with high probability,
  \begin{align}
    \sup_{t\ge 0}\frac{1}{\sqrt{m}}\|\bW_t-\bW_0\|_F \le C \frac{1}{\oalpha}\sqrt{\frac{n}{md}}\, ,\label{eq:UpperBoundInitialMov}
  \end{align}
  where $\oalpha = \alpha/(1+\alpha)$ for initialization $\btheta^{(1)}_0$ and  $\oalpha = \alpha$ for initialization $\btheta^{(2)}_0$.
   In the mean field regime $\oalpha \asymp \alpha \asymp 1/\sqrt{m}$ and  the right hand side above is of order
   $\sqrt{n/d}$, and hence it does not vanish.
   This is not due to a weakness of the analysis.
   By \eqref{eq:WeightChangeMF}, we can choose $\eps$ a small enough constant so that
  \begin{align}
    \lim_{m\to\infty} \sup_{t\ge 0}\frac{1}{\sqrt{m}}\|\bW_t-\bW_0\|_F \ge \lim_{m\to\infty} \frac{1}{\sqrt{m}}\|\bW_\eps-\bW_0\|_F \ge
    \frac{1}{2}v_2(\rho_0)^{1/2}\, \eps\, .\label{eq:WeightMovement}
  \end{align}
  This is bounded away from $0$ as long as $v_2(\rho_0)$ is non-vanishing. In order to see this, note that $\lambda_{\min}(\bK_0)\ge c_0\, d$
  with high probability   for $c_0$ a   constant (note that $\bK_0$ is a kernel inner product random matrix, and hence
  this claim follows from the general results of \cite{mei2020generalization}). Noting that $F_n(\rho_0) = 0$ (because $\int b \rho_0(\de b,\de\bw)=0$), this implies, with high probability, 
  \begin{align}
    v(\rho_0) &= \frac{1}{n^2}\<\by,\bK_0\by\> \ge  \frac{c_0d}{n^2} \|\by\|_2^2 \ge  \frac{c_0'd}{n}\, .
  \end{align}
  We expect this lower bound to be tight, as can be seen by considering the pure noise case
  $\by\sim\normal(0,\tau^2\id_n)$, which leads to   $v(\rho_0) = \tau^2\Trace(\bK_0)/n^2(1+o_n(1))\asymp d/n$.

  To summarize, \eqref{eq:UpperBoundInitialMov} (setting $\alpha \asymp 1/\sqrt{m}$)
  and \eqref{eq:WeightMovement} conclude that, for $d\le n\le d^{\ell_0}$,
  \begin{align}
    c_1\sqrt{\frac{d}{n}} \le  \lim_{m\to\infty} \sup_{t\ge 0}\frac{1}{\sqrt{m}}\|\bW_t-\bW_0\|_F \le c_2\sqrt{\frac{n}{d}}\, ,
    \label{eq:WeightMovementUB-LB}
\end{align}
hence the limit on the left-hand side of \eqref{eq:WeightMovement} is indeed
non-vanishing as $m\to\infty$ at $n,d$ fixed. In other words, the fact that the 
upper bound in \eqref{eq:UpperBoundInitialMov} is non-vanishing is not an artifact of the bounding technique, but a 
consequence of the change of training regime. We also note a gap between the upper and lower bounds in
\eqref{eq:WeightMovementUB-LB} when $n\gg d$: a better understanding of this quantity is an interesting open problem.
  In conclusion, both a linear and a nonlinear regime can be obtained in the infinite-width limit of two-layer neural networks,
  for different scalings of the normalization factor $\alpha$. 

  As mentioned above, the mean field limit can be used to prove global convergence results, both for two-layer \cite{mei2018mean,chizat2018global}
  and for multilayer  networks \cite{nguyen2020rigorous}. Rather than stating these (rather technical) results formally, it is instructive
  to discuss the nature of fixed points of the evolution \eqref{eq:PDE}: this will also indicate the key role played by the
  support of the distribution $\rho_t$.
  \begin{lemma}
    Assume $t\mapsto \sigma(t)$ to be differentiable with bounded derivative. Let $\hRisk(\rho) = \hE\{[y-F(\bx;\rho)]^2\}$
    be the empirical risk of an infinite-width network with neuron's distribution $\rho$, and define $\psi(\bw;\rho):= \hE\{\sigma(\<\bw,\bx\>) [y-F(\bx;\rho)]\}$.
    \begin{enumerate}
    \item[$(a)$] $\rho_*$ is a global minimizer of $\hRisk$ if and only if $\psi(\bw;\rho_*)=0$ for all $\bw\in\reals^d$.
    \item[$(b)$] $\rho_*$ is a fixed point of the evolution \eqref{eq:PDE} if and only if, for all $(b,\bw)\in\supp(\rho_*)$,
      we have $\psi(\bw;\rho_*)=0$ and $b\nabla_{\bw}\psi(\bw;\rho_*)=0$.
    \end{enumerate}
    The same statement holds if the empirical averages above are replaced by  population averages (that is,
    the empirical risk $\hRisk(\rho)$ is replaced by its population version $\Risk_n(\rho) = \E\{ [y-F(\bx;\rho)]^2\}$).
  \end{lemma}

  This statement clarifies that fixed points of the gradient flow are
  only a `small' superset of global minimizers, as $m\to\infty$. Consider for instance the case of an analytic activation function $t\mapsto\sigma(t)$.
 Let $\rho_*$ be a stationary point and assume that its support
 contains a sequence of distinct points $\{
 (b_i,\bw_i)\}_{i\ge 1}$ such that 
 $\{\bw_i\}_{i\ge 1}$ has an accumulation point. Then, by condition $(b)$, $\psi(\bw;\rho_*) = 0$ identically and therefore $\rho_*$ is a
 global minimum. In other words, the only local minima correspond to $\rho_*$ supported on a set of isolated points.
 Global convergence proofs aim at ruling out this case.

 \subsection{Other approaches}

 The mean-field limit is only one of several analytical approaches
 that have been developed
 to understand training beyond the linear regime. A full survey of these directions goes beyond the
 scope of this review. Here we limit ourselves to highlighting a few of them that have a direct connection to
 the analysis in the previous section.

 A natural idea is to view the linearized evolution as the first order in a Taylor expansion,
 and to construct higher order approximations. This can be achieved by writing an ordinary differential equation for the evolution of
 the kernel $\bK_t$ (see
 \eqref{eq:EvolutionPredictedMinfty} for the infinite-width limit). This takes the form
 \cite{huang2020dynamics} 
 \begin{align}
   \frac{\de\phantom{t}}{\de t}\bK_t = -\frac{1}{n}\bK^{(3)}_t\cdot(F_n(\rho_t)-\by)\, ,
 \end{align}
 where $\bK^{(3)}_t\in (\reals^n)^{\otimes 3}$ is a certain higher order kernel (an order-$3$ tensor), which is contracted along one
 direction with $(F_n(\rho_t)-\by)\in \reals^n$. The linearized approximation amounts to replacing $\bK^{(3)}_t$ with $0$.
 A better approximation could be to replace $\bK^{(3)}_t$ with its value at initialization $\bK^{(3)}_0$. This construction can be
 repeated, leading to a hierarchy of increasingly complex (and accurate) approximations. 

 Other approaches towards constructing a Taylor expansion around the linearized evolutions were proposed, among others,  by
 \cite{dyer2019asymptotics} and \cite{hanin2019finite}.

 Note that the linearized approximation relies on the assumption that the Jacobian $\bD f_n(\btheta_0)$ is
 non-vanishing and well conditioned. \cite{bai2019beyond} propose specific neural network parametrizations
 in which the Jacobian at initialization vanishes, and the first non-trivial term in the Taylor expansion is quadratic.
 Under such initializations the gradient flow dynamics is `purely nonlinear'.

%% file: ntk.tex
\section{Generalization in the linear regime}
\label{sec:NTK}

As discussed in Sections~\ref{sec:slt} and~\ref{sec:benign}, approaches that control the test error via uniform convergence 
fail for overparametrized interpolating models. So far, the most complete generalization results for such models have  been obtained in the linear regime,
namely under the assumption that we can approximate
$f(\btheta)$ by its first order Taylor approximation $f_{\slin}(\btheta) = f(\btheta_0)+\bD f(\btheta)
(\btheta-\btheta_0)$. 
While  Theorem \ref{thm:Linearization} provides a set of sufficient conditions for this approximation
to be accurate, in this section we leave aside the question of whether or when this is indeed the case, and review what we know
about the generalization properties of these linearized models.
We begin in Section \ref{sec:BiasRF} by discussing the inductive bias induced by gradient descent on wide two-layer networks.
Section \ref{sec:Setup-RF} describes a general setup.
Section \ref{sec:RF} reviews random features models: two-layer neural networks in which the first layer is not trained
and entirely random. While these are simpler than neural networks in the linear regime, their generalization
behavior is in many ways similar. Finally, in Section \ref{sec:NT} we review progress on the generalization error of
linearized two-layer networks.

\subsection{The  implicit regularization of gradient-based training}
\label{sec:BiasRF}
  
 As emphasized in previous sections, in an overparametrized setting,  convergence to global minima is not sufficient to characterize the
 generalization properties of neural networks. It is equally important to understand which global minima are selected by the training algorithm,
 in particular by gradient-based training.
 As shown in Section \ref{sec:implicit}, in linear models gradient descent converges
 to the minimum $\ell_2$-norm interpolator.
 Under the assumption that training takes place in the linear regime (see Section \ref{sec:LinearRegime}), we can apply this observation
 to neural networks. 
 Namely, the neural network trained by gradient descent will be well approximated by the  model\footnote{With a slight abuse of notation,
 in this section we parametrize the linearized model by the shift with respect to the initialization $\btheta_0$.}
 $f_{\slin}(\hba) = f(\btheta_0)+\bD f(\btheta_0)\hba$ where $\hba$ minimizes
 $\|\ba\|_2$  among empirical risk minimizers
\begin{align}
  \hba &:= \argmin{\ba\in\reals^p}\Big\{\|\ba\|_2:~  y_i=f_{\slin}(\bx_i;\ba) ~\text{for all}~ i\le n\Big\}\, .\label{eq:MinNormLin}
\end{align}
For simplicity, we will set $f(\bx;\btheta_0)=0$. This can be achieved either by properly constructing the initialization $\btheta_0$
(as in the initialization $\btheta_0^{(2)}$ in Section \ref{sec:LinearRegime}) or by redefining the response vector $\by' =\by-f_n(\btheta_0)$.
If  $f(\bx;\btheta_0)=0$, the interpolation constraint  $y_i=f_{\slin}(\bx_i;\ba)$ for all $i\le n$  can be written as $\bD f_n(\btheta_0)\ba=\by$.

Consider the case of two-layer neural networks in which only first-layer weights are trained.  Recalling the form of
the Jacobian \eqref{eq:TwoLayersJacobian}, we can rewrite \eqref{eq:MinNormLin} as
\begin{align}
  \hba &:= \argmin{\ba\in\reals^{md}}\Big\{\|\ba\|_2:~ y_i=\sum_{j=1}^m\<\ba_j,\bx_i\>\sigma'(\<\bw_j,\bx_i\>)\Big\}\, ,
\end{align}
where we write $\ba=(\ba_1,\dots,\ba_m)$, $\ba_i\in\reals^d$.
In this section we will study the generalization properties of this
{\em neural tangent} (NT) model and some of its close relatives.
Before formally defining our setup, it is instructive to rewrite the norm that we are minimizing
in
function space:
\begin{align}
  \label{eq:NT_Norm_m}
  \|f\|_{\NT,m}:= \inf\Big\{\frac{1}{\sqrt{m}}\|\ba\|_2:~ f(\bx)=\frac{1}{m}\sum_{j=1}^m\<\ba_j,\bx_i\>\sigma'(\<\bw_j,\bx\>)\; \mbox{a.e.}
  \Big\}.
\end{align}
This is an RKHS norm defining  a finite-dimensional subspace of
$L^2(\reals^d,\P)$.
We can also think of it as a finite approximation to the norm
\begin{align}
  \|f\|_{\NT}:= \inf\Big\{\|\ba\|_{L^2(\rho_0)}:~ f(\bx)=\int \!\<\ba(\bw),\bx_i\>\sigma'(\<\bw,\bx\>)\, \rho_0(\de\bw)\Big\}\, .
  \label{eq:NT_Norm}
\end{align}
Here $\ba:\reals^d\to\reals^d$ is a measurable function with
  \[
    \|\ba\|_{L^2(\rho_0)}^2 := \int \|\ba(\bw)\|^2\rho_0(\de \bw)<\infty,
  \]
and we are assuming that the weights $\bw_j$ in \eqref{eq:NT_Norm_m}
are initialized as
$$(\bw_j)_{j\le m}\sim_{i.i.d.}\rho_0.
$$
This is also an RKHS norm whose kernel $K_{\NT}(\bx_1,\bx_2)$
will be described below; see \eqref{eq:K_NT}.

Let us emphasize that moving out of the linear regime leads to different---and possibly more interesting---inductive biases
than those described in \eqref{eq:NT_Norm_m} or \eqref{eq:NT_Norm}.
As an example, \cite{chizat2020implicit} analyze the mean field limit of two-layer networks, trained with logistic loss,
for activation functions that have Lipschitz gradient and are positively $2$-homogeneous. For instance, the square ReLU
$\sigma(x) = (x_+)^2$ with fixed second-layer coefficients  fits this framework. The usual ReLU with trained second-layer coefficients
$b_j\sigma(\<\bw_j,\bx\>)= b_j(\<\bw_j,\bx\>)_+$ is $2$-homogeneous but not differentiable.
In this setting, and under a convergence
assumption, they show that gradient flow minimizes the following norm
among interpolators:
\begin{align}
  \|f\|_{\sigma}:= \inf\Big\{\|\nu\|_{\sTV}:~ f(\bx)=\int \sigma(\<\bw,\bx\>)\, \nu(\de\bw) \mbox{ a.e.}\Big\}\, .
  \label{eq:TV_Norm} 
\end{align}
Here, minimization is over the finite signed measure $\nu$ with Hahn
decomposition  $\nu= \nu_+-\nu_-$,
and $\|\nu\|_{\sTV}:=\nu_+(\reals^d)+\nu_-(\reals^d)$ is the
associated total variation. The norm   $\|f\|_{\sigma}$ is a special
example of the variation norms introduced in \cite{kurkova1997dimension} and further studied in  \cite{kurkova2001bounds,kurkova2002comparison}. 

This norm differs in two  ways from the RKHS norm of \eqref{eq:NT_Norm}. Each is  defined in terms of a different integral operator,
  \[
    \ba\mapsto \int \<\ba(\bw),\bx\>\sigma'(\<\bw,\bx_i\>)\, \rho_0(\de\bw)
  \]
for \eqref{eq:NT_Norm} and
$$\nu\mapsto  \int \sigma(\<\bw,\bx\>)\, \nu(\de\bw)$$ for
\eqref{eq:TV_Norm}. However, more importantly, the norms
are very
different: in~\eqref{eq:NT_Norm} it is a Euclidean norm while in
\eqref{eq:TV_Norm} it is a total variation norm.
Intuitively, the total variation norm $\|\nu\|_{\sTV}$ promotes `sparse' measures $\nu$, and hence the functional
norm $\|f\|_{\sigma}$ promotes functions that depend primarily on a small number of directions in $\reals^d$
\cite{bach2017breaking}.

\subsection{Ridge regression in the linear regime}
\label{sec:Setup-RF}

We generalize the min-norm procedure of \eqref{eq:MinNormLin} to consider the ridge regression estimator:
\begin{align}
  \hba(\lambda) &:= \argmin{\ba\in\reals^p}\Big\{\frac{1}{n}\sum_{i=1}^n\big(y_i-f_{\slin}(\bx_i;\ba)\big)^2+\lambda\|\ba\|^2_2\Big\}\, ,\label{eq:NT_Ridge}\\
          f_{\slin}(\bx_i;\ba)&:= \<\ba,\bD f(\bx_i;\btheta_0)\>\, .\label{eq:TangentModel}
\end{align}
The min-norm estimator can be recovered by taking the limit of vanishing regularization $\lim_{\lambda\to 0}\hba(\lambda) = \hba(0^+)$ (with a slight abuse of notation, we will identify $\lambda=0$ with this limit).
Apart from being intrinsically interesting, the behavior of $\hba(\lambda)$ for $\lambda >0$ is a good
approximation of the behavior of the estimator produced by gradient flow with early stopping \cite{ali2019continuous}.
More precisely, letting $(\hba_{\sGF}(t))_{t\ge 0}$ denote the path of gradient flow initialized at $\hba_{\sGF}(0) =0$, there
exists a parametrization $t\mapsto \lambda(t)$, such that the test error at  $\hba_{\sGF}(t)$ is well approximated by the test error at $\hba(\lambda(t))$.

Note that the function class $\{ f_{\slin}(\bx_i;\ba):= \<\ba,\bD f(\bx_i;\btheta_0)\>:\; \ba\in\reals^p\}$ is a linear space,
which is linearly parametrized by $\ba$.
We consider two specific examples which are obtained by linearizing two-layer neural networks (see \eqref{eq:Two-Layers-Net}):
\begin{align}
  \cF^m_{\RF}&:=\Big\{f_{\slin}(\bx;\ba) = \sum_{i=1}^m a_i\sigma(\<\bw_i,\bx\>):\;\; a_i\in \reals\Big\}\, ,\label{eq:F_RF}\\
  \cF^m_{\NT}&:=\Big\{f_{\slin}(\bx;\ba) = \sum_{i=1}^m \<\ba_i,\bx\>\sigma'(\<\bw_i,\bx\>):\;\; \ba_i\in \reals^d\Big\}\, .\label{eq:F_NT}
\end{align}
 Namely,  $\cF^m_{\RF}$ (RF stands for `random features') is the class
  of functions obtained by linearizing a two-layer network
  with respect to second-layer weights and keeping the first layer fixed,
  and $\cF^m_{\NT}$ (NT stands for `neural tangent') is the class
   obtained by linearizing a two-layer network with respect to the first layer and keeping the second fixed.
The first example was introduced by \cite{balcan2006kernels} and \cite{rahimi2007random} and can be viewed as a linearization of the
two-layer neural networks in which only second-layer weights are
trained. Of course, since the network is linear in the second-layer
weights,
it coincides with its linearization. The second example is the
linearization of a neural network in which only the first-layer
weights are trained.
In both cases, we draw $(\bw_i)_{i\le m}\sim_{i.i.d.}\Unif(\S^{d-1})$ (the Gaussian initialization $\bw_i\sim\normal(0,\id_d/d)$ behaves very similarly).

Ridge regression \eqref{eq:NT_Ridge} within either model $\cF_{\RF}$
or  $\cF_{\NT}$ can be viewed as kernel ridge regression (KRR)
with respect to the kernels
\begin{align}
  K_{\RF,m}(\bx_1,\bx_2) &:=\frac{1}{m}\sum_{i=1}^m\sigma(\<\bw_i,\bx_1\>)\sigma(\<\bw_i,\bx_2\>)\, ,\\
  K_{\NT,m}(\bx_1,\bx_2) &:=\frac{1}{m}\sum_{i=1}^m\<\bx_1,\bx_2\>\sigma'(\<\bw_i,\bx_1\>)\sigma'(\<\bw_i,\bx_2\>)\, .
\end{align}
These kernels are random (because the weights $\bw_i$ are) and have finite rank, namely rank at most $p$, where $p=m$ in the first case
and $p=md$ in the second. The last property is equivalent to the fact that the RKHS is at most $p$-dimensional. 
As the number of neurons diverge, these kernels converge to their expectations
$K_{\RF}(\bx_1,\bx_2)$ and $K_{\NT}(\bx_1,\bx_2)$. Since
the distribution of $\bw_i$ is invariant under rotations in $\reals^d$, so are these kernels.
The kernels  $K_{\RF}(\bx_1,\bx_2)$ and $K_{\NT}(\bx_1,\bx_2)$  can therefore be written as functions
of $\|\bx_1\|_2$, $\|\bx_2\|_2$ and $\<\bx_1,\bx_2\>$. In particular, if we assume that data are normalized, say $\|\bx_1\|_2=\|\bx_2\|_2=\sqrt{d}$,
then we have the particularly simple form
\begin{align}
  K_{\RF}(\bx_1,\bx_2) & = H_{\RF,d}(\<\bx_1,\bx_2\>/d)\, , \label{eq:K_RF}\\
  K_{\NT}(\bx_1,\bx_2) &= d\, H_{\NT,d}(\<\bx_1,\bx_2\>/d)\, ,\label{eq:K_NT}
\end{align}
where
\begin{align}
H_{\RF,d}(q)& :=\E_{\bw}\{\sigma(\sqrt{d}\<\bw,\bfe_1\>) \sigma(\sqrt{d}\<\bw,q\bfe_1+\oq\bfe_2\>)\}\, ,\label{eq:HRF}\\
  H_{\NT,d}(q)&:=q\E_{\bw}\{\sigma'(\sqrt{d}\<\bw,\bfe_1\>) \sigma'(\sqrt{d}\<\bw,q\bfe_1+\oq\bfe_2\>)\}\, ,
\end{align}
with $\oq:=\sqrt{1-q^2}$.

The convergence  $K_{\RF,m}\to K_{\RF}$, $K_{\NT,m}\to K_{\NT,m}$
takes place under suitable assumptions,
pointwise  \cite{rahimi2007random}.
However, we would like to understand the qualitative behavior of the generalization error in the above linearized models.
\begin{enumerate}
\item[$(i)$] Does the procedure \eqref{eq:NT_Ridge}  share qualitative behavior
 with KRR, as discussed in Section \ref{sec:benign}? In particular, can min-norm interpolation be (nearly) optimal in the \RF\ or \NT\ models as well?
\item[$(ii)$] How large should $m$ be for the  generalization properties of \RF\ or \NT\ ridge regression
  to match those of  the associated kernel?
\item[$(iii)$] What discrepancies between KRR and \RF\ or \NT\ regression can we observe when $m$ is not sufficiently large?
\item[$(iv)$] Is there any advantage of one of the three methods (KRR, \RF, \NT) over the others?
\end{enumerate}

Throughout this section we assume an isotropic model for the
distribution of the covariates $\bx_i$,
namely we assume $\{(\bx_i,y_i)\}_{i\le n}$ to be i.i.d., with 
\begin{align}
  y_i = f^*(\bx_i)+\eps_i\, ,\;\;\;\; \bx_i\sim\Unif(\Sp^{d-1}(\sqrt{d}))\, ,
\end{align}
where $f^*\in L^2(\Sp^{d-1})$ is a square-integrable function on the sphere and $\eps_i$ is noise independent of
$\bx_i$, with $\E\{\eps_i\}=0$, $\E\{\eps_i^2\}=\tau^2$. We will also consider a modification of this model in
which $\bx_i\sim\normal(0,\id_d)$; the two settings are very close to each other in high dimension.
Let us emphasize that we do not make any regularity assumption about
the target function beyond square integrability, which is the bare minimum for the risk to be well defined.
On the other hand, the covariates have a simple isotropic distribution and the noise has  variance independent of $\bx_i$
(it is homoscedastic).

While homoscedasticity is not hard to relax to an upper bound on the noise variance, it is useful
to comment on the isotropicity assumption. The main content of this assumption is that the ambient dimension $d$ of the covariate
vectors does coincide with the intrinsic dimension of the data. If,
for instance, the $\bx_i$ lie on a $d_0$-dimensional
subspace in $\reals^d$, $d_0\ll d$, then it is intuitively clear that $d$ would have to be replaced by $d_0$ below. Indeed this
is a special case of a generalization studied in \cite{ghorbani2020neural}. An even more general setting is considered
in \cite{mei2020generalization}, where $\bx_i$ belongs to an abstract space. The key assumption there is that leading eigenfunctions of the associated kernel are delocalized.

We evaluate the quality of method \eqref{eq:NT_Ridge} using the square loss
\begin{align}
  \Risk(\lambda):= \E_{\bx}\big\{(f^*(\bx)-f_{\slin}(\bx;\hba(\lambda))^2\big\}\, .
\end{align}
The expectation is with respect to the test point $\bx\sim \Unif(\Sp^{d-1}(\sqrt{d}))$;
note that the risk is random because $\hba(\lambda)$ depends on the training data. However,
in all the results below, it concentrates around a non-random value.
We add subscripts, and write $\Risk_{\RF}(\lambda)$ or $\Risk_{\NT}(\lambda)$ to refer to the two classes of models above.

\subsection{Random features model}
\label{sec:RF}

We begin by considering the random features model $\cF_{\RF}$. 
A number of authors have established upper bounds on its minimax
generalization error for suitably chosen positive values of the
regularization \cite{rudi2017generalization,rahimi2009weighted}.
Besides the connection to neural networks,  $\cF_{\RF}$  can be viewed as a randomized approximation  for
the RKHS associated with $K_{\RF}$. A closely related approach in this context is provided by randomized subset selection,
also known as Nystr\"om's method \cite{williams2001using,bach2013sharp,alaoui2015fast,rudi2015less}.

The classical random features model $\cF_{\RF}$ is mathematically easier to analyze than the neural tangent model
 $\cF_{\NT}$,   and a precise picture can be established that covers
 the interpolation limit. Several elements of this picture have been
proved to generalize to the $\NT$ model as well, as discussed in the next subsection.

We focus on the high-dimensional regime, $m,n,d\to\infty$; as discussed in Section~\ref{sec:benign}, interpolation methods have
appealing properties in high dimension. Complementary asymptotic descriptions are obtained depending on how
$m,n,d$ diverge.
In Section \ref{sec:RF_Polynomial} we discuss the behavior at a
coarser scale, namely when $m$ and $n$ scale polynomially
in $d$: this type of analysis provides a simple quantitative answer to the question of how large $m$ should be to approach the $m=\infty$ limit.
Next, in Section \ref{sec:RF_Proportional}, we consider the proportional regime $m\asymp n\asymp d$.
This allows us to explore more precisely what happens in the
transition from underparametrized to overparametrized.

\subsubsection{Polynomial scaling}
\label{sec:RF_Polynomial}

The following characterization was proved in
\cite{mei2020generalization} (earlier work
by \cite{ghorbani2020linearized} established this result for
the two limiting cases $m=\infty$ and $n=\infty$). In what follows, we let $L^2(\gamma)$ denote
the space of square integrable functions on $\reals$, with respect to the standard Gaussian measure
$\gamma(\de x) = (2\pi)^{-1/2}e^{-x^2/2}\de x$, and we
write $\<\,\cdot\, ,\,\cdot\,\>_{L^2(\gamma)}$, $\|\,\cdot\,\|_{L^2(\gamma)}$
for the associated scalar product and norm.
\begin{theorem}\label{thm:RFPolynomial}
  Fix an integer $\ell>0$. Let the activation function $\sigma:\reals\to\reals$ be independent of $d$ and such that:
  $(i)$~$|\sigma(x)|\le c_0\exp(|x|^{c_1})$ for some constants $c_0>0$ and $c_1<1$, and 
  $(ii)$ $\<\sigma, q\>_{L^2(\gamma)} \neq 0$ for any non-vanishing polynomial $q$, with $\deg(q)\le \ell$.
  Assume  $\max((n/m),(m/n) )\ge  d^{\delta}$ and
  $d^{\ell+\delta}\le \min(m,n) \le d^{\ell+1-\delta}$ for some constant $\delta>0$. Then  for any $\lambda = O_d((m/n)\vee 1)$,
  and all $\eta>0$,
  \begin{align}
    \Risk_{\RF}(\lambda) = \|\proj_{>\ell} f^*\|_{L^2}^2 +o_d(1) \big(\| f^*\|_{L^2}^2+\| \proj_{>\ell}f^*\|_{L^{2+\eta}}^2+\tau^2\big)\, .
  \end{align}
\end{theorem}
In words, as long as the number of parameters $m$ and the number of samples $n$ are well separated,
the test error is determined by the minimum of $m$ and $n$:
\begin{itemize}
\item For $m\ll n$, the approximation error dominates. If $d^{\ell}\ll m\ll d^{\ell+1}$, the model fits the projection of $f$ onto
  degree-$\ell$ polynomials perfectly but does not fit the higher degree components at all: $\hf_{\lambda}\approx \proj_{\le\ell}f$.
  This is consistent with a parameter-counting heuristic: degree-$\ell$ polynomials form a subspace of dimension $\Theta(d^{\ell})$
  and in order to approximate them we need a network with $\Omega(d^{\ell})$ parameters. Surprisingly, this transition is sharp.   
\item For $n\ll m$, the statistical error dominates. If $d^{\ell}\ll n\ll d^{\ell+1}$, $\hf_{\lambda}\approx \proj_{\le\ell}f$.
  This is again consistent with a parameter-counting heuristic: to learn degree-$\ell$ polynomials we need roughly as many samples as parameters.
\item Both of the above are achieved for any sufficiently small value of the regularization parameter $\lambda$. In particular, they apply to
  min-norm interpolation (corresponding to the case $\lambda=0^+$).
\end{itemize}
From a practical perspective, if the sample size $n$ is given, we might be interested in choosing the number of neurons $m$.
The above result indicates that the test error roughly decreases until the overparametrization threshold $m\approx n$,
and that there is limited improvement from increasing the network size
beyond $m\ge nd^{\delta}$.  At this point,
RF ridge regression achieves the same error as the corresponding kernel method. Indeed the statement of Theorem \ref{thm:RFPolynomial} holds
for the case of KRR as well, by identifying it with the limit $m=\infty$ \cite{ghorbani2020linearized}.

Note that the infinite width (kernel) limit $m=\infty$ corresponds to the setting already investigated
in Theorem \ref{thm:var_multiple_descent}. Indeed, the staircase
phenomenon in the $m=\infty$ case of Theorem \ref{thm:RFPolynomial}
corresponds to the multiple descent behavior seen in  Theorem \ref{thm:var_multiple_descent}. The two results
do not imply each other because  Theorem~\ref{thm:var_multiple_descent} assumes $f^*$ to have bounded
RKHS norm;  Theorem \ref{thm:RFPolynomial} does not make this assumption, but is not as sharp for
functions with bounded RKHS norm.

The significance of polynomials in Theorem \ref{thm:RFPolynomial} is related to the fact that the kernel $K_{\RF}$
is invariant under rotations (see \eqref{eq:K_RF}). As a consequence, the eigenfunctions of $K_{\RF}$ are
spherical harmonics, that is, restrictions of homogeneous harmonic polynomials in $\reals^d$ to the sphere $\S^{d-1}(\sqrt{d})$,
with eigenvalues given by their degrees.
\cite{mei2020generalization} have obtained analogous results for more general probability spaces  $(\cX,\P)$ for the covariates,
and more general random features models.
The role of low-degree polynomials is played by the top eigenfunctions of the associated kernel.

The mathematical phenomenon underlying Theorem \ref{thm:RFPolynomial}
can be understood by considering the feature
matrix $\bPhi\in \reals^{n\times m}$:
\begin{align}
  \label{eq:RF-Phimatrix}
  \bPhi:=
  \left[\begin{matrix}
    \sigma(\<\bx_1,\bw_1\>) &   \sigma(\<\bx_1,\bw_2\>) &   \cdots & \sigma(\<\bx_1,\bw_m\>)\\
    \sigma(\<\bx_2,\bw_1\>) &   \sigma(\<\bx_2,\bw_2\>) &   \cdots & \sigma(\<\bx_2,\bw_m\>)\\
    \vdots & \vdots& & \vdots\\
    \sigma(\<\bx_n,\bw_1\>) &   \sigma(\<\bx_n,\bw_2\>) &   \cdots & \sigma(\<\bx_n,\bw_m\>)\\
  \end{matrix}\right]\, .
\end{align}
The $i$th row of this matrix is the feature vector associated with the $i$th sample. We can decompose
$\bPhi$ according to the eigenvalue decomposition of $\sigma$, seen as an integral operator
from $L^2(\S^{d-1}(1))$ to $L^2(\S^{d-1}(\sqrt{d}))$:
$$a(\bw)\mapsto \int\sigma(\<\bx,\bw\>)\,
a(\bw)\,\tau_d(\de\bw)$$
(where $\tau_d$ is the uniform measure on
$\S^{d-1}(1)$).
This takes the form
\begin{align}
  \sigma(\<\bx,\bw\>) = \sum_{k=0}^{\infty}\eval_{k}\psi_k(\bx)\phi_k(\bw)\, ,
\end{align}
where $(\psi_j)_{j\ge 1}$ and $(\phi_j)_{j\ge 1}$ are two orthonormal
systems in $L^2(\S^{d-1}(\sqrt{d}))$ and  $L^2(\S^{d-1}(1))$
respectively, and the $\eval_{j}$ are singular values $\eval_0\ge
\eval_1\ge \dots \ge 0$.
(In the present example, $\sigma$ can be regarded as a self-adjoint operator on
$L^2(\S^{d-1}(\sqrt{d}))$ after rescaling the $\bw_j$, and hence the $\phi_j$ and $\psi_j$ can be taken to coincide up
to a rescaling, but this is not crucial.)

The eigenvectors are grouped into eigenspaces $\cV_{\ell}$ indexed by $\ell\in\integers_{\ge 0}$,
where $\cV_{\ell}$ consists of the degree-$\ell$ polynomials, and
$$\dim(\cV_{\ell}) =: B(d,\ell) =
\frac{d-2+2\ell}{d-2}\binom{d-3+\ell}{\ell},~ B(d,\ell)\asymp
d^{\ell}/\ell!.$$ 
We write $\eval^{(\ell)}$ for the eigenvalue
associated with eigenspace $\cV_{\ell}$:
it turns out that $\eval^{(\ell)}\asymp d^{-\ell/2}$, for a generic
$\sigma$; $(\eval^{(\ell)})^2B(d,\ell)\le C$ since $\sigma$
is square integrable.
Let $\bpsi_k = (\psi_k(\bx_1),\dots,\psi_k(\bx_n))^{\sT}$ be
the evaluation of the $k$th left eigenfunction at the $n$ data points, and let $\bphi_k = (\phi_k(\bw_1),\dots,\phi_k(\bw_m))^{\sT}$
be the evaluation of the $k$th right eigenfunction at the $m$ neuron
parameters. Further, let
$k(\ell) := \sum_{\ell'\le \ell}B(d,\ell')$.  Following our approach in Section~\ref{sec:benign}, we decompose $\bPhi$ into a `low-frequency' and a `high-frequency' component,
\begin{align}
  \bPhi &= \bPhi_{\le \ell}+\bPhi_{>\ell}\, ,\label{eq:PhiDecomposition}\\
  \bPhi_{\le \ell} & = \sum_{j=0}^{k(\ell)}\eval_j
  \bpsi_{j}\bphi_j^\top =
   \bpsi_{\le \ell}\Eval_{\le \ell} \bphi_{\le \ell}^\top\, ,
\end{align}
where $\Eval_{\le \ell}= \diag(\eval_1,\dots,\eval_{k(\ell)})$,
$\bpsi_{\le \ell}\in \reals^{n\times k(\ell)}$ is the
matrix whose $j$th column is $\bpsi_j$, and $\bphi_{\le
\ell}\in \reals^{m\times k(\ell)}$ is the
matrix whose $j$th column is $\bphi_j$.

Consider, to be definite, the overparametrized case $m\ge n^{1+\delta}$, and assume $ d^{\ell+\delta}\le n$.
Then we can think of $\bphi_j$, $\bpsi_j$, $j\le k(\ell)$ as densely sampled eigenfunctions.
This intuition is accurate  in the sense that
$\bpsi^{\sT}_{\le
\ell}\bpsi_{\le \ell}\approx n\id_{k(\ell)}$
and  $\bphi^{\sT}_{\le \ell}\bphi_{\le \ell}\approx
m\id_{k(\ell)}$ \cite{mei2020generalization}.
Further, if $n\le d^{\ell+1-\delta}$, the `high-frequency' part of the decomposition \eqref{eq:PhiDecomposition}
behaves similarly to noise
along directions orthogonal to the previous ones. Namely, $(i)$
$\bPhi_{>\ell}\bphi_{\le \ell}\approx 0$,
$\bpsi^{\sT}_{\le \ell}\bPhi_{>\ell}\approx 0$, and $(ii)$ its singular values (except those along the low-frequency components)
concentrate: for any $\delta'>0$, 
$$\kappa_{\ell}^{1/2} n^{-\delta'}\le\sigma_{n-k(\ell)}(\bPhi_{>\ell})/m^{1/2}\le
\sigma_{1}(\bPhi_{>\ell})/m^{1/2}\le \kappa_{\ell}^{1/2} n^{\delta'},$$
where $\kappa_{\ell}:=\sum_{j\ge k(\ell)+1}\eval_{j}^2$.

In summary, regression with respect to the random features $\sigma(\<\bw_{j},\,\cdot\,\>)$ turns
out to be essentially equivalent to kernel ridge regression with respect to a polynomial kernel of degree $\ell$,
where $\ell$ depends on the smaller of the sample size and the network size. Higher degree parts in the activation
function effectively behave as noise in the regressors. We will next see that this picture can become even more precise in the proportional
regime $m\asymp n$.
  
\subsubsection{Proportional scaling}
\label{sec:RF_Proportional}
  
Theorem \ref{thm:RFPolynomial} requires that $m$ and $n$ are well separated. When $m,n$ are close
to each other, the feature matrix \eqref{eq:RF-Phimatrix} is nearly square and we might expect its
condition number to be large. When this is the case, the variance component of the risk can also be large.

Theorem \ref{thm:RFPolynomial}  also requires the smaller of $m$ and $n$ to be well separated from $d^{\ell}$,
with $\ell$ any integer.
For $d^{\ell}\ll m\ll d^{\ell+1}$ the model
has enough degrees of freedom to represent (at least in principle) all polynomials of degree at most $\ell$ and not enough
to represent even a vanishing fraction of all polynomials of degree $\ell+1$. Hence it behaves in a particularly simple
way. On the other hand, when $m$ is comparable to $d^{\ell}$, the model can partially represent degree-$\ell$
polynomials, and its behavior will be more complex.  
Similar considerations apply to the sample size $n$.

\begin{figure}[t]
\centering
\includegraphics[width=0.75\textwidth]{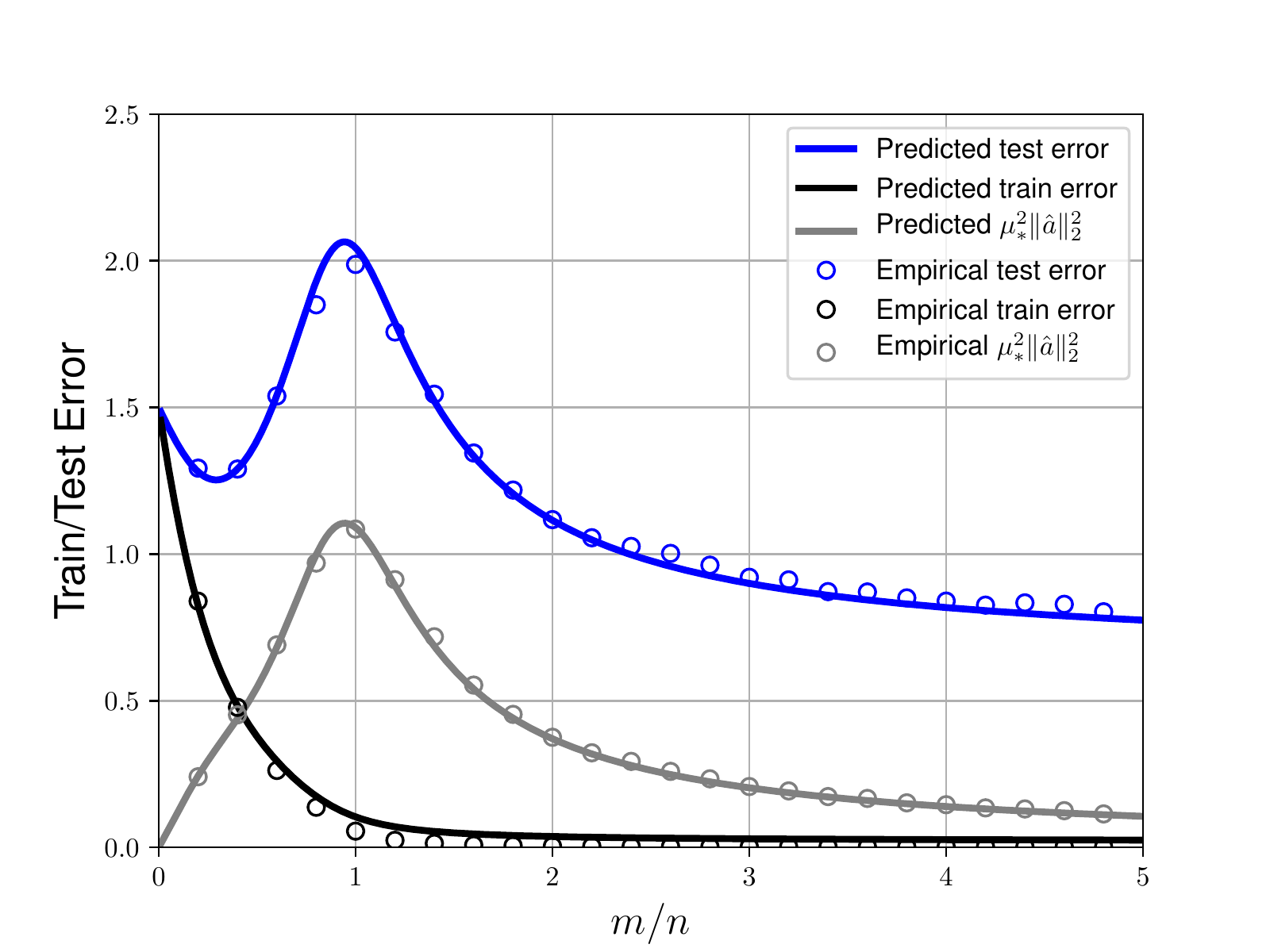} 
\caption{Train and test error of a random features model (two-layer
  neural net with random first layer) as a function of the
  overparametrization ratio $m/n$. Here $d=100$,  $n=400$,
    $\tau^2=0.5$, and the target function is $f^*=\<\bbeta_0,\bx\>$,
  $\|\bbeta_0\|_2=1$.
  The model is fitted using ridge
  regression with a small regularization parameter $\lambda=10^{-3}(m/d)$. Circles report the results of numerical
  simulations (averaged over 20 realizations), while lines are theoretical predictions for the $m,n,d\to\infty$ asymptotics. }  
\label{fig:proportional} 
\end{figure}

What happens when $m$ is comparable to $n$, and both are comparable to an integer power of $d$?
Figure \ref{fig:proportional} reports simulations
within the data model introduced above.
We performed ridge regression as per \eqref{eq:NT_Ridge}, with a small value of the regularization parameter, $\lambda=10^{-3}(m/d)$.
We report test error and train error for several network widths $m$, plotting them as a function of the overparametrization ratio $m/n$.

We observe that the train error decreases with the overparametrization ratio, and becomes very small for $m/n\ge 1$:
it is not exactly $0$ because we are using $\lambda>0$, but for $m/n>1$ it vanishes as $\lambda\to 0$.
On the other hand,
the test error displays a peak at the interpolation threshold $m/n=1$. For $\lambda=0^+$ the error actually diverges at this threshold. It  then
decreases and converges rapidly to an asymptotic value as $m/n\gg 1$.
If both $n/d\gg 1$, and $m/n\gg 1$, the asymptotic value of the test error is given by $\|\proj_{>1} f^*\|_{L^2}$:
the model is fitting
the degree-one polynomial component of the target
function perfectly and behaves trivially on higher degree components.
This matches the picture obtained under polynomial scalings, in Theorem \ref{thm:RFPolynomial}, and actually indicates that 
a far smaller separation between $m$ and $n$ is required than assumed in that theorem. Namely, $m/n\gg 1$  instead of $m/n\ge d^{\delta}$
appears to be sufficient for the risk to be dominated by the statistical error.

The peculiar behavior illustrated in Figure~\ref{fig:proportional}
was first observed empirically in neural networks and then shown to be ubiquitous for
numerous over-parametrized models \cite{geiger2019jamming,spigler2019jamming,belkin2019reconciling}.
It is commonly referred to as the `double descent phenomenon', after \cite{belkin2019reconciling}.

Figure \ref{fig:proportional} displays curves that are exact asymptotic predictions in the limit $m,n,d\to\infty$,
with $m/d\to\psi_{\wid}$, $n/d\to\psi_{\samp}$. Explicit formulas for these asymptotics were originally established in
\cite{mei2019generalization} using an approach from random matrix theory, which we will briefly outline.
The first step is to write the risk as an explicit
function of the matrices $\bX\in\reals^{n\times d}$ (the matrix whose $i$th row is the sample $\bx_i$),
$\bTheta\in\reals^{m\times d}$ (the matrix whose $j$th row is the sample $\btheta_j=\sqrt{d}\bw_j$), and
$\bPhi = \sigma(\bX\bTheta^{\sT}/\sqrt{d})$ (the feature matrix in \eqref{eq:RF-Phimatrix}). After a straightforward calculation,
one obtains
\begin{align}
  \Risk_{\RF}(\lambda) =&\E_{\bx}[f^*(\bx)^2] - \frac{2}{n} \by^\sT \bPhi(\bPhi^{\sT}\bPhi/n+\lambda\bI_m)^{-1} \bV  \label{eq:ExplicitRisk}\\
  &+  \frac{1}{n^2}\by^\sT \bPhi(\bPhi^{\sT}\bPhi/n+\lambda\bI_m)^{-1} \bU (\bPhi^{\sT}\bPhi/n+\lambda\bI_m)^{-1}  \bPhi^\sT \by \, ,\nonumber
\end{align}
where  $\bV\in\reals^m$, $\bU\in\reals^{m\times m}$ are matrices with entries
\begin{align}
  V_i &:= \E_{\bx}\{\sigma(\<\btheta_i,\bx\>/\sqrt{d})\, f^*(\bx)\}\, ,\\
  U_{ij}&:=\E_{\bx}\{\sigma(\<\btheta_i,\bx\>/\sqrt{d})\, \sigma(\<\btheta_j,\bx\>/\sqrt{d})\}\, .
\end{align}
Note that the matrix $\bU$ takes the form of an empirical kernel matrix, although expectation is taken over the covariates
$\bx$ and the kernel is evaluated at the neuron parameters $(\btheta_i)_{i\le m}$.
Namely, we have $U_{ij}=H_{\RF,d}(\<\btheta_i,\btheta_j\>/d)$, where the kernel $H_{\RF,d}$ is defined
exactly\footnote{The two kernels coincide because we are using the same distribution for $\bx_i$ and $\btheta_j$: while this symmetry simplifies some calculations, it is not really crucial.} as in \eqref{eq:HRF}.
Estimates similar to those of Section
\ref{sec:benign} apply here (see also \cite{el2010spectrum}):
since $m\asymp d$ we can approximate the kernel $H_{\RF,d}$  by a linear kernel in operator norm. Namely, if we
decompose $\sigma(x) = \mu_0+\mu_1 x+\sigma_{\perp}(x)$, where
$\E\{\sigma_{\perp}(G)\}=\E\{G\sigma_{\perp}(G)\}=0$,
and $\E \{\sigma_{\perp}(G)^2\}=\mu_*^2$, we have
\begin{align}
  \bU = \mu_0^2\bone\bone^{\sT}+\mu_1^2\bTheta\bTheta^{\sT}+\mu_*\,\id_m+\bDelta\, ,
\end{align}
where $\bDelta$ is an error term that vanishes asymptotically in operator norm. Analogously, $\bV$ can be approximated
as $\bV\approx a\bone+\bTheta\bb$ for suitable coefficients $a\in\reals$, $\bb\in \reals^d$.

Substituting these approximations for $\bU$ and $\bV$  in \eqref{eq:ExplicitRisk} yields an expression of the risk in
terms of the three (correlated) random matrices $\bX$, $\bTheta$, $\bPhi$.
Standard random matrix theory does not apply directly to compute the asymptotics of this expression. The main difficulty is
that the matrix $\bPhi$ does not have independent or nearly independent entries. It is instead obtained by applying a nonlinear function
to a product of matrices with (nearly) independent entries; see \eqref{eq:RF-Phimatrix}. The name
`nonlinear random matrix theory' has been coined to refer to this setting \cite{pennington2017nonlinear}.
Techniques from
random matrix theory have been adapted to this new class of random matrices. In particular, the leave-one-out
method can be used to derive a recursion for the resolvent, as first shown for this
type of matrices in  \cite{cheng2013spectrum}, and the moments method was first
used in \cite{fan2019spectral} (both of these papers consider symmetric random matrices, but these techniques extend to the asymmetric case).
Further results on kernel random matrices can be found in  \cite{do2013spectrum,louart2018random} and \cite{pennington2018spectrum}. 

Using these approaches, the exact asymptotics of   $\Risk_{\RF}(\lambda)$ was determined in
the proportional asymptotics $m,n,d\to\infty$ with $m/d\to \psi_{\wid}$ ( $\psi_{\wid}$
represents the number of neurons per dimension),
$n/d\to \psi_{\samp}$ ($\psi_{\samp}$ represents the number of samples
per dimension).
 The target function $f^*$ is assumed to be square
  integrable and such that $\proj_{>1}f^*$ is a Gaussian isotropic
  function.\footnote{Concretely, for each $\ell\ge 2$, let ${\boldsymbol f}_{\ell}=(f_{k,\ell})_{k\le B(d,\ell)}$ be the coefficients of
    $f^*$ in a basis of degree-$\ell$ spherical harmonics. Then  ${\boldsymbol f}_{\ell} \sim\normal(0,F_{\ell}^2\id_{B(d,\ell})$
    independently across $\ell$.} In this setting,
the risk takes the form
\begin{align}
  \Risk_\RF(\lambda) = &\|\proj_1f^*\|_{L^2}^2 \cuB(\ratio, \psi_{\wid}, \psi_{\samp}, \lambda/\mu_*^2)\\
  &+ (\tau^2 + \|\proj_{>1}f^*\|_{L^2}^2) \cuV(\ratio, \psi_{\wid}, \psi_{\samp}, \lambda/\mu_*^2) + 
\|\proj_{>1}f^*\|_{L^2}^2  +o_d(1)\, ,\nonumber
\end{align}
where $\ratio := |\mu_1|/\mu_*$. The functions $\cuB$, $\cuV\ge 0$ are explicit and
correspond to an effective bias term and an effective variance term. Note the additive term $\|\proj_{>1}f^*\|_{L^2}^2$:
in agreement with Theorem \ref{thm:RFPolynomial}, the nonlinear component of $f^*$ cannot be learnt at all
(recall that $m,n =O(d)$ here). Further $\|\proj_{>1}f^*\|_{L^2}^2$ is added to the noise strength in the `variance' term:
high degree components of $f^*$ are equivalent to white noise at small sample/network size.

The expressions for $\cuB$, $\cuV$ can be used to plot curves such as those in Figure \ref{fig:proportional}:
we refer to \cite{mei2020generalization} for explicit formulas.
As an interesting conceptual consequence, these results establish a universality phenomenon:
the risk under the random features model is asymptotically the same as the risk
of a mathematically simpler model. This  simpler model can be analyzed by
a direct application of standard random matrix theory \cite{hastie2019surprises}.

We refer to the simpler equivalent model as the `noisy features model.' In order to motivate it, recall the decomposition
$\sigma(x) = \mu_0+\mu_1 x+\sigma_{\perp}(x)$ (with the three components being orthogonal in $L^2(\gamma)$).
Accordingly, we decompose the feature matrix as
\begin{align*}
  \bPhi &= \bPhi_{\le 1}+ \bPhi_{>1}\\
        & = \mu_0\bone\bone^{\sT}+\frac{\mu_1}{\sqrt{d}}\bTheta\bX^{\sT}+\mu_*\tbZ\, ,
\end{align*}
where $\tZ_{ij} = \sigma_{\perp}(\<\bx_i,\btheta_j\>/\sqrt{d})/\mu_*$.
Note that the entries of $\tbZ$ have zero mean and are asymptotically
uncorrelated.
Further they are asymptotically uncorrelated with the entries\footnote{Uncorrelatedness holds only asymptotically,
  because the distribution of $\<\bx_i,\btheta_j\>/\sqrt{d}$ is not exactly Gaussian, but only asymptotically so,
  while the decomposition $\sigma(x) = \sigma_0+\sigma_1x+\sigma_{\perp}(x)$ is taken in $L^2(\gamma)$.}    of
$\bTheta\bX^{\sT}/\sqrt{d}$.

As we have seen in Section \ref{sec:RF_Polynomial}, the matrix $\tbZ$ behaves in many ways as a matrix with independent entries,
independent of $\bTheta,\bX$. In particular, if $\max(m,n)\ll d^2$ and either $m\gg n$ or $m\ll n$, its eigenvalues concentrate
around a deterministic value (see discussion below \eqref{eq:PhiDecomposition}).

The noisy features model is obtained by replacing $\tbZ$ with a matrix $\bZ$, with independent entries, independent of
$\bTheta$, $\bX$. Accordingly, we replace the target function with a linear function with additional noise. 
In summary:
\begin{align}
  \bPhi^{\NF} &= \mu_0\bone\bone^{\sT}+\frac{\mu_1}{\sqrt{d}}\bTheta\bX^{\sT}+\mu_*\bZ,\, \,\;\;\; (Z_{ij})_{i\le n, j\le m}\sim \normal(0,1)\, ,\\
  \by & = b_0\bone + \bX\bbeta + \tau_+ \tbg, \, \, \,\;\;\; (\tg_{i})_{i\le n}\sim \normal(0,1)\,  .
\end{align}
Here the random variables $(\tg_i)_{i\le n}, (Z_{ij})_{i\le n, j\le m}$ are mutually independent, and independent of  all the others, and the 
parameters $b_0,\bbeta,\tau_+$  are fixed by the conditions
$\proj_{\le 1}f^*(\bx) = b_0+\<\bbeta,\bx\>$ and
$\tau_+^2=\tau^2+\|\proj_{> 1}f^*\|_{L^2}^2$.

The next statement establishes asymptotic equivalence of the noisy and random features model. 
\begin{theorem}\label{thm:proportional}
  Under the data distribution introduced above, let $\Risk_{\RF}(\lambda)$ denote the risk of ridge regression in the random features model with regularization $\lambda$, and let
  $\Risk_{\NF}(\lambda)$ be the risk in the noisy features model. Then we have, in
  $n,m,d\to\infty$ with $m/d\to\psi_{\wid}$, $n/d\to\psi_{\samp}$, 
  \begin{align}
    \Risk_{\RF}(\lambda) = \Risk_{\NF}(\lambda)\cdot \big(1+o_n(1)\big)\, .
  \end{align}
\end{theorem}

Knowing the exact asymptotics of the risk allows us to identify phenomena
that otherwise would be out of reach. A particularly interesting one is the optimality of interpolation
at high signal-to-noise ratio.
\begin{corollary}
  Define the signal-to-noise ratio of the random features model  as  $\SNR_d:=\|\proj_1 f^*\|_{L^2}^2/(\|\proj_{>1} f^*\|_{L^2}^2+\tau^2)$,
  and let $\Risk_{\RF}(\lambda)$ be the risk of ridge regression with regularization $\lambda$.
  Then there exists a critical value $\SNR_*>0$ such that the following hold.
\begin{itemize}
\item[$(i)$]~If $\lim_{d\to\infty}\SNR_{d}=\SNR_{\infty}> \SNR_*$, then the optimal regularization parameter is $\lambda= 0^+$,
  in the  sense that $\Risk_{\RF,\infty}(\lambda):=\lim_{d\to\infty}\Risk_{\RF}(\lambda)$ is monotone increasing for $\lambda\in(0,\infty)$.
 \item[$(ii)$]~If $\lim_{d\to\infty}\SNR_{d}=\SNR_{\infty}< \SNR_*$, then the optimal regularization parameter is $\lambda > 0$,
   in the  sense that $\Risk_{\RF,\infty}(\lambda):=\lim_{d\to\infty}\Risk_{\RF}(\lambda)$ is monotone decreasing for $\lambda\in(0,\lambda_0)$ with $\lambda_0>0$.
 \end{itemize}
 \end{corollary}
 In other words, above a certain threshold in $\SNR$, (near)
 interpolation is required in order to achieve optimal risk,
 not just optimal rates.

The universality phenomenon of Theorem \ref{thm:proportional} first emerged in random matrix theory studies of (symmetric)
kernel inner product random matrices. In that case, the spectrum of such a random matrix
was shown in \cite{cheng2013spectrum}  to behave asymptotically as the
one of the sum of independent Wishart and Wigner matrices, which correspond respectively to the linear
and nonlinear parts of the kernel 
(see also \cite{fan2019spectral} where this remark is made more explicit). In the context of random features ridge regression, this type
of universality was first pointed out in \cite{hastie2019surprises},
which proved a special case of Theorem \ref{thm:proportional}.
In \cite{goldt2019modelling} and \cite{goldt2020gaussian}, a
universality conjecture was put forward  on the basis of statistical physics arguments and
proved to hold in online learning schemes (that is, if each sample is visited only once).

Universality is conjectured to hold in significantly broader settings than ridge-regularized least-squares.
This is interesting because analysing the noisy feature models is often significantly easier than the original random features model. For instance
\cite{montanari2019generalization} studied max margin classification under the universality hypothesis, and derived an
asymptotic characterization of the test error using Gaussian comparison inequalities.
Related results were obtained by \cite{taheri2020fundamental} and \cite{kini2020analytic}, among others.

Finally, a direct proof of universality for general strongly convex smooth losses  was recently proposed in
\cite{hu2020universality} using the Lindeberg interpolation method.

\subsection{Neural tangent model}
\label{sec:NT}

The neural tangent model  $\cF_{\NT}$  ---recall \eqref{eq:F_NT}---  has not (yet) been studied in as much detail as
the random features model. The fundamental difficulty is related to the fact that the features matrix $\bPhi\in\reals^{n\times md}$
no longer has independent columns:
\begin{align}
 \label{eq:NT-Phmatrix}
  \bPhi:=
  \left[\begin{matrix}
    \sigma'(\<\bx_1,\bw_1\>)\bx_1^{\sT} &   \sigma'(\<\bx_1,\bw_2\>) \bx_1^{\sT} &   \cdots & \sigma(\<\bx_1,\bw_m\>) \bx_1^{\sT} \\
    \sigma'(\<\bx_2,\bw_1\>) \bx_2^{\sT} &   \sigma(\<\bx_2,\bw_2\>) \bx_2^{\sT} &   \cdots & \sigma(\<\bx_2,\bw_m\>) \bx_2^{\sT} \\
    \vdots & \vdots& & \vdots\\
    \sigma'(\<\bx_n,\bw_1\>) \bx_n^{\sT} &   \sigma(\<\bx_n,\bw_2\>) \bx_n^{\sT} &   \cdots & \sigma(\<\bx_n,\bw_m\>)\bx_n^{\sT}\\
  \end{matrix}\right]\, .
\end{align}
Nevertheless, several results are available and point to a common conclusion: the generalization properties of \NT\, are
very similar to those of \RF, provided we keep the number of parameters constant, which amounts
to reducing the number of neurons according to $m_{\NT}d= p_{\NT} = p_{\RF}= m_{\RF}$.

Before discussing rigorous results pointing in this direction, it is important to emphasize that, even if the two models are
statistically equivalent, they can differ from other points of view. In particular, at prediction time both models have complexity
$O(md)$. Indeed, in the case of \RF\,  the most complex operation is the matrix vector multiplication $\bx\mapsto \bW\bx$,
while for \NT\, two such multiplications are needed $\bx\mapsto \bW\bx$ and $\bx\mapsto \bA\bx$ (here $\bA\in\reals^{m\times d}$
is the matrix with rows $(\ba_i)_{i\le m}$. If we keep the same number of parameters (which we can regard
as a proxy for expressivity of the model), we obtain complexity
$O(pd)$ for \RF\, and $O(p)$ for \NT.
Similar considerations apply at training time.
In other words, if we are constrained by computational complexity, in
high dimension \NT\ allows significantly better expressivity. 

A first element confirming this picture is provided by the following result, which partially generalizes
Theorem \ref{thm:RFPolynomial}. In order to state this theorem, we introduce a useful notation. Given a function $f:\reals\to\reals$,
such that $\E\{f(G)^2\}<\infty$, we let $\mu_k(f):=\E\{\He_k(G)f(G)\}$ denote the $k$th coefficient of $f$ in the basis of Hermite polynomials.
\begin{theorem}\label{thm:NTPolynomial}
  Fix an integer $\ell>0$. Let the activation function $\sigma:\reals\to\reals$ be weakly differentiable, independent of $d$, and such that:
  $(i)$~$|\sigma'(x)|\le c_0\exp(c_1x^2/4)$ for some constants $c_0>0$, and $c_1<1$,
  $(ii)$~there exist $k_1$, $k_2\ge 2\ell+7$ such that $\mu_{k_1}(\sigma'),\mu_{k_2}(\sigma')\neq 0$,
  and $\mu_{k_1}(x^2\sigma')/\mu_{k_1}(\sigma') \neq \mu_{k_1}(x^2\sigma')/\mu_{k_1}(\sigma')$, and $(iii)$~$\mu_k(\sigma)\neq 0$
  for all $k\le \ell+1$. Then the following holds.
  
  Assume either  $n=\infty$ (in which case we are considering pure approximation error) or $m=\infty$
(that is, the test error of kernel ridge regression)  and
$d^{\ell+\delta}\le \min(md;n) \le d^{\ell+1-\delta}$ for some constant $\delta>0$. Then, for any $\lambda = o_d(1)$
  and all $\eta>0$,
  \begin{align}
    \Risk_{\NT}(\lambda) = \|\proj_{>\ell} f^*\|_{L^2}^2 +o_d(1) \big(\| f^*\|_{L^2}^2+\tau^2\big)\, .
  \end{align}
\end{theorem}

In this statement we  abused notation in letting $m=\infty$ denote the case of KRR, and
letting $n=\infty$  refer to the approximation error:
  \begin{align}
\lim_{n\to\infty}    \Risk_{\NT}(\lambda) = \inf_{\hf\in \cF_{\NT}^m}\E\big\{[f^*(\bx)-\hf(\bx)]^2\big\}\, .
  \end{align}
  Note that here the \NT\, kernel is a rotationally invariant kernel on $\S^{d-1}(\sqrt{d})$ and hence takes the same form
  as the \RF\, kernel,  namely $K_{\NT}(\bx_1,\bx_2) = d\, H_{\NT,d}(\<\bx_1,\bx_2\>/d)$
  (see \eqref{eq:K_RF}). Hence the $m=\infty$ case of the last theorem is not new: it can be regarded as a special case of Theorem
  \ref{thm:RFPolynomial}.

  On the other hand, the $n=\infty$ portion of the last theorem is new. In words,
  if $d^{\ell+\delta}\le md \le d^{\ell+1-\delta}$, then $\cF^m_\NT$ can approximate degree-$\ell$ polynomials
  to an arbitrarily good relative accuracy, but is roughly orthogonal to polynomials of higher degree
  (more precisely, to polynomials that have vanishing projection onto degree-$\ell$ ones). 
  Apart from the technical assumptions, this result is identical to the $n=\infty$ case of Theorem  \ref{thm:RFPolynomial},
  with the caveat that, as mentioned above, the two models should be compared by keeping the number of parameters
  (not the number of neurons) constant.

  How do \NT\, models behave when both $m$ and $n$ are finite?
  By analogy with the \RF\, model, we would expect that the model undergoes an `interpolation' phase transition
  at $md\approx n$: the test error is bounded away from $0$ for $md\lesssim n$ and can instead vanish for $md\gtrsim n$.
  Note that finding an interpolating function $f\in \cF^m_{\NT}$ amounts to solving the system of linear equations
  $\bPhi\ba=\by$, and hence a solution exists for generic $\by$ if and only if $\rank(\bPhi) = n$.
  Lemma~\ref{lemma:TwoLayers} implies\footnote{To be precise, Lemma \ref{lemma:TwoLayers} assumes
  the covariate vectors $\bx_i\sim\normal(0,\id_d)$.} that this is indeed the case for $md\ge C_0n\log n$ and
  $n\le d^{\ell_0}$ for some constant $\ell_0$ (see \eqref{eq:LemmaLin3}).

  In order to study the test error, it is not sufficient to lower-bound the minimum singular value
  of $\bPhi$, but we need to understand the structure of this matrix: results in this direction were obtained in \cite{montanari2020interpolation}, for $m\le C_0 d$, for some constant $C_0$.  
  Following the same strategy of previous sections,
  we decompose 
  \begin{align}
    \bPhi &= \bPhi_0+\bPhi_{\ge 1},\\
    \bPhi_0& = \mu_1
  \left[\begin{matrix}
 \bx_1^{\sT} &   \bx_1^{\sT} &   \cdots & \bx_1^{\sT} \\
    \bx_2^{\sT} &  \bx_2^{\sT} &   \cdots &\bx_2^{\sT} \\
    \vdots & \vdots& & \vdots\\
    \bx_n^{\sT} & \bx_n^{\sT} &   \cdots & \bx_n^{\sT}\\
  \end{matrix}\right],
  \end{align}
  where $\mu_1:= \E\{\sigma'(G)]$ for $G\sim\normal(0,1)$. The
  empirical kernel matrix
  $\bK = \bPhi\bPhi^{\sT}/m$ then reads
  \begin{align}
    \bK & =\frac{1}{m} \bPhi_0\bPhi_0^{\sT}+\frac{1}{m} \bPhi_0\bPhi_{\ge 1}^{\sT}+\frac{1}{m} \bPhi_{\ge 1}\bPhi_0^{\sT}+\frac{1}{m} \bPhi_{\ge 1}\bPhi_{\ge 1}^{\sT}\\
        & = \mu_1^2\bX\bX^{\sT}+\frac{1}{m} \bPhi_{\ge 1}\bP^{\perp}\bPhi_{\ge 1}^{\sT}+\bDelta\, .
          \label{eq:DecompositionNTK}
  \end{align}
  Here $\bP\in\reals^{md\times md}$ is a block-diagonal projector, with $m$ blocks of dimension $d\times d$,
  with $\ell$th block given by 
  $$\bP_{\ell}:=\bw_{\ell}\bw^{\sT}_{\ell},~ \bP^{\perp} = \id_{md}-\bP ~\text{and}~
  \bDelta := (\bPhi_0\bPhi_{\ge 1}^{\sT}+\bPhi_{\ge 1}\bPhi_0^{\sT}+ \bPhi_{\ge 1}\bP\bPhi_{\ge 1}^{\sT})/m.$$
   For the diagonal entries we have (assuming for simplicity
   $\bx_i\sim\normal(0,\id_d)$), 
  \begin{align*}
    \E\Big\{\frac{1}{m} \big(\bPhi_{\ge 1}\bP^{\perp}\bPhi_{\ge 1}^{\sT}\big)_{ii}\Big\} &=
                                                                                           \E\big\{\<\bx_i,(\id_d-\bP_{\ell})\bx_i\>
                                                                                           (\sigma'(\<\bw_{\ell},\bx_i\>)-\mu_1)^2\big\}\\
                                                                                         &=\E\big\{\<\bx_i,(\id_d-\bP_{\ell})\bx_i\> \big\}\,\E\big\{ \sigma'(\<\bw_{\ell},\bx_i\>)-\mu_1)^2\big\}\\
                                                                                         & = (d-1)\E\{(\sigma'(G)-\E\sigma'(G))^2\}=: (d-1)v(\sigma),
  \end{align*}
  where the second equality follows because $(\id_d-\bP_{\ell})\bx_i$ and $\<\bw_{\ell},\bx_i\>$ are independent
  for  $\bx_i\sim\normal(0,\id_d)$, and the last expectation is with respect to $G\sim\normal(0,1)$.
  As proved in \cite{montanari2020interpolation} the matrix 
  $\bPhi_{\ge 1}\bP^{\perp}\bPhi_{\ge 1}^{\sT}$ is well approximated by this diagonal expectation.
  Namely, under the model above,  there exists a constant $C$ such that, with high probability:
  \begin{align}
    \Big\|\frac{1}{md} \bPhi_{\ge 1}\bP^{\perp}\bPhi_{\ge 1}^{\sT}-v(\sigma)\, \id_n\Big\|_{\op}\le
    \sqrt{\frac{n(\log d)^C}{md}}\, .\label{eq:ConcentrationNTK}
  \end{align}

  Equations \eqref{eq:DecompositionNTK} and \eqref{eq:ConcentrationNTK} suggest that for $m=O(d)$,
  ridge regression in the
  NT model can be approximated by ridge regression in the raw covariates, as long as the regularization parameter
  is suitably modified. 
  The next theorem confirms this intuition \cite{montanari2020interpolation}. We define
  ridge regression with respect to the raw covariates as per 
  \begin{align}
   \hbbeta(\gamma):= \argmin{}\Big\{\frac{1}{d}\big\|\by-\bX\bbeta\big\|_2^2+\gamma\|\bbeta\|_2^2\Big\}.
  \end{align}
  
  \begin{theorem}
    Assume $d^{1/C_0}\le m\le C_0d$, $n\ge d/C_0$  and $md\gg n$. Then with high probability there exists an interpolator. Further assume $\bx_i\sim\normal(0,\id_d)$ and  $f^*(\bx) = \<\bbeta_*,\bx\>$.
    Let $$\Risk_{\lin}(\gamma):=\E\{(f^*(\bx)-\<\hbbeta(\gamma),\bx\>)^2\}$$ denote the risk of ridge regression with respect to the raw features.
    
   Set $\lambda=\lambda_0(md/n)$ for some $\lambda_0\ge 0$. Then there exists a constant $C>0$ such that, with high probability, 
    \begin{align}
      \Risk_{\NT}(\lambda) = \Risk_{\lin}(\gamma_{\seff}(\lambda_0,\sigma)) + O\left(\sqrt{\frac{n(\log d)^C}{md}}\right),
    \end{align}
    where $\gamma_{\seff}(\lambda_0,\sigma):=(\lambda_0+v(\sigma))/\E\{\sigma'(G)\}^2$.
  \end{theorem}
  Notice that the shift in regularization parameter matches the
heuristics given above (the scaling in $\lambda=\lambda_0(md/n)$ is introduced to match the typical scale of $\bPhi$).

%% file: future.tex
\section{Conclusions and future directions}
\label{sec:future}

Classical statistical learning theory establishes guarantees on the
performance of a  statistical estimator $\hf$, by bounding the generalization error $\Risk(\hf) - \hRisk(\hf)$.
This is often thought of as a small quantity compared to the training
error $\Risk(\hf) - \hRisk(\hf)\ll \hRisk(\hf)$.
Regularization
methods are designed precisely with the aim of keeping  the generalization error $\Risk(\hf) - \hRisk(\hf) $ small.

The effort to understand deep learning has recently led to the discovery of a different learning
scenario, in which the test error $\Risk(\hf)$ is optimal or nearly optimal, despite being much larger than the training error.
Indeed in deep learning the training error often  vanishes or is extremely small. The model
is so rich that it overfits the data,
that is, $\hRisk(\hf)\ll \inf_f\Risk(f)$.
When pushed,
gradient-based training leads to interpolation or near-interpolation   $\hRisk(\hf)\approx 0$ \cite{zhang2016understanding}. We
regard this as a particularly illuminating limit case.

This behavior is especially puzzling from a statistical point of view, that is, if we view data
$(\bx_i,y_i)$
as inherently noisy. In this case $y_i-f^*(\bx_i)$ is of the order of the noise level and therefore,  for
a model that interpolates, $\hf(\bx_i)-f^*(\bx_i)$ is also large.
Despite this, near-optimal test error means that
$\hf(\bx_{\stest})-f^*(\bx_{\stest})$
must be small   at `most'  test points $\bx_{\stest}\sim
\P$.

As pointed out in Section \ref{sec:slt}, interpolation poses less of a conceptual problem if data are noiseless.
Indeed, unlike the noisy case, we can exhibit at least one interpolating solution that has vanishing test error, for any sample size:
the true function $f^*$. 
Stronger results can also be established in the noiseless case:
\cite{feldman2020does} proved that interpolation is necessary to achieve optimal
error rates when the data distribution is heavy-tailed in a suitable sense.

In this review we have focused on understanding when and why interpolation can be optimal or nearly optimal
even with noisy data.  Rigorous  work has largely focused on models that are linear in a certain feature space, with the 
featurization map being independent of the data. Examples are RKHSs, the features produced by random network layers,
or the neural tangent features defined by the Jacobian of the network at initialization.
Mathematical work has established that interpolation can indeed
be optimal and has described the underlying mechanism in a number of settings. While the scope of this analysis
might appear to be limited (neural networks are notoriously nonlinear in their parameters), it is relevant to deep learning in two ways.
First, in a direct way: as explained in Section \ref{sec:efficient}, there are training regimes in which an overparametrized neural network
is well approximated by a linear model that corresponds to the first-order Taylor expansion of the network around its initialization
(the `neural tangent' model). Second, in an indirect way: insights and hypotheses arising from the analysis of linear models can provide useful guidance for studying more complex settings.

Based on the work presented in this review, we can distill a few
insights worthy of exploration in broader contexts.

\vspace{0.2cm}

\noindent{\bf Simple-plus-spiky decomposition.} The function learnt in the 
overfitting (interpolating) regime takes the form
  \begin{align}
    \hf (\bx) = \hf_0(\bx)+\Delta(\bx)\, .\label{eq:Decomposition}
  \end{align}
  Here $\hf_0$ is simple in a suitable sense (for instance, it is smooth) and hence
  is far from interpolating the data, while $\Delta$ is spiky:
  it has large complexity and allows interpolation of the data, but it is small, in the sense that it has negligible effect on the test error, \emph{i.e.} 
  $\Risk(\hf_0+\Delta)\approx \Risk(\hf_0)$.

  In the case of linear models, the decomposition \eqref{eq:Decomposition} corresponds to
  a decomposition of $\hf$ into two orthogonal subspaces that do not depend on the data. Namely,  $\hf_0$
  is the projection of $\hf$ onto the top eigenvectors of the associated kernel and $\Delta$ is its orthogonal complement.
  In nonlinear models, the two components need not be orthogonal and the associated subspaces are likely to be data-dependent.

  Understanding whether such a decomposition  is possible, and what is its nature is a wide-open problem, which could be investigated
  both empirically and mathematically. A related question is whether the decomposition  \eqref{eq:Decomposition}
  is related to the widely observed `compressibility' of neural network models. This is the observation that the test error of deep
  learning models does not change significantly if ---after training--- the model is simplified by a suitable compression operation \cite{han2015deep}.

\vspace{0.2cm}

\noindent{\bf Implicit regularization.}
Not all interpolating models generalize equally well. This is easily seen in the case of
  linear models, where the set of interpolating models forms an affine space of dimension $p-n$ (where $p$ is the number of
  parameters).
  Among these, we can find models of arbitrarily large norm, that are arbitrarily far from the target regression function.
  Gradient-based training selects a specific model in this subspace, which is the closest in $\ell_2$ norm to the initialization.  

  The  mechanism by which the training algorithm selects a specific empirical risk minimizer is understood in only a handful of cases: we refer to Section~\ref{sec:implicit} for pointers to this literature.
  It would be important to understand how the model nonlinearity interacts with gradient flow dynamics.
  This  in turn impacts the decomposition \eqref{eq:Decomposition}, namely which part of the
  function $\hf$ is to be considered `simple' and which one is
  `spiky'. Finally, the examples of kernel machines, random features
  and neural tangent models show that---in certain regimes---the simple component $\hf_0$ is also regularized in a non-trivial way,
  a phenomenon that we called self-induced regularization. Understanding these mechanisms in a more general setting is an outstanding challenge.

  \vspace{0.2cm}
  
\noindent{\bf Role of dimension.}
As pointed out in Section \ref{sec:benign}, interpolation is
sub-optimal in a fixed dimension in the presence of noise, for certain kernel methods \cite{rakhlin2019consistency}.
  The underlying mechanism is as described above: for an interpolating model, $\hf(\bx_i) -f^*(\bx_i)$
  is of the order of the noise level. If $\hf$ and $f^*$ are sufficiently regular (for instance, uniformly continuous, both in $\bx$ and in $n$)
  $\hf(\bx_{\stest}) -f^*(\bx_{\stest})$ is expected to be of the same order
  when $\bx_{\stest}$ is close to the training set. This happens with constant probability in fixed dimension.
  However, this probability decays rapidly with the dimension.

  Typical data in deep learning applications are  high-dimensional (images, text, and so on). On the other hand,
  it is reasonable to believe that deep learning methods are not affected by the ambient dimension (the number of pixels
  in an image), but rather by an effective or intrinsic dimension. This is the case for random feature models \cite{ghorbani2020neural}.
  This raises the question of how deep learning methods escape the intrinsic limitations of interpolators in low  dimension.
  Is it because they construct a (near) interpolant $\hf$ that is highly irregular (not  uniformly continuous)?
  Or perhaps because the effective dimension is at least moderately large?  (After all the lower bounds mentioned above
  decrease rapidly with dimension.) What is the proper mathematical definition of effective dimension?

  \vspace{0.2cm}
  
\noindent{\bf Adaptive model complexity.} As mentioned above, in the case of linear models, the terms $\hf_0$ and $\Delta$ in the
  decomposition \eqref{eq:Decomposition} correspond to the projections of $\hf$ onto $\cV_k$ and $\cV_k^\perp$.
  Here $\cV_k$ is the space spanned by the  top $k$ eigenfunctions  of
  the kernel associated with the linear regression problem. 
  Note that this is the case also for the random features and neural tangent models of Section \ref{sec:NTK}. In this case the
  relevant kernel is the expectation of the finite-network kernel $\bD f(\btheta_0)^{\sT}\bD f(\btheta_0)$
  with respect to the choice of random weights at initialization.

  A crucial element of this behavior is the dependence of $k$ (the dimension of the eigenspace  $\cV_k$) on  various features of the
  problem at hand: indeed $k$ governs the complexity of the `simple' part of the model $\hf_0$, which is the one actually relevant for prediction. As discussed in Section \ref{sec:benign}, in kernel methods $k$ increases with the sample size $n$: as more data are used,
  the model $\hf_0$ becomes more complex. In random features
  and  neural tangent models (see Section~\ref{sec:NTK}),
  $k$ depends on the minimum of $n$ and the number of network parameters (which is proportional to the width for two-layer
  networks). The model complexity increases with sample size, but saturates when it reaches the number of network parameters.

  This suggests a general hypothesis that would be interesting to investigate beyond linear models. Namely, if a decomposition of the
  type \eqref{eq:Decomposition} is possible, then the complexity of
  the simple part $\hf_0$ increases with the sample size
  and the network size.

  \vspace{0.2cm}
  
\noindent{\bf Computational role of overparametrization.} We largely focused on the surprising discovery that overparametrization and
  interpolation do not necessarily hurt generalization, even in the presence of noise. However, we should emphasize once more that 
  the real motivation for working with overparametrized models is not statistical but computational. The empirical risk minimization problem for neural
  networks is computationally hard, and in general we cannot hope to
  be able to find a global minimizer using gradient-based
  algorithms. However, empirical evidence indicates that global
  optimization becomes tractable when the model is sufficiently overparametrized.

  The linearized and mean field theories of Section \ref{sec:efficient} provide general arguments to confirm this empirical finding.
  However, we are far from understanding precisely what amount of overparametrization is necessary, even in simple neural network models.

%% file: appendix_body.tex
\section{Kernels on $\Re^d$ with $d\asymp n$}
\subsection{Bound on the variance of the minimum-norm interpolant}

\begin{lemma}
	For any $\bX\in\Re^{n\times d}$ and any positive semidefinite $\bSigma\in\Re^{d\times d}$, for $n\lesssim d$ and any $k< d$,
	\begin{align}
		\trace \left((\bX\bX^\tr + d\gamma \bI_n)^{-2}\bX\bSigma\bX^\tr \right) \lesssim \frac{1}{\gamma} \left(\frac{\lambda_1 k}{n} + \lambda_{k+1}\right),
	\end{align}
	where $\lambda_1\geq \ldots  \geq \lambda_d$ are the eigenvalues of $\bSigma$.
\end{lemma}
\begin{proof}
	This deterministic argument is due to T. Liang~\cite{liang2020amend}. We write $\bSigma=\bSigma_{\leq k}+\bSigma_{> k}$, with $\bSigma_{\leq k} = \sum_{i\leq k} \lambda_i \bu_i \bu_i^\tr$. Then by the argument in \cite[Remark 5.1]{liang2020just}, 
 	\begin{align}
 		\trace \left((\bX\bX^\tr + d\gamma \bI_n)^{-2}\bX\bSigma_{>k}\bX^\tr \right) \leq \lambda_{k+1} \sum_{i=1}^n \frac{\widehat{\lambda}_i}{(d\gamma + \widehat{\lambda}_i)^2} \leq \lambda_{k+1} \frac{n}{4d\gamma} \lesssim \frac{\lambda_{k+1}}{\gamma}
 	\end{align}
	where $\widehat{\lambda}_i$ are the eigenvalues of $\bX\bX^\tr$. Here we use the fact that $\frac{t}{(r+t)^2}\leq \frac{1}{4r}$ for all $t,r>0$. On the other hand,
 	\begin{align}
 		\trace \left((\bX\bX^\tr + d\gamma \bI_n)^{-2}\bX\bSigma_{\leq k}\bX^\tr \right) \leq \sum_{i\leq k} \lambda_i \norm{(d\gamma \boldsymbol{I}_n+\bX\bX^\tr)^{-1} \bX \bu_i}^2\, .\label{eq:BD_Sasha}
 	\end{align}
	Now, using the argument similar to that in \cite{bartlett2020benign}, we define $A_{-i} = d\gamma \boldsymbol{I}_n + \bX(\boldsymbol{I}_n-\bu_i\bu_i^\tr)\bX^\tr$, $\bv=\bX\bu_i$ and write
	\begin{align}
		\norm{(d\gamma \boldsymbol{I}_n+\bX\bX^\tr)^{-1} \bX \bu_i}^2 = \norm{(A_{-i}+\bv\bv^\tr)^{-1}\bv}^2 = \frac{\bv^\tr A_{-i}^{-2}\bv}{(1+\bv^\tr A_{-i}^{-1}\bv)^2} 
	\end{align}
	by the Sherman-Morrison formula. The last quantity is upper bounded by
	\begin{align}
		\frac{1}{d\gamma}\frac{\bv^\tr A_{-i}^{-1}\bv}{(1+\bv^\tr A_{-i}^{-1}\bv)^2} \leq \frac{1}{4\gamma d}.
	\end{align}
	Substituting in \eqref{eq:BD_Sasha}, we obtain an upper bound of 
	$$\frac{1}{4\gamma d}\sum_{i\leq k} \lambda_i \lesssim \frac{\lambda_1 k}{\gamma n},$$
	assuming $n\lesssim d$.
\end{proof}


 \subsection{Exact characterization in the proportional asymptotics}

     We will denote by $\bK= (h(\<\bx_i,\bx_j\>/d))_{i,j\le n}$ the
     kernel matrix. We will also denote by $\bK_1$ the linearized kernel 
     \begin{align}
       	\bK_1& = \beta \frac{\bX\bX^\tr}{d} + \beta \gamma \bI_n + \alpha \bone\bone^\tr\, ,\label{eq:K1def}\\
  \alpha &:= h(0)+ h''(0)\frac{\trace(\bSigma^2)}{2d^2},~ \beta := h'(0),\\
  \gamma &:= \frac{1}{h'(0)}\big[h(\trace(\bSigma)/d) - h(0)-h'(0)\trace(\bSigma/d)\big].
     \end{align}

 {	  
 \renewcommand{\thetheorem}{\ref{asmpt:kernel_nd_regime}}	  	 
 \begin{assumption}
      We assume that the coordinates of $\bz=\bSigma^{-1/2}\bx$ are
        independent, with zero mean and unit variance, so that
        $\bSigma=\En \bx\bx^\tr$.
        Further assume there are constants
        $0<\eta,M<\infty$, such that the following hold.
        \begin{enumerate}
        \item[$(a)$] For all $i\le d$,
          $\E[|\bz_i|^{8+\eta}]\le M$.
        \item[$(b)$] $\|\bSigma\| \le M$, $d^{-1}\sum_{i=1}^d\lambda_i^{-1}\le M$, where $\lambda_1,\dots,\lambda_d$ are the eigenvalues of $\bSigma$.
        \end{enumerate}
      \end{assumption}
  \addtocounter{theorem}{-1} 	
  }

 {	  
 \renewcommand{\thetheorem}{\ref{thm:npropd}}	  
 \begin{theorem}
   Let $0< M,\eta<\infty$ be fixed constants and suppose that
  Assumption~\ref{asmpt:kernel_nd_regime} holds with $M^{-1}\le d/n\le
  M$. Further assume that  $h$ is continuous on $\reals$ and smooth in a
  neighborhood of $0$ with $h(0)$, $h'(0)>0$, that $\|f^*\|_{L^{4+\eta}(\P)}\le M$ and that the $z_i$'s are $M$-sub-Gaussian.
  Let $y_i=f^*(\bx_i)+\xi_i$, $\E(\xi_i^2)= \sigma_{\xi}^2$,
  and $\bbeta_0 := \bSigma^{-1}\E[\bx f^*(\bx)]$.
  Let $\lambda_*>0$ be the unique positive solution of
        \begin{align}
          n \Big(1-\frac{\gamma}{\lambda_*}\Big) = \Trace\Big( \bSigma(\bSigma+\lambda_*\id)^{-1}\Big)\, .\label{eq:FixedPointProportional-2}
        \end{align}
        Define $\cuB(\bSigma,\bbeta_0)$ and $\cuV(\bSigma)$ by
        \begin{align}
          \cuV(\bSigma) &:=\frac{\Trace\big( \bSigma^2(\bSigma+\lambda_*\id)^{-2}\big)}{n-\Trace\big( \bSigma^2(\bSigma+\lambda_*\id)^{-2}\big)}
                          \, , \label{eq:VarianceProportional-2}\\
          \cuB(\bSigma,\bbeta_0) &:= \frac{\lambda_*^2\<\bbeta_0,(\bSigma+\lambda_*\id)^{-2}\bSigma\bbeta_0\>}
           {1-n^{-1} \Trace\big(
           \bSigma^2(\bSigma+\lambda_*\id)^{-2}\big)}\,
           .\label{eq:BiasProportional-2}
        \end{align}
        Finally, let $\biassquaredemp$ and  $\varianceemp$ denote the
        squared bias and variance for the minimum-norm interpolant. 
	Then there exist  $C,c_0>0$  (depending also on the constants in Assumption~\ref{asmpt:kernel_nd_regime})
        such that the following holds with probability at least $1-Cn^{-1/4}$
        (here $\proj_{>1}$ denotes the projector orthogonal to affine
        functions in $L^2(\P)$):
	\begin{align}
          \big|\biassquaredemp -
          \cuB(\bSigma,\bbeta_0)-\|\proj_{>1}f^*\|_{L^2}^2
           (1+\cuV(\Sigma)) \big|&\le C n^{-c_0}\, , \label{eq:bias_kernel_linear_precise-2}\\
           \big|	\varianceemp-  \sigma_{\xi}^2\cuV(\bSigma) \big|&\le C n^{-c_0}\, .\label{eq:var_kernel_linear_precise-2}
	\end{align}
      \end{theorem}
	  
	  \addtocounter{theorem}{-1} 	
	  }								

      \begin{newremark}
        The result for the variance will be proved under weaker
        assumptions and in a stronger form than stated. In particular, it does not require any assumption on the target function $f_*$, and it holds with smaller error terms than stated.
      \end{newremark}

       \begin{newremark}
         Notice that by positive definiteness of the kernel, we have $h'(0),h''(0)\ge 0$. Hence the conditions that these are strictly positive
        is essentially a non-degeneracy requirement.
      \end{newremark}

      We note for future reference that the target function $f^*$ is decomposed as 
      \begin{align}
        f^*(\bx) = b_0+\<\bbeta_0,\bx\>+\proj_{>1}f^*(\bx)\, ,
      \end{align}
      where $b_0:=\E\{f^*(\bx)\}$,   $\bbeta_0 := \bSigma^{-1}\E[\bx f^*(\bx)]$ as defined above and $\E\{\proj_{>1}f^*(\bx)\}$,
      $\E\{\bx\proj_{>1}f^*(\bx)\}=\bfzero$.

     \subsubsection{Preliminaries}

     Throughout the proof, we will use $C$ for constants that depend uniquely on the constants in Assumption
     \ref{asmpt:kernel_nd_regime} and Theorem \ref{thm:npropd}. We also write that an inequality holds \emph{with very high probability}
     if, for any $A>0$, we can choose the constants $C$ in the inequality such that this holds with probability at least $1-n^{-A}$
     for all $A$ large enough.
     
  We will repeatedly use the following bound, see e.g. \cite{el2010spectrum}.
     \begin{lemma}\label{lemma:ApproxK}
       Under the assumptions of Theorem \ref{thm:npropd}, we have, with very high probability
       \begin{align}
\bK=\bK_1+\bDelta\, ,\;\;\;\;\; \|\bDelta\|_{\op}\le n^{-c_0}\, .
       \end{align}
       In particular, as long as $h$ is non-linear, we have $\bK\succeq  c_*\id_n$, $c_*=\beta\gamma>0$ with probability at least $1-Cn^{-D}$.
     \end{lemma}
     
     Define the matrix $\bM\in\reals^{n\times n}$, and the vector $\bv\in\reals^n$ by
     \begin{align}
      M_{ij}&:=\E_{\bx}\left\{h\Big(\frac{1}{d}\<\bx_i,\bx\>\Big)
                h\Big(\frac{1}{d}\<\bx_j,\bx\>\Big)\right\} \, ,\label{eq:Mdef}\\
        v_i&:=\E_{\bx}\left\{h\Big(\frac{1}{d}\<\bx_i,\bx\>\Big) f^*(\bx)\right\}\, .\label{eq:Vdef}
     \end{align}
     Our first lemma provides useful approximations of these quantities.
     \begin{lemma}\label{lemma:ApproxMV}
       Define (here expectations are over $G\sim\normal(0,1)$):
       \begin{align}
         \bv_0&:= \ba_0 b_0+\frac{1}{d} h'(0)\bX\bSigma\bbeta_0\, ,\\
                a_{i,0} &:= \E\Big\{h\Big(\sqrt{\frac{Q_{ii}}{d}} \, G\Big)\Big\}\, , \;\;\;\; Q_{ij}: =\frac{1}{d}\<\bx_i,\bSigma\bx_j\>\, .
       \end{align}
       and
       \begin{align}
         \bM_0&:= \ba\ba^{\tr}+\bB \, ,\;\;\;\;\;\; \bB:= \frac{1}{d}\,\bD\bQ\bD\, , \label{eq:M0def1}\\
         a_i &:= a_{i,0}+a_{i,1}\, ,\;\;\;\; a_{i,1} = \frac{1}{6} \Big(\frac{Q_{ii}}{d}\Big)^{3/2} h^{(3)}(0)
               \sum_{j=1}^d\frac{(\bSigma^{1/2} \bx_i)_j^3}{\|\bSigma^{1/2} \bx_i\|_2^3}\E(z_j^3)\, ,\label{eq:A1def}\\
         \bD &:={\rm diag}(D_1,\dots,D_n)\, ,\;\;\;\;\;\; D_i:= \E \Big\{h'\Big(\sqrt{\frac{Q_{ii}}{d}} \, G\Big)\Big\}\label{eq:M0def2}
     \, .
       \end{align}
       Then the following hold with very high probability
       (in other words, for any $A>0$ there exists $C$ such that the following hold  with probability at least $1-n^{-A}$ for all $n$ large enough)
       \begin{align}
           \max_{i\le n}\big|v_i-v_{0,i}\big| & \le  C\frac{\sqrt{\log d}}{d^{3/2}}\, ,\label{eq:Vbound}\\
         \max_{i\neq j\le n}\big|M_{ij}-M_{0,ij}\big| & \le  C\frac{\log d}{d^{5/2}}\, , \label{eq:MBound1}\\
           \max_{i\le n}\big|M_{ii}-M_{0,ii}\big|& \le  C\frac{\log d}{d^{2}}\, . \label{eq:MBound2}
       \end{align}
       In particular, this implies $\|\bv-\bv_0\|_2\le Cd^{-1}\sqrt{\log d}$, $\|\bM-\bM_0\|_F\le C d^{-3/2}\log d$.
       \end{lemma}
       \begin{proof}
         Throughout the proof we will work on the intersection $\cE_1\cap\cE_2$
         of following events, which hold with very high probability by standard concentration arguments. These events are defined by 
        \begin{align}
          \cE_1& := \Big\{C^{-1}\le \frac{1}{\sqrt{d}}\|\bSigma\bz_i\|_2\le C ; \;\; \frac{1}{\sqrt{d}}\|\bSigma\bz_i\|_\infty\le C \sqrt{\frac{\log d}{d}} \; \;
          \forall i\le n\Big\}\label{eq:Event1}\\
          &= \Big\{C^{-2}\le \frac{1}{d}\<\bx_i,\bSigma\bx_i\>\le C^2 ; \;\; \frac{1}{\sqrt{d}}\|\bSigma^{1/2}\bx_i\|_\infty\le C \sqrt{\frac{\log d}{d}} \; \;
          \forall i\le n\Big\}\, ,
        \end{align}
        and
        \begin{align}
               \cE_2& := \Big\{ \frac{1}{d}\sum_{\ell=1}^d (\bSigma\bz_i)_\ell (\bSigma\bz_j)^2_\ell \le \frac{\log d}{d^{1/2}}\,;\;\;
                      \frac{1}{d}|\<\bz_i,\bSigma\bz_j\>|\le C \sqrt{\frac{\log d}{d}} \,;\;\;
                      \frac{1}{d}|\<\bz_i,\bSigma^2\bz_j\>|\le C \sqrt{\frac{\log d}{d}}\;\; \forall i\neq j\le n\Big\}
                      \label{eq:Event2}\\
          &= \Big\{ \frac{1}{d}\sum_{\ell=1}^d (\bSigma^{1/2}\bx_i)_\ell (\bSigma^{1/2}\bx_j)^2_\ell \le \frac{\log d}{d^{1/2}}\, ; \;\;
            \frac{1}{d}|\<\bx_i,\bx_j\>|\le C \sqrt{\frac{\log d}{d}}\, \;\;  \frac{1}{d}|\<\bx_i,\bSigma\bx_j\>|\le C \sqrt{\frac{\log d}{d}}\;\;
            \forall i\neq j\le n\Big\}\, .\nonumber
        \end{align} 
        Recall that, by assumption, $h$ is smooth on an interval $[-t_0,t_0]$, $t_0>0$.
        On the event $\cE_2$,
        we have $\<\bx_i,\bx_j\>/d\in[-t_0,t_0]$ for all $i\neq j$.
        If $h$ is not smooth everywhere, we
        can always modify it outside $[-t_0/2,t_0/2]$ to obtain a kernel $\tih$ that is smooth everywhere.
        Since $\bx$ is sub-Gaussian, as long as $\|\bx_i\|/\sqrt{d}\le C$ for all $i\le n$ (this happens on $\cE_1$)
        we have (for $\bx\sim\P$), $\<\bx_i,\bx\>/d\in [-t_0/2,t_0/2]$ with probability at least $1-e^{-d/C}$. Further using the
        fact that $f$ is bounded in Eqs.~\eqref{eq:Mdef}, \eqref{eq:Vdef}, we get, 
        \begin{align}
              M_{ij}&:=\E_{\bx}\left\{\tih\Big(\frac{1}{d}\<\bx_i,\bx\>\Big)
                \tih\Big(\frac{1}{d}\<\bx_j,\bx\>\Big)\right\} +O(e^{-d/C})\, ,\\
        v_i&:=\E_{\bx}\left\{\tih\Big(\frac{1}{d}\<\bx_i,\bx\>\Big) f^*(\bx)\right\} +O(e^{-d/C})\, ,
        \end{align}
        where the term $O(e^{-d/C})$ is uniform over $i,j\le n$. Analogously, in the definition of $\bv_0$, $\bM_0$
        (more precisely, in defining $\ba_0$, $\bD$), we can replace $h$ by $\tih$ at   the price of an $O(e^{-d/C})$ error.
        Since these error terms are negligible as compared to the ones in the statement, we shall hereafter neglect them
        and set $\tih =h$ (which corresponds to
        defining arbitrarily the derivatives of $h$ outside a neighborhood of $0$).

        We denote by $h_{i,k}$ the $k$-th coefficient of $h((Q_{ii}/d)^{1/2} x)$ in the basis of Hermite polynomials.
        Namely:
        \begin{align}
          h_{i,k} =  \E \Big\{h\Big(\sqrt{\frac{Q_{ii}}{d}} \, G\Big)\, \He_k(G)\Big\} = \Big(\frac{Q_{ii}}{d}\Big)^{k/2}
           \E \Big\{h^{(k)}\Big(\sqrt{\frac{Q_{ii}}{d}} \, G\Big)\Big\}\, .\label{eq:hik}
        \end{align}
        Here $h^{(k)}$ denotes the $k$-th derivative of $h$ (recall that by the argument above we can assume, without loss of generality, that
        $h$ is $k$-times differentiable).

        We write $h_{i,>k}$ for the remainder after the first $k$
        terms of the Hermite expansion have been removed:
        \begin{align}
          h_{i,>k}\Big(\sqrt{\frac{Q_{ii}}{d}} \, x\Big) &:= h\Big(\sqrt{\frac{Q_{ii}}{d}} \, x\Big) -\sum_{\ell=0}^k \frac{1}{\ell!} h_{i,\ell}\, \He_\ell(x)
          \label{eq:FirstHermite}\\
                                                         &= h\Big(\sqrt{\frac{Q_{ii}}{d}} \, x\Big) -\sum_{\ell=0}^k \frac{1}{\ell!}
                                                             \Big(\frac{Q_{ii}}{d}\Big)^{\ell/2} \E \Big\{h^{(k)}\Big(\sqrt{\frac{Q_{ii}}{d}} \, G\Big)\Big\} \, \He_\ell(x)
          \, .\nonumber
        \end{align}
        Finally, we denote by $\hh_{i,>k}(x)$ the  remainder after the first $k$ terms in the Taylor expansion have been subtracted:
        \begin{align}
          \hh_{>k}(x) &:= h( x) -\sum_{\ell=0}^k \frac{1}{\ell!}
                                                     h^{(\ell)}(0)\, x^{\ell}\, .
        \end{align}
        Of course $h-\hh_{>k}$ is a polynomial of degree $k$, and therefore its projection orthogonal to the first $k$
        Hermite polynomials vanishes, whence
        \begin{align}
          h_{i,>k}\Big(\sqrt{\frac{Q_{ii}}{d}} \, x\Big) &= \hh_{>k}\Big(\sqrt{\frac{Q_{ii}}{d}} \, x\Big) -\sum_{\ell=0}^k \frac{1}{\ell!}
                                                             \Big(\frac{Q_{ii}}{d}\Big)^{\ell/2} \E \Big\{\hh_{>k}^{(\ell)}\Big(\sqrt{\frac{Q_{ii}}{d}} \, G\Big)\Big\} \, \He_\ell(x)
          \, .\label{eq:RemainderTaylor}
        \end{align}
        Note that, by smoothness of $h$, we have $|\hh_{>k}^{(\ell)}(t)|\le C\min(|t|^{k+1-\ell},1)$, and therefore
        \begin{align}
\left|\frac{1}{\ell!}
          \Big(\frac{Q_{ii}}{d}\Big)^{\ell/2} \E \Big\{\hh_{>k}^{(\ell)}\Big(\sqrt{\frac{Q_{ii}}{d}} \, G\Big)\Big\} \right|\le C d^{-(k+1)/2}\, .
        \end{align}
       We also have that  $|\hh_{>k}(t)|\le C\min(1,|t|^{k+1})$.
        Define $\bv_i=\bSigma^{1/2}\bx_i/\sqrt{d}$, $\|\bv_i\|_2^2=Q_{ii}$. For any fixed $m\ge 2$,
        by Eq.~\eqref{eq:RemainderTaylor} and the triangle inequality,
        \begin{align}
          \E_{\bz}\Big\{\Big|h_{i,>k}\Big(\frac{1}{\sqrt{d}}\<\bv_i,\bz\>\Big)\Big|^m\Big\}^{1/m}&\stackrel{(a)}{\le}
                                                                                                   \E\Big\{\Big|\hh_{>k}\Big(\frac{1}{\sqrt{d}}\<\bv_i,\bz\>\Big)\Big|^m\Big\}^{1/m} +Cd^{-(k+1)/2}
                                                                                                   \sum_{\ell=0}^k \E\Big\{\Big|\He_{\ell}\Big(\frac{\<\bv_i,\bz\>}{\|\bv_i\|_2}
                                                                                                   \Big)\Big|^m\Big\}^{1/m}\nonumber\\
          &\le
          C\, \Big(\frac{Q_{ii}}{d}\Big)^{(k+1)/2} +C\, d^{-(k+1)/2}\le C\, d^{-(k+1)/2}\, ,\label{eq:UB_HighDeg}
        \end{align}
        where the inequality $(a)$ follows since $\<\bv_i,\bz\>$ is
        $C$-sub-Gaussian.
        Note that Eqs.~\eqref{eq:FirstHermite}, \eqref{eq:UB_HighDeg}  can also be rewritten as
        \begin{align}
          h \Big(\frac{1}{d}\<\bx_i,\bx\>\Big)&=\sum_{\ell=0}^k\frac{1}{\ell!} h_{i,\ell}\,\He_{\ell}   \Big(\frac{1}{\sqrt{d Q_{ii}}}\<\bx_i,\bx\>\Big)
          +h_{i,>k} \Big(\frac{1}{d}\<\bx_i,\bx\>\Big)\, , \label{eq:HermiteApprox1}\\
                                              \E\Big\{\Big| h_{i,>k} \Big(\frac{1}{d}\<\bx_i,\bx\>\Big)\Big|^m\}^{1/m}& \le  C\, d^{-(k+1)/2}\, . \label{eq:HermiteApprox2}
        \end{align}
        
        We next prove Eq.~\eqref{eq:Vbound}.  Using Eq.~\eqref{eq:HermiteApprox1} with $k=2$ and recalling $\He_0(x) = 1$,
        $\He_1(x)=x$, $\He_2(x)=x^2-1$, we get
        \begin{align*}
          v_i &= \E_{\bx}\left\{h\Big(\frac{1}{d}\<\bx_i,\bx\>\Big) f^*(\bx)\right\}\\
              &= h_{i,0}  \E_{\bx}\left\{f^*(\bx)\right\}+  \frac{h_{i,1}}{\sqrt{d Q_{ii}}}\<\bx_i,\E_{\bx}\left\{\bx f^*(\bx)\right\}\>\\
          &\qquad{}
          +\frac{h_{i,2}}{2dQ_{ii}}\E_{\bx}\left\{f^*(\bx)(\<\bx,\bx_i\>^2-dQ_{ii})\right\}+
            \E_{\bx}\left\{h_{i,>2}\Big(\frac{1}{d}\<\bx_i,\bx\>\Big) f^*(\bx)\right\}\\
              &= h_{i,0}  b_0+  \frac{h_{i,1}}{\sqrt{d Q_{ii}}}\<\bSigma\bbeta_0,\bx_i\>
          +\frac{h_{i,2}}{2dQ_{ii}}\<\bx_i,\bF_2\bx_i\> +
            \E_{\bx}\left\{h_{i,>2}\Big(\frac{1}{d}\<\bx_i,\bx\>\Big) f^*(\bx)\right\}
          \, .
        \end{align*}
       Here we defined the  $d\times d$ matrix $\bF_2= \E\{[f^*(\bx)-b_0]\bx\bx^{\tr}\}$.
       Recalling the definitions of $h_{i,k}$, in Eq.~\eqref{eq:hik}, we get $h_{i,0}=a_{i,0}$.
       Comparing other terms we obtain that the following holds with very high probability,
        \begin{align*}
          |v_i-v_{0,i}|\le & \frac{1}{d}\Big|\E \Big\{h'\Big(\sqrt{\frac{Q_{ii}}{d}} \, G\Big)\Big\}- h'(0)\Big| \cdot |\<\bSigma\bbeta_0,\bx_i\>|
          +\frac{1}{d^2}\Big|\E \Big\{h''\Big(\sqrt{\frac{Q_{ii}}{d}} \, G\Big)\Big\}\Big|\cdot\Big|\<\bx_i,\bF_2\bx_i\>\Big|\\
                           &+\Big|\E_{\bx}\left\{h_{i,>2}\Big(\frac{1}{d}\<\bx_i,\bx\>\Big) f^*(\bx)\right\}\Big|\\
          \stackrel{(a)}{\le} & \frac{1}{d}\times \frac{C}{d} \times C\log d+ \frac{C}{d^2} \big|\<\bx_i,\bF_2\bx_i\>\big|+C\, d^{-3/2}\\
          \le & \frac{C}{d^2} \Big|\<\bx_i,\bF_2\bx_i\>\Big|+C\, d^{-3/2}.
        \end{align*}
        Here the  inequality $(a)$ follows since $|\E \{h'(Z)-h'(0)\}|\le C\E\{Z^2\}$ by smoothness of $h$ and Taylor expansion,
       $\max_{i\le n} |\<\bSigma\bbeta_0,\bx_i\>|\le C\sqrt{\log n}$ by sub-Gaussian tail bounds, and we used Eq.~\eqref{eq:HermiteApprox2}
       for the last term.

       The proof of
       Eq.~\eqref{eq:Vbound} is completed by showing that, with very high probability, $\max_{i\le n}|\<\bx_i,\bF_2\bx_i\>|\le
       C\|\proj_{>1}f^*\|_{L^2}\sqrt{d\log d}$. Without loss of generality, we assume here $\|\proj_{>1}f^*\|_{L^2}= 1$. In order to
       show this claim, note that (defining $\proj_{>0}f^*(\bx):=f^*(\bx)-\E f^*(\bx)$)
       \begin{align}
        \E \<\bx_i,\bF_2\bx_i\>= \trace(\bSigma\bF_2) \le C\E\{\proj_{>0}f^*(\bx)\|\bx\|_2^2\} \le \Var(\|\bx\|_2^2)^{1/2}\le C\sqrt{d}\, .
       \end{align}
       Further notice that
       \begin{align}
         \|\bF_2\|_{\op}&= \max_{\|\bv\|_2=1}|\<\bv,\bF_2\bv\>| \\
                        & = \max_{\|\bv\|_2=1}\big|\E\big\{\proj_{>0}f^*(\bx)\<\bv,\bx\>^2\big\}\big| \\
                        & \le \max_{\|\bv\|_2=1}\E\big\{\<\bv,\bx\>^4\big\}^{1/2}\le  C\, .
         \end{align}
       By the above and the Hanson-Wright inequality
       \begin{align}
         \P\big(\<\bx_i,\bF_2\bx_i\> \ge C\sqrt{d}+t \big)
&\le 2\, \exp\Big(-c\Big(\frac{t^2}{\|\bF_2\|_F^2} \wedge\frac{t}{\|\bF_2\|_{\op}}\Big)\Big)\le 2\, e^{-c((t^2/d)\wedge t)}\, ,
       \end{align}
and similarly for the lower tail.
By taking a union bound over $i\le n$, we obtain $\max_{i\le n}|\<\bx_i,\bF_2\bx_i\>|\le C\sqrt{d\log d}$ as claimed, thus
completing the proof of Eq.~\eqref{eq:Vbound}. 

        We next prove Eq.~\eqref{eq:MBound1}. We claim that this bound
        holds for any realization in $\cE_1\cap \cE_2$. Therefore we can fix
        without loss of generality $i=1$, $j=2$. We use Eq.~\eqref{eq:HermiteApprox1} with $k=4$. Using Cauchy-Schwarz and
        Eqs.~\eqref{eq:HermiteApprox1}, \eqref{eq:HermiteApprox2}, we get
        \begin{align}
          M_{12} &= \sum_{\ell_1,\ell_2=0}^4\frac{1}{\ell_1!\ell_2!} h_{1,\ell_1}h_{2,\ell_2}M_{1,2}(\ell_1,\ell_2)+ \Delta_{12}\, ,\\
          M _{1,2}(\ell_1,\ell_2)&:=\E_{\bx}\left\{\He_{\ell_1}\Big(\frac{1}{\sqrt{d Q_{11}}}\<\bx_1,\bx\>\Big)
                                     \He_{\ell_2}\Big(\frac{1}{\sqrt{d Q_{22}}}\<\bx_2,\bx\>\Big)\right\} \, ,\;\;\;\;|\Delta_{12}|\le Cd^{-5/2}\, .
        \end{align}
        Note that, by Eq.~\eqref{eq:hik}, $|h_{ik}|\le Cd^{-k/2}$, and $M_{1,2}(\ell_1,\ell_2)$
        is bounded on the event $\cE_1\cap\cE_2$, by the  sub-Gaussianity of $\bz$.
        Comparing with Eqs.~\eqref{eq:M0def1}, \eqref{eq:M0def2}, we get
\begin{align}
  |M_{12}-M_{0,12}|&\le\Big|\sum_{\ell_1,\ell_2=0}^4 \frac{1}{\ell_1!\ell_2!} h_{1,\ell_1}h_{2,\ell_2}M_{1,2}(\ell_1,\ell_2)-M_{0,12}\Big|+Cd^{-5/2}\\
                     &\phantom{AAAAA}+2\sum_{(\ell_1,\ell_2)\in \cS} \big|h_{1,\ell_1}h_{2,\ell_2}M_{1,2}(\ell_1,\ell_2)\big|\label{eqMainMbound}\\
  &\phantom{AAAAA}+
2 \Big|\frac{1}{6}h_{1,0}h_{2,3}M_{1,2}(0,3)-a_{1,0}a_{2,1}\Big|+|a_{1,1}a_{2,1}| 
  + Cd^{-5/2}\, ,\nonumber\\
 \cS&:= \big\{(0,1), (0,2), (0,4), (1,2), (1,3), (2,2)\big\}\, ,
\end{align}
where in the inequality we used the identities $h_{1,0}h_{2,0}M_{1,2}(0,0)= h_{1,0}h_{2,0}=a_{1,0}a_{2,0}$, and
\begin{align*}
  h_{1,1}h_{2,1}M_{1,2}(1,1) = \frac{1}{d^2}\<\bx_1,\bSigma\bx_2\>\E h'\Big(\sqrt{\frac{Q_{11}}{d}}G\Big) \E h'\Big(\sqrt{\frac{Q_{22}}{d}}G\Big) =
  B_{12}\, .
\end{align*}

We next bound each of the terms above separately.

We begin with the terms $(\ell_1,\ell_2)\in\cS$. Since by Eq.~\eqref{eq:hik}, $|h_{ik}|\le Cd^{-k/2}$,
for each of these pairs, we need to show  $|M_{1,2}(\ell_1,\ell_2)|\le Cd^{(\ell_1+\ell_2-5)/2}\log d$.
Consider $(\ell_1,\ell_2) = (0,k)$, $k\in\{1,2,4\}$. Set $\bw = \bSigma^{1/2}\bx_2/\sqrt{d Q_{22}}$, $\|\bw\|_2=1$,
and write $\He_k(x) = \sum_{m=0}^kc_{k,\ell}x^{\ell}$. If $\bg$ is a standard Gaussian vector, we have $\E_{\bg}\He_{k}\big(\<\bw,\bg\>)=0$
and therefore
        \begin{align}
          M _{1,2}(0,k)&=\E_{\bz}\left\{  \He_{k}\big(\<\bw,\bz\>\big)\right\}-\E_{\bg}\left\{  \He_{k}\big(\<\bw,\bg\>\big)\right\}\\
                       & = \sum_{\ell=0}^kc_{k,\ell} \sum_{i_1,\dots,i_\ell\le n} w_{i_1}\cdots w_{i_\ell}\big\{\E(z_{i_1}\cdots z_{i_{\ell}})-
                         \E(g_{i_1}\cdots g_{i_{\ell}})\big\}\, .
        \end{align}
        Note that the only non-vanishing terms in the above sum are those in which all of the indices appearing in $(i_1,\dots, i_{\ell})$
        appear at least twice, and at least one of the indices appears at least $3$ times (because otherwise the two expectations are equal).
        This immediately implies $M _{1,2}(0,1)=M _{1,2}(0,2)=0$. Analogously, all terms $\ell\le 2$ vanish in the above sum.
        
As for $k=4$, we have (recalling $\He_4(x) = x^4-3x^2$):
        \begin{align}
         M _{1,2}(0,4)& =\left|\sum_{i_1,\dots,i_4\le n} w_{i_1}\cdots w_{i_4}\big\{\E(z_{i_1} \cdots z_{i_{4}})-
                           \E(g_{i_1}\cdots g_{i_{4}})\big\} \right|\\
          &\le \sum_{i\le n} w_{i}^4\big|\E(z_{i}^4)-3\big| \le C\|\bw\|^2_{\infty} \|\bw\|_2^2\le \frac{C\log d}{d}\, ,
        \end{align}
        where the last inequality follows since $\|\bw\|_2=1$ by construction and $\|\bw\|_{\infty} \le C\sqrt{(\log d)/d}$ on
        $\cE_1\cap\cE_2$.

        Next consider $(\ell_1,\ell_2) = (1,2)$. Setting $\bw_i= \bSigma^{1/2}\bx_i/\sqrt{d Q_{ii}}$, $i\in\{1,2\}$, we get
        \begin{align}
          M _{1,2}(1,2)&=\E_{\bz}\left\{\He_{1}\big(\<\bw_1,\bz\>\big)  \He_{2}\big(\<\bw_2,\bz\>\big)\right\}-
                         \E_{\bg}\left\{\He_{1}\big(\<\bw_1,\bg\>\big)  \He_{2}\big(\<\bw_2,\bg\>\big)\right\}\\
          & = \E_{\bz}\left\{\big(\<\bw_1,\bz\>\big) \big(\<\bw_2,\bz\>\big)^2\right\}-
            \E_{\bg}\left\{\big(\<\bw_1,\bg\>\big) \big(\<\bw_2,\bg\>\big)^2\right\}\\
          & = \sum_{i_1,i_2,i_3\le n} w_{1,i_1}w_{2,i_2}w_{2,i_3}\big\{\E(z_{i_1} z_{i_2}z_{i_{2}})-
            \E(g_{i_1}g_{i_2} g_{i_{3}})\big\}\\
          & = \sum_{i=1}^n w_{1,i}w_{2,i}w_{2,i}\E(z_i^3)\, .
        \end{align}
        Therefore, on $\cE_1\cap\cE_2$,
 \begin{align}
   \big|M _{1,2}(1,2)\big|&\le C\Big|\sum_{i=1}^n w_{1,i}w_{2,i}^2\Big|\le \frac{C\log d}{d}\, .
 \end{align}
 
 Next consider $(\ell_1,\ell_2) = (1,3)$. Proceeding as above (and noting that the degree-one term in $\He_3$ does not contribute),
 we get
    \begin{align}
          M _{1,2}(1,3)&=\E_{\bz}\left\{\He_{1}\big(\<\bw_1,\bz\>\big)  \He_{3}\big(\<\bw_2,\bz\>\big)\right\}-
                         \E_{\bg}\left\{\He_{1}\big(\<\bw_1,\bg\>\big)  \He_{3}\big(\<\bw_2,\bg\>\big)\right\}\\
          & = \E_{\bz}\left\{\big(\<\bw_1,\bz\>\big) \big(\<\bw_2,\bz\>\big)^3\right\}-
            \E_{\bg}\left\{\big(\<\bw_1,\bg\>\big) \big(\<\bw_2,\bg\>\big)^3\right\}\\
       & = \sum_{i_1,\dots,i_4\le d} w_{1,i_1}w_{2,i_2}w_{2,i_3}w_{2,i_4}\big\{\E(z_{i_1} z_{i_2}z_{i_{2}}z_{i_4})-
            \E(g_{i_1}g_{i_2} g_{i_{3}}g_{i_4})\big\}\\
          & = \sum_{i=1}^d w_{1,i}w_{2,i}^3(\E(z_i^4)-3)\, .
    \end{align}
      Therefore, on $\cE_1\cap\cE_2$
 \begin{align}
   \big|M _{1,2}(1,3)\big|&\le C\Big|\sum_{i=1}^d w_{1,i}w_{2,i}^3\Big|\le C\|\bw_1\|_{\infty}\|\bw_2\|_{\infty}\|\bw_2\|_2^2\le \frac{C\log d}{d}\, .
 \end{align}
 Finally, for $(\ell_1,\ell_2)=(2,2)$, proceeding as above we get
 \begin{align}
   M _{1,2}(2,2)&= \left|\sum_{i=1}^d w_{1,i}^2w_{2,i}^2(\E(z_i^4)-3)\right|\le C\|\bw_1\|_{\infty}^2\|\bw_2\|_2^2\le \frac{C\log d}{d}\, .
 \end{align}

 Next consider the term $|h_{1,0}h_{2,3}M_{1,2}(0,3)/6-a_{1,0}a_{2,1}|$ in Eq.~\eqref{eqMainMbound}. Using the fact that
 $h_{1,0}=a_{1,0}$ is bounded, we get
    \begin{align}
      \Big|\frac{1}{6}h_{1,0}h_{2,3}M_{1,2}(0,3)-a_{1,0}a_{2,1}\Big|\le C\big|h_{2,3}M_{1,2}(0,3)-6a_{2,1}\big|\, .
      \end{align}
Recalling $\He_3(x) = x^3-3x$, and letting $\bw =\bSigma^{1/2} \bx_2/\|\bSigma^{1/2} \bx_2\|_2$:
        \begin{align}
         M _{1,2}(0,3)& =\sum_{i_1,\dots,i_3\le d} w_{i_1}w_{i_2}w_{i_3}\big\{\E(z_{i_1} z_{i_2}z_{i_{3}})-
                           \E(g_{i_1}g_{i_2} g_{i_{3}})\big\} \\
          &=\sum_{i\le d} w_{i}^3\E(z_{i}^3) \, .
        \end{align}
        In particular, on the event $\cE_1\cap\cE_2$, $|M _{1,2}(0,3)|\le C\sqrt{(\log d)/d}$.
        Comparing  the definitions of $a_{2,1}$ and $h_{2,3}$, we get
    \begin{align}
      \big|h_{1,0}h_{2,3}M_{1,2}(0,3)-a_{1,0}a_{2,1}\big|&\le
      C|M_{1,2}(0,3)|\times\left(\frac{Q_{22}}{d}\right)^{3/2}\Big| \E \Big\{h^{(3)}\Big(\sqrt{\frac{Q_{ii}}{d}} \, G\Big)\Big\} -h^{(3)}(0)\Big|\\
                                                         &\le C\sqrt{\frac{\log d}{d}}\times
                                                           \frac{1}{d^{3/2}}\times \frac{1}{d^{1/2}} \le \frac{C(\log d)^{1/2}}{d^{5/2}}
                              \, .                             
      \end{align}

      Finally, consider term $|a_{1,1}a_{2,1}|$ in Eq.~\eqref{eqMainMbound}.   By the above estimates, we get $|a_{2,1}|\le C d^{-2}(\log d)^{1/2}$,
      and hence this term is negligible as well. This completes the proof of Eq.~\eqref{eq:MBound1}. 
      
         Equation~\eqref{eq:MBound2} follows by a similar argument, which we omit.
       \end{proof}

       \subsubsection{An estimate on the entries of the resolvent}

       \begin{lemma}\label{lemma:ExpResolvent}
         Let $\bZ=(z_{ij})_{i\le n, j\le d}$ be a random matrix with
         iid rows $\bz_1,\dots,\bz_n\in\reals^d$ that are zero mean and $C$-sub-Gaussian.
         Further assume $C^{-1}\le n/d\le C$. Let $\bS\in\reals^{d\times d}$ be a symmetric
         matrix such that $\bfzero \preceq \bS\preceq C\id_d$ for some finite constant $C>1$. 
         Finally, let $g:\reals^d\to\reals$ be a measurable function such that $\E\{g(\bz_1)\} = \E\{\bz_1g(\bz_1)\}=0$,
         and $\E\{g(\bz_1)^2\}=1$.

         Then, for any $\lambda>0$ there exists a finite constant $C$ such that, for  any $i\neq j$,
         \begin{align}
           \Big|\E\Big\{\big(\bZ\bS\bZ^{\tr}/d+\lambda\id_n\big)^{-1}_{i,j}g(\bz_i)g(\bz_j)\Big\}\Big|\le
           C\, d^{-3/2}\, .
         \end{align}
       \end{lemma}
       \begin{proof}
         Without loss of generality, we can consider $i=1$, $j=2$.  Further, we let $\bZ_0\in\reals^{(n-2)\times d}$
         be the matrix comprising the last $n-2$ rows of $\bZ$, and $\bU\in \reals^{d\times 2}$ be the matrix with columns
         $\bU\bfe_1= \bz_{1}$,  $\bU\bfe_2= \bz_{2}$. We finally define the matrices $\bR_0\in\reals^{d\times d}$ and
         $\bY=(Y_{ij})_{i,j\le 2}$:
         \begin{align}
           \bR_0 & := \lambda\bS^{1/2}\big(\bS^{1/2}\bZ_0^{\tr}\bZ_0\bS^{1/2}/d+\lambda\id_{d}\big)^{-1}\bS^{1/2}\, ,\\
           \bY& := \big(\bZ\bS\bZ^{\tr}/d+\lambda\id_n\big)^{-1}\, .
         \end{align}
         Then, by a simple linear algebra calculation, we have
         \begin{align}
           \bY &= \Big(\bU^{\tr}\bR_0\bU/d +\lambda\id_2\Big)^{-1}\, ,\\
             Y_{12}& = -\frac{\<\bz_{1},\bR_0\bz_2\>/d}{(\lambda+\<\bz_{1},\bR_0\bz_{2}\>/d)
               (\lambda+\<\bz_{1},\bR_0\bz_{2}\>/d)-\<\bz_{1},\bR_0\bz_2\>^2/d^2}\, .
         \end{align}
         Note that since $\bR_0\succeq 0$, we have $\<\bz_{1},\bR_0\bz_2\>^2\le \<\bz_{1},\bR_0\bz_{1}\>\<\bz_{2},\bR_0\bz_{2}\>$, and therefore
         \begin{align}
           Y_{12} & = Y_{12}^{(1)}+Y_{12}^{(2)}\, ,\\
           Y_{12}^{(1)} &:= -\frac{\<\bz_{1},\bR_0\bz_2\>/d}{(\lambda+\<\bz_{1},\bR_0\bz_{1}\>/d)
               (\lambda+\<\bz_{2},\bR_0\bz_{2}\>/d)}\, ,\\
           |Y_{12}^{(2)}|& \le \frac{1}{\lambda^4d^3}|\<\bz_{1},\bR_0\bz_2\>|^3\, .
         \end{align}
         Denote by $\E_{+}$ expectation with respect to $\bz_{1},\bz_2$ (conditional on $(\bz_i)_{2<i\le n}$).
         We  have
         \begin{align*}
           \big|\E_+\{Y_{12}\, g(\bz_{1})\, g(\bz_2)\}\big| & \le \big|\E_+\{Y_{12}^{(1)}\, g(\bz_{1})\, g(\bz_2)\}\big| +
                                                              \E_+\{(Y_{12}^{(2)})^2\}^{1/2}\, \E_+\{g(\bz_{1})^2\, g(\bz_2)^2\}^{1/2} \\
          &\le \big|\E_+\{Y_{12}^{(1)}\, g(\bz_{1})\, g(\bz_2)\}\big| +
            \E_+\{(Y_{12}^{(2)})^2\}^{1/2}\\
                                                        & \le \big|\E_+\{Y_{12}^{(1)}\, g(\bz_{1})\, g(\bz_2)\}\big| +C\, d^{-3/2}\, .
         \end{align*}
         Here the last step follows by the Hanson-Wright inequality.
         We therefore only have to bound the first term. Defining $q_j:= \lambda+\<\bz_{j},\bR_0\bz_{j}\>/d$,
         $\oq_j=\E_+ q_j$, $g_j=g(\bz_j)$, $j\in\{1,2\}$,
         \begin{align*}
           \big|\E_+\{Y_{12}^{(1)}\, g_1\, g_2\}\big| & \le
                                                             \Big|\E_+\Big\{\oq^{-2}\frac{\<\bz_1,\bR_0\bz_{2}\>}{d} g_1g_2\Big\}\Big|\\
         &\qquad{}  +2
           \Big|\E_+\Big\{\Big(q_1^{-1}-\oq^{-1}\Big)\oq^{-2}\frac{\<\bz_1,\bR_0\bz_{2}\>}{d} g_1g_2\Big\}\Big|\\
                                                      &\qquad{}+ \Big|\E\Big\{\Big(q_1^{-1}-\oq^{-1}\Big) \Big(q_2^{-1}-\oq^{-1}\Big)\frac{\<\bz_1,\bR_0\bz_{2}\>}{d}
                                                        g_1g_2\Big\}\Big|\\
                                                      &\stackrel{(a)}{\le} \Big|\E_+\Big\{\Big(q_1^{-1}-\oq^{-1}\Big) \Big(q_2^{-1}-\oq^{-1}\Big)
                                                        \frac{\<\bz_1,\bR_0\bz_{2}\>}{d} g_1g_2\Big\}\Big|\\
                                                           &\le  \frac{1}{\lambda^4}\E\Big\{|q_1-\oq| |q_2-\oq|\Big|\frac{\<\bz_1,\bR_0\bz_{2}\>}{d}\Big|
                                                             |g_1g_{2}|\Big\}\, .
         \end{align*}
         Here $(a)$ follows from the orthogonality of $g(\bz)$ to linear functions.

         We then conclude
         \begin{align*}
           \big|\E_+\{Y^{(1)}_{12}\, g_1\, g_2\}\big| &\stackrel{(a)}{\le}  C
                                                             \E_+\big\{|q_1-\oq|^8\big\}^{1/4}
                                                             \E_+\big\{(\<\bz_1,\bR_0\bz_{2}\>/d)^4\big\}^{1/4}\\
           &\le  C \E_+\big\{|\<\bz_1,\bR_0\bz_1\>/d-\E_+\<\bz_1,\bR_0\bz_1\>/d|^8\big\}^{1/4}
             \E_+\big\{(\<\bz_1,\bR_0\bz_2\>/d)^4\big\}^{1/4}\\
                                                           & \stackrel{(b)}{\le}  C (d^{-1/2})^2\times Cd^{-1/2}\le Cd^{-3/2}\, .
         \end{align*}
         Here $(a)$ follows from H\"older's inequality and $(b)$ from
         the Hanson-Wright inequality using the fact that $\|\bR_0\|_{\op}$ is bounded. The proof is completed by taking expectation over $(\bz_i)_{2<i\le n}$.
       \end{proof}

       \begin{lemma}\label{lemma:CovResolvent}
         Under the definitions and assumptions of
         Lemma~\ref{lemma:ExpResolvent}, let
         $Y_{ij}:=(\bZ\bS\bZ^{\tr}/d+\lambda\id_n)^{-1}_{i,j}$.
         Then, for any tuple of four distinct indices
         $i,j,k,l$, we have
         \begin{align}
           \big|\E\{Y_{ij}Y_{kl}g(\bz_i)g(\bz_j) g(\bz_k)g(\bz_l)\}\big|\le Cd^{-5/2}\, .
         \end{align}
       \end{lemma}
       \begin{proof}
         The proof is analogous to the one of Lemma \ref{lemma:ExpResolvent}. Without loss of generality,
         we set $(i,j,k,l) = (1,2,3,4)$, denote by $\bZ_0\in \reals^{(n-4)\times d}$ the matrix with rows $(\bz_{\ell})_{\ell\ge 5}$, and define the $d\times d$ matrix
          \begin{align}
            \bR_0 & := \lambda\bS^{1/2}\big(\bS^{1/2}\bZ_0^{\tr}\bZ_0\bS^{1/2}/d+\lambda\id_{n-2}\big)^{-1}
                    \bS^{1/2}\, .
          \end{align}
          We then have that $\bY = (Y_{ij})_{i,j\le 4}$ is given by
          \begin{align}
            \bY &= (\diag(\bq)+\bA)^{-1}\, ,\\
            q_i & := \oq+ Q_i\, ,\;\;\; \oq:=\lambda+\trace(\bR_0)/d\, ,\;\;
                  Q_i =  (\<\bz_i,\bR_0\bz_i\>-\E  \<\bz_i,\bR_0\bz_i\>)/d \, , \\
            A_{ij} & := \begin{cases}
              \<\bz_i,\bR_0\bz_j\>/d & \mbox{ if $i\neq j$,}\\
              0 & \mbox{ if $i = j$.}\\
              \end{cases}
          \end{align}
          In what follows we denote by $\E_+$ expectation with respect
          to $(\bz_i)_{i\le 4}$, with $\bZ_0$ fixed.
          Note that, by the Hanson-Wright inequality,
          $\E_+\{|A_{ij}|^k\}^{1/k}\le c_k\, d^{-1/2}$,
          $\E_+\{|Q_{i}|^k\}^{1/k}\le c_k\, d^{-1/2}$ for each $k\ge 1$.
          We next compute the Taylor expansion of $Y_{12}$ and $Y_{3,4}$ in powers of $\bA$ to get
          \begin{align}
            Y_{12}& = Y^{(1)}_{12}+Y^{(2)}_{12}+Y^{(3)}_{12}+Y^{(4)}_{12}\, ,\\
            Y^{(1)}_{12}&:= -q_{1}^{-1}A_{12}q_2\, ,\\
            Y^{(2)}_{12}&:= q_1^{-1}A_{13}q_3^{-1}A_{32}q_2^{-1} +q_1^{-1}A_{14}q_4^{-1}A_{41}q_2^{-1}\, ,\\
            Y^{(3)}_{12}& := -\sum_{i_1\neq i_2, i_1\neq 1 i_2\neq 2} q_1^{-1}A_{1i_1}q_{i_1}^{-1}A_{i_1i_2}q_{i_2}^{-1}
                          A_{i_22}q_{2}^{-1}\, ,
          \end{align}
          and similarly for $Y_{34}$.
          It is easy to show that $\E_+\{|Y^{(\ell)}_{ab}|^k\}^{1/k}\le c_kd^{-\ell/2}$, for all $k\ge 1$. Therefore,
          using $\E\{g(\bz_i)^2\}\le C$ and Cauchy-Schwarz inequality, and writing $g_i = g(\bz_i)$:
          \begin{align}
            \big|\E\{Y_{12}Y_{34}g_1g_2g_3g_4\}\big|= \sum_{\ell_1+\ell_2\le 4}
            \big|\E\{Y^{(\ell_1)}_{12}Y^{(\ell_2)}_{34}g_1g_2g_3g_4\}\big|+ Cd^{-5/2}\, .
          \end{align}
          The proof is completed by bounding each of the terms above, which we now do. By symmetry
          it is sufficient to consider $\ell_1\le \ell_2$ and therefore we are left with the $4$ pairs
          $(\ell_1,\ell_2)\in \{(1,1),(1,2),(1,3),(2,2)\}$.

          \vspace{0.1cm}
          
          \noindent{\bf Term $(\ell_1,\ell_2) = (1,1)$.} By the same argument as in the proof of Lemma \ref{lemma:ExpResolvent}, we have $|\E\{A_{ij}q_i^{-1}q_j^{-1}g_ig_j\} |\le Cd^{-3/2}$ and therefore
          \begin{align}
            \big|\E_+\{Y^{(1)}_{12}Y^{(1)}_{34}g_1g_2g_3g_4\}\big|= \big|\E_+\{A_{12}q_1^{-1}q_2^{-1}g_1g_2\} \big|\cdot \big|\E\{A_{34}q_3^{-1}q_4^{-1}g_3g_4\} \big|\le Cd^{-3}\, .
          \end{align}

          \vspace{0.1cm}
          
          \noindent{\bf Term $(\ell_1,\ell_2) = (1,2)$.} Note that each of the two terms in the definition of
          $Y^{(2)}_{34}$ contributes a summand with the same
          structure.  Hence we can consider just the one 
          resulting in the largest expectation, say $q_3^{-1}A_{31}q_{1}^{-1}A_{14}q_4^{-1}$
          \begin{align*}
            \big|\E_+\{Y^{(1)}_{12}Y^{(2)}_{34}g_1g_2g_3g_4\}\big|&= 2   \big|\E_+\{q_{1}^{-1}A_{12}q^{-1}_2
                                                                  q_3^{-1}A_{31}q_{1}^{-1}A_{14}q_4^{-1}g_1g_2g_3g_4\}\big|\\
 & \stackrel{(a)}{=} 2   \big|\E_+\{q_{1}^{-2}A_{12}(q^{-1}_2-\oq^{-1})
   (q_3^{-1}-\oq^{-1})A_{31}A_{14}(q_4^{-1}-\oq^{-1})g_1g_2g_3g_4\}\big|\\
                                                                & \stackrel{(b)}{\le} C \E_+\{|A_{12}|^{p}\}^{1/p} \E\{|A_{13}|^{p}\}^{1/p} \E\{|A_{13}|^{p}\}^{1/p}\E_+\{|q^{-1}_2-\oq^{-1}|^p\}^{1/p}\E\{|q^{-1}_3-\oq^{-1}|^p\}^{1/p}\\
            &\phantom{AAAAA}\cdot \E\{|q^{-1}_4-\oq^{-1}|^p\}^{1/p}\|g\|^4_{L^{2}}\\
                                                                &\stackrel{(c)}{\le} Cd^{-3}\, .
          \end{align*}
          Here $(a)$ holds because $g_i$ is orthogonal to $\bz_i$ for $i\in\{2,3,4\}$ and hence the terms $\oq^{-1}$ 
          have vanishing contribution; $(b)$~ by H\"older for $p= 12$, and using the fact that $q_i^{-1}$ is bounded; $(c)$~ by the above bounds on the moments of $A_{ij}$, $Q_i$, plus
          $|q^{-1}_i-\oq^{-1}|\le C|Q_i|$.
          
          \vspace{0.1cm}
          
          \noindent{\bf Term $(\ell_1,\ell_2) = (1,3)$.} Taking into account symmetries,
          there are only two distinct terms to  consider in the sum
          defining $Y^{(3)}_{34}$, which we can identify with the
          following ones:
          \begin{align*}
            \big|\E_+\{Y^{(1)}_{12}Y^{(3)}_{34}&g_1g_2g_3g_4\}\big|\le
                                                              C\big|\E_+\big\{q_{1}^{-1}A_{12}q^{-1}_2
                                                                  q_3^{-1}A_{31}q_{1}^{-1}A_{12}q_{2}^{-1}A_{24}q_4^{-1}g_1g_2g_3g_4\big\}\big|\\
            &+C\big|\E_+\big\{q_{1}^{-1}A_{12}q^{-1}_2
                                                                  q_3^{-1}A_{31}q_{1}^{-1}A_{13}q_{3}^{-1}A_{34}q_4^{-1}g_1g_2g_3g_4\big\}\big|=: C\cdot T_1+C\cdot T_2\, .
          \end{align*}
          Notice that in the first term $\bz_3$ only appears in $q_3$, $A_{31}$, and $g_3$, and similarly $\bz_4$ only appears in $q_4$, $A_{24}$, and $g_4$. Hence
          \begin{align*}
            T_1=   \big|\E_+\big\{q_{1}^{-1}A_{12}q^{-1}_2
                                                                  (q_3^{-1}-\oq^{-1})A_{31}q_{1}^{-1}A_{12}q_{2}^{-1}A_{24}(q_4^{-1}-\oq^{-1})g_1g_2g_3g_4\big\} \big|\le C d^{-3}\, ,
          \end{align*}
          where the last inequality follows again by H\"older.
          Analogously, for the second term we have
          \begin{align*}
  T_2  =   \big|\E_+\big\{q_{1}^{-1}A_{12}(q^{-1}_2-\oq^{-1})
                                                                  q_3^{-1}A_{31}q_{1}^{-1}A_{32}q_{3}^{-1}A_{24}(q_4^{-1}-\oq^{-1})g_1g_2g_3g_4\big\} \big|\le C d^{-3}\, ,
          \end{align*}
          This proves the desired bound for $(\ell_1,\ell_2) = (1,3)$.
          
          \vspace{0.1cm}
          
          \noindent{\bf Term $(\ell_1,\ell_2) = (2,2)$.} There are four terms that arise from the sum in the
          definition of $Y^{(2)}_{ij}$. By symmetry, these are equivalent by pairs
          \begin{align*}
            \big|\E_+\{Y^{(2)}_{12}Y^{(2)}_{34}g_1g_2g_3g_4\}\big|&\le
                                                              2\big|\E_+\big\{q_{1}^{-1}A_{13}q^{-1}_3A_{32}q_2^{-1}
                                                                  q_3^{-1}A_{31}q_{1}^{-1}A_{14}q_{4}^{-1}g_1g_2g_3g_4\big\}\big|\\
            &\qquad{} +   2\big|\E_+\big\{q_{1}^{-1}A_{13}q^{-1}_3A_{32}q_2^{-1}
              q_3^{-1}A_{32}q_{2}^{-1}A_{24}q_{4}^{-1}g_1g_2g_3g_4\big\}\big|\\
            &\le
              2\big|\E_+\big\{q_{1}^{-1}A_{13}q^{-1}_3A_{32}(q_2^{-1}-\oq^{-1})
                                                                  q_3^{-1}A_{31}q_{1}^{-1}A_{14}(q_{4}^{-1}-\oq^{-1})g_1g_2g_3g_4\big\}\big|\\
            &\qquad{} +   2\big|\E_+\big\{(q_{1}^{-1}-\oq^{-1})A_{13}q^{-1}_3A_{32}q_2^{-1}
              q_3^{-1}A_{32}q_{2}^{-1}A_{24}(q_{4}^{-1}-\oq^{-1})g_1g_2g_3g_4\big\}\big|\\
                                                                &\le Cd^{-3}\, .
          \end{align*}
          This completes the proof of this lemma.
       \end{proof}

        \begin{lemma}\label{lemma:CovResolvent2}
          Under the definitions and assumptions of Lemma \ref{lemma:CovResolvent}, further assume $\E\{|g(\bz)|^{2+\eta}\}\le C$ for some constants $0<C,\eta<\infty$.
          for any triple of four distinct indices
         $i,j,k$, we have
         \begin{align}
           \big|\E\{Y_{ij}Y_{jk}g(\bz_i) g(\bz_j)^2g(\bz_k)\}\big|&\le Cd^{-3/2}\, ,\label{eq:3Vertices}\\
           \big|\E\{Y_{ij}^2g(\bz_i)^2 g(\bz_l)^2\}\big|&\le Cd^{-1}\, . \label{eq:2Vertices}
         \end{align}
       \end{lemma}
       \begin{proof}
         This proof is very similar to the one of Lemma
         \ref{lemma:CovResolvent}, and we will follow the same
         notation introduced there.

         Consider Eq.~\eqref{eq:3Vertices}. Without loss of generality, we take $(i,j,k)=(1,2,3)$.
         Since $\E\{|Y^{(\ell)}_{ij}|^k\}\le c_kd^{-\ell/2}$, we have
         \begin{align}
           \big|\E_+\{Y_{12}Y_{23}g_1 g_2^2g_3\}\big|\le   \big|\E_+\{Y_{12}^{(1)}Y^{(1)}_{23}g_1 g_2^2g_3\}\big|+
           Cd^{-3/2}\, .
         \end{align}
         Further
         \begin{align*}
           \big|\E_+\{Y_{12}^{(1)}Y^{(1)}_{23}g_1 g_2^2g_3\}\big|& =      \big|\E_+\{q_1^{-1}A_{12}q_2^{-2}A_{23}q_{3}^{-1}g_1 g_2^2g_3\}\big|\\
                                                               & =  \big|\E_+\{(q_1^{-1}-\oq^{-1})A_{12}q_2^{-2}A_{23}(q_{3}^{-1}-\oq^{-1})g_1 g_2^2g_3\}\big|\\
           &\le C d^{-2}\, ,
         \end{align*}
         where the last bound follows from H\"older inequality.

         Finally, Eq.~\eqref{eq:2Vertices} follows immediately by H\"older inequality since $\E\{|Y_{ij}|^k\}^{1/k}\le C_kd^{-1/2}$ for all $k$.
       \end{proof}

       \begin{theorem}\label{thm:ExpResolvent}
         Let $\bZ=(z_{ij})_{i\le n, j\le d}$ be a random matrix with iid rows $\bz_1,\dots,\bz_n\in\reals^d$, with zero mean $C$-sub-Gaussian.
         Let $\bS\in\reals^{d\times d}$ be a symmetric
         matrix such that $\bfzero\preceq \bS\preceq C\id_d$ for some finite constant $C>1$. 
         Finally, let $g:\reals^d\to\reals$ be a measurable function such that $\E\{g(\bz_1)\} = \E\{\bz_1g(\bz_1)\}=0$,
         and $\E\{|g(\bz_1)|^{4+\eta}\}\le C$.

         Then, for any $\lambda>0$, with probability at least $1-Cd^{-1/4}$, we have
         \begin{align}
           \left|\frac{1}{d}\sum_{i< j\le n}\big(\bZ\bS\bZ^{\tr}/d+\lambda\id_n\big)^{-1}_{i,j}g(\bz_i)g(\bz_j)\right|\le
           C\, d^{-1/8}\, .\label{eq:MainExpResolvent}
         \end{align}
       \end{theorem}
       \begin{proof}
         Denote by $X$ the sum on the left-hand side of Eq.~\eqref{eq:MainExpResolvent}, and define
         $Y_{ij}:=(\bZ\bS\bZ^{\tr}/d+\lambda\id_n)^{-1}_{i,j}$, $g_i=g(\bz_i)$. Further,
         let
         $\cI_m:=\{(i,j,k,l): \; i<j\le n,  k<l\le n,
         |\{i,j\}\cap\{k,j\}|=m\}$,  $m\in\{0,1\}$. Then we have
         \begin{align*}
           \E\{X^2\}& = \frac{1}{d^2}\sum_{i<j}\sum_{k<l}\E\{Y_{ij}Y_{kl}g_ig_jg_kg_l\}\\
                    &\le  \frac{1}{d^2}\sum_{(i,j,k,l)\in\cI_0}\E\{Y_{ij}Y_{kl}g_ig_jg_kg_l\}
                      +\frac{1}{d^2}\sum_{(i,j,k,l)\in\cI_1}\E\{Y_{ij}Y_{kl}g_ig_jg_kg_l\}+
                      +\frac{1}{d^2}\sum_{i<j}\E\{Y_{ij}^2g^2_ig_j^2\}\\
                    &\le Cd^2 \big|\E\{Y_{12}Y_{34}g_1g_2g_3g_4\}\big|+ Cd \big|\E\{Y_{12}Y_{23}g_1g^2_2g_3\}\big|
                      + C\big|\E\{Y^2_{12}g^2_1g^2_2\}\big|\\
                    &\le Cd^{-1/2}\, .
         \end{align*}
         The proof is completed by Chebyshev inequality.
       \end{proof}
       
      \subsubsection{Proof of Theorem \ref{thm:npropd}: Variance term}

      Throughout this section we will refer to the events $\cE_1$, $\cE_2$ defined in Eqs.~\eqref{eq:Event1}, \eqref{eq:Event2}. 
      The variance is given by
        \begin{align}
          \varianceemp = \sigma_{\xi}^2 \E_{\bx}\big\{K(\bx,\bX)^{\tr} K(\bX,\bX)^{-2} K(\bx,\bX)\big\}\, .
        \end{align}

        The following lemma allows us to take the expectation with respect to $\bx$.
      \begin{lemma}\label{lemma:VarApprox}
        Under the assumptions of Theorem \ref{thm:npropd}, define  $\bM_0\in\reals^{n\times n}$ as in the statement of
        Lemma \ref{lemma:ApproxMV}.
        Then, with very high probability, we have
        \begin{align}
          \left| \frac{1}{\sigma_{\xi}^2}\varianceemp - \<\bM_0,\bK^{-2}\>\right|\le
          \frac{C\log d}{d}\, .
        \end{align}
      \end{lemma}
      \begin{proof}
        First notice that, defining $\bM$ as in Eq.~\eqref{eq:Mdef}, we have
        \begin{align}
          \frac{1}{\sigma_{\xi}^2}\varianceemp = \<\bM,\bK^{-2}\>\, .
        \end{align}
        We then have,  with very high probability,
        \begin{align}
          \left| \frac{1}{\sigma^2}\varianceemp - \<\bM_0,\bK^{-2}\>\right|&\le
                                                                                               \left|\<\bM- \bM_0,\bK^{-2}\>\right|\\
                                                                                             &\le \|\bM- \bM_0\|_F \sqrt{n}\|\bK^{-2}\|_{\op}\\
                                                                                             &\stackrel{(a)}{\le} \frac{C\log d}{d^{3/2}}\times\sqrt{d}\times \|\bK^{-1}\|^2_{\op}\\
          &\stackrel{(b)} {\le} \frac{C\log d}{d}\, ,
        \end{align}
       where $(a)$ follows from Lemma \ref{lemma:ApproxMV} and $(b)$ from  Lemma \ref{lemma:ApproxK}. 
        \end{proof}
        In the following we define $\bB_0\in\reals^{n\times n}$ via
        \begin{align}
          \bB_0:= \frac{h'(0)}{d^2}\bX\bSigma\bX^{\tr}\, .\label{eq:B0Def}
        \end{align}
        The next lemma shows that $\bB_0$ is a good approximation for $\bB$, defined in Eq.~\eqref{eq:M0def1}.
        
        \begin{lemma}\label{lemma:BB0}
          Let $\bB$ be defined as per Eq.~\eqref{eq:M0def1}.
          With very high probability,  we have $\|\bB-\bB_0\|_{\op}\le Cd^{-3/2}$ and
          $\|\bB-\bB_0\|_*\le Cd^{-1/2}$.
        \end{lemma}
        \begin{proof}
         Notice that $\bB= \bD\bX\bSigma\bX^{\tr}\bD/d^2$ and, on $\cE_1\cap\cE_2$,
              \begin{align}
                \big\|\bD-h'(0)\id\|_{\op} = \max_{i\le n} \Big|\E h'\Big(\sqrt{\frac{Q_{ii}}{d}}G\Big)-h'(0)\Big|\le \frac{C}{\sqrt{d}}\, .\label{eq:BoundD}
              \end{align}
              We then have
              \begin{align}
                \big\|\bB- \bB_0\big\|_{\op}&\le \frac{C}{\sqrt{d}} \Big\|\frac{1}{d^2}\bX\bSigma\bX^{\tr}\Big\|_{\op}\le \frac{C}{d^{5/2}}\|\bX\|_{\op}^2\le \frac{C}{d^{3/2}}\, .
              \end{align}
              This immediately implies $\|\bB-\bB_0\|_*\le n\|\bB-\bB_0\|_{\op}\le C/\sqrt{d}$.
            \end{proof}
            
          \begin{lemma}\label{lemma:BKLemma}
            Under the assumptions of Theorem \ref{thm:npropd}, let  $\bB$ be defined as per Eq.~\eqref{eq:M0def1}and
            $\bB_0$ as per Eq.~\eqref{eq:B0Def}. Also, recall the definition of $\bK_1$ in Eq.~\eqref{eq:K1def}. Then, with very high probability, we have
            \begin{align}
              \big|\<\bB,\bK^{-2}\>- \<\bB_0,\bK_1^{-2}\>\big|\le C\, n^{-c_0}\, .
              \end{align}
            \end{lemma}
            \begin{proof}
              Throughout this proof, we work under events $\cE_1\cap\cE_2$ defined in the proof of Lemma \ref{lemma:ApproxMV}.
              Recall that $\max_{i\le n}|D_i|$ is bounded (see, e.g., Eq.~\eqref{eq:BoundD}), whence
              \begin{align}
                |B_{ij}|\le \frac{C}{d^2}\big| \<\bx_i,\bSigma\bx_j\>\big| \le \begin{cases}
                  C/d & \mbox{ if $i=j$,}\\
                  C(\log d)^{1/2}/d^{3/2} & \mbox{ if $i\neq j$,}
                \end{cases}
              \end{align}
              whence $\|\bB\|_F\le C\, \sqrt{(\log d)/d}$. Using Lemma \ref{lemma:ApproxK}, we have
              \begin{align}
                \big|\<\bB,\bK^{-2}\>-\<\bB,\bK_1^{-2}\>\big|& \le \|\bB\|_Fn^{1/2}\|\bK^{-2}-\bK_1^{-2}\|_{\op}\nonumber\\
                                                             & \le C \sqrt{(\log d)/d}\times n^{1/2} [\lambda_{\min}(\bK)\wedge  \lambda_{\min}(\bK_1)]^{-3}\|\bK-\bK_1\|_{\op}\label{eq:BLinKernel}\\
                                                             &\le C \sqrt{\log d}\|\bK-\bK_1\|_{\op}\le C\,n^{-c_0}\, .  \nonumber
              \end{align}

              Using again  Lemma \ref{lemma:ApproxK} together with
              Lemma \ref{lemma:BB0}, we obtain that the following
              holds with very high probability:
              \begin{align*}
                \Big|\<\bB,\bK_1^{-2}\>- \<\bB_0,\bK_1^{-2}\>\Big|&\le \lambda_{\min}(\bK_1)^{-2}
                                                                                                     \Big\|\bB-\bB_0\Big\|_*\\
                                                                                                   & \le \frac{C}{d^{1/2}} \, .
              \end{align*}
               The desired claim follows from this display alongside Eq.~\eqref{eq:BLinKernel}.
            \end{proof}

\begin{lemma}\label{lemma:AK}
  Under the assumptions of Theorem \ref{thm:npropd}, let $\ba$ be defined as in Lemma~\ref{lemma:ApproxMV}.
Then, with very high probability we have
            \begin{align}
              0\le  \<\ba,\bK^{-2}\ba\>\le \frac{C}{n}\, .
              \end{align}
            \end{lemma}
            \begin{proof}
              Notice that the lower bound is trivial since $\bK$ is positive semidefinite.
              We will write
              \begin{align}
                \bK &= \alpha\, \bfone\bfone^{\tr}+\bK_*\, ,\\
                \ba &= h(0)\bfone + \tba\, .
              \end{align}
              By standard bounds on the norm of matrices with i.i.d. rows (and using $\|\bSigma\|_{\op}\le C$), we have
              $0\preceq \bX\bX^\tr/d\preceq C\,\id$, with probability at least $1-C\exp(-n/C)$.
              Therefore, by Lemma \ref{lemma:ApproxK}, and since $\beta\gamma>0$ is bounded away from zero by assumption,
              with very high probability we have $C^{-1}\id \preceq \bK_*\preceq C\id$, for a suitable constant $C$.
              Note that $\tba = (\ba_0-h(0)\bfone)+\ba_1$. 
              Under event $\cE_1\cap \cE_2$, the following holds by smoothness of $h$:
              \begin{align}
                \|\ba_0-h(0)\bfone\|_{\infty} = \max_{i\le d}\left|\E \left\{h\Big(\sqrt{\frac{Q_{ii}}{d}} G\Big)- h(0)\right\}\right|\le \frac{C}{d}\, .
              \end{align}
              On the other hand, recalling the definition of $\ba_1$ in Eq.~\eqref{eq:A1def}, we have, always on $\cE_1\cap\cE_2$,
              \begin{align}
               \|\ba_{1}\|_{\infty}& \le C \frac{1}{d^{3/2}}  \max_{i\le d}Q^{3/2}_{ii}\times
                                     d\times \max_{i\le n}\frac{\|\bSigma^{1/2} \bx_i\|_{\infty}^3}{\|\bSigma^{1/2} \bx_i\|_2^3}\\
                &\le C \frac{1}{d^{3/2}} \times
                  d\times \Big(\frac{\log d}{d}\Big)^{3/2}\le C\frac{(\log d)^{3/2}}{d^2}\, .
              \end{align}
              Therefore we conclude that $\|\tba\|_{\infty}\le C/d$, whence $\|\tba\|_2\le C/\sqrt{d}$.

              We therefore obtain, again using Lemma \ref{lemma:ApproxK},
              \begin{align}
                \Big|\<\ba,\bK^{-2}\ba\> -h(0)^2\<\bfone,\bK^{-2}\bfone\>-2h(0)\<\bfone,\bK^{-2}\tba\>\Big| =
                \<\tba,\bK^{-2}\tba\>\le
               \lambda_{\min}(\bK)^{-2}\|\tba\|_2^2\le \frac{C}{d}\, .\label{eq:aKa}
              \end{align}
              We are therefore left with the task of controlling the two terms $\<\bfone,\bK^{-2}\bfone\>$ and $\<\tba,\bK^{-2}\bfone\>$.
              We will assume $h(0)\neq 0$ because otherwise there is nothing to control. Since $h$ is a positive semidefinite kernel, this
              also implies $h(0)>0$ and $\alpha\ge h(0)>0$. By an application of the Sherman-Morrison formula, we get
              \begin{align}
                \<\bfone,\bK^{-2}\bfone\> & = \<\bfone,(\bK_*+\alpha\bfone\bfone^{\tr})^{-2}\bfone\>\\
                                          & = \frac{\<\bfone,\bK_*^{-2}\bfone\>}{(1+\alpha\<\bfone,\bK_*^{-1}\bfone\>)^2}\\
                & \le \frac{1}{\alpha^2} \frac{\<\bfone,\bK_*^{-2}\bfone\>}{\<\bfone,\bK_*^{-1}\bfone\>^2}\le \frac{C}{\alpha^2}\frac{1}{\|\bfone\|^2}\le \frac{C}{d}\, ,
              \end{align}
              where we used the above remark       $C^{-1}\id \preceq \bK_*\preceq C\id$.

              Using again Sherman-Morrison formula,
              \begin{align}
                \<\bfone,\bK^{-2}\tba\> & = \frac{\<\tba,\bK_*^{-2}\bfone\>}{1+\alpha \<\bfone,\bK_*^{-1}\bfone\>}-
                                          \frac{\alpha\<\bfone,\bK_*^{-2}\bfone\>\<\tba,\bK_*^{-1}\bfone\>}{(1+\alpha \<\bfone,\bK_*^{-1}\bfone\>)^2}\, ,\\
          \big| \<\bfone,\bK^{-2}\tba\>\big|         &\le C \frac{\|\tba\|_2\|\bfone\|_2}{\alpha\|\bfone\|_2^2} + \frac{\|\bfone\|_2^3\|\tba\|_2}{\alpha\|\bfone\|_2^4}\\
                                        & \le \frac{C}{d}\, .
              \end{align}
              Using the last two displays in Eq.~\eqref{eq:aKa} yields the desired claim.
            \end{proof}

            \begin{proof}[Proof of Theorem \ref{thm:npropd}: Variance term]
              By virtue of Lemmas \ref{lemma:VarApprox}, \ref{lemma:BKLemma}, \ref{lemma:AK}, we have
               \begin{align}
                 \frac{1}{\sigma^2_{\xi}}\varianceemp &= \<\bB_0,\bK_1^{-2}\>+{\rm Err}(n)\\
                 & = \<\bB_0,(\bK_0+\alpha\bfone\bfone^{\tr})^{-2}\>   +{\rm Err}(n)\, .
               \end{align}
              Here and below we denote by ${\rm Err}(n)$ an error term bounded as  $ |{\rm Err}(n)|\le Cn^{-c_0}$ with very high probability, and we defined
               \begin{align}
                 \bK_0  &:= \beta \frac{\bX\bX^\tr}{d} + \beta \gamma \bI_n \, .
               \end{align}
               By an application of the Sherman-Morrison formula, and recalling that $\beta\gamma>0$ is bounded away from zero, we get
               \begin{align}
                 \frac{1}{\sigma_{\xi}^2}\varianceemp =&\trace(\bB_0\bK_0^{-2})-
                                                  \frac{2\alpha}{1+\alpha A_1}\, \trace(\bB_0\bK_0^{-2}\bfone\bfone^{\tr}\bK_0^{-1})\\
                                                 &+ \frac{\alpha^2A_2}{(1+\alpha A_1)^2}\, \trace(\bB_0\bK_0^{-1}\bfone\bfone^{\tr}\bK_0^{-1}) +{\rm Err}(n)\, ,
               \end{align}
               where $A_\ell:= \<\bfone,\bK_0^{-\ell}\bfone\>$, $\ell\in\{1,2\}$. By standard bounds on the norm of matrices with i.i.d. rows
               (and using $\|\bSigma\|_{\op}\le C$), we have
               $0\preceq \bX\bX^\tr/d\preceq C\,\id$. Therefore  $C^{-1}\id \preceq \bK_0\preceq C\id$, for a suitable constant $C$,
               with very high probability. This implies  $d/C \le A_\ell\le Cd$ for $\ell\in\{1,2\}$ and some constant $C>0$.
               Further $\|\bB_0\|_{\op}\le C\|\bX\|_{\op}^2/d^2\le C/d$.
               Therefore, (since $\alpha>0$):
               \begin{align}
                 \left|  \frac{1}{\sigma_{\xi}^2}\varianceemp -\trace(\bB_0\bK_0^{-2})\right| &
                                                                                                    \le
                                                                                                  \frac{C}{d}\, \big|\< \bfone,\bK_0^{-1}\bB_0\bK_0^{-2}\bfone\>\big|+\frac{C}{d}\,
                                                                                                  \<\bfone,\bK_0^{-1}\bB_0\bK_0^{-1}\bfone\> +{\rm Err}(n)\\
                 &\le \frac{C}{d}+{\rm Err}(n)\, .
               \end{align}

               We are therefore left with the task of evaluating the asymptotics of
               \begin{align}
             \trace(\bB_0\bK_0^{-2})       = 
                 \trace\big(\bX\bSigma\bX^{\tr}  (\bX\bX^\tr+\gamma d\bI_n)^2\big)\, .
               \end{align}
               However, this is just the variance of ridge regression with respect to the simple features $\bX$, with ridge regularization
               proportional to $\gamma$. We apply the results of \cite{hastie2019surprises} to obtain the claim.
             \end{proof}

      \subsubsection{Proof of Theorem \ref{thm:npropd}: Bias term}
      
      We recall the decomposition
      \begin{align}
        f^*(\bx) = b_0 +\<\bbeta_0,\bx\> +f^*_{\sNL}(\bx)=: f^*_{\sL}(\bx)+ f^*_{\sNL}(\bx)\, ,
      \end{align}
      where $b_0$, $\bbeta_0$ are defined by the orthogonality conditions  $\E\{f^*_{\sNL}(\bx)\}=\E\{\bx f^*_{\sNL}(\bx)\}=0$.
      This yields $b_0=\E\{f^*(\bx)\}$ and $\bbeta_0 = \bSigma^{-1}\E\{f^*(\bx)\bx\}$. 
      We denote by $\bff^* = (f^*(\bx_1),\dots,f^*(\bx_n))^{\sT}$ the vector of noiseless responses, which we correspondingly decompose as
      $\bff^* = \bff^*_{\sL}+\bff^*_{\sNL}$. Recalling the definition of $\bM$, $\bv$ in Eqs.~\eqref{eq:Mdef}, \eqref{eq:Vdef}, the bias reads
      \begin{align}
        \biassquaredemp &= \<\bff^*,\bK^{-1}\bM\bK^{-1}\bff^*\>-2\<\bv,\bK^{-1}\bff^*\>+\|f^*\|_{L^2}^2\, .
      \end{align}
      We begin with an elementary lemma on the norm of $\bff^*$.
      \begin{lemma}\label{lemma:Fnorm}
        Assume $\E\{f^*(\bx)^4\}\le C_0$ for a constant $C_0$ (in particular, this is the case if $\E\{|f^*(\bx)|^{4+\eta}\}\le C_0$). Then,
        there exists a constant $C$ depending uniquely on $C_0$ such that the following hold:
        \begin{enumerate}
        \item[$(a)$] $|b_0|\le C$, $\|\bSigma^{1/2}\bbeta_0\|_2\le C$, $\E\{f^*_{\sNL}(\bx)^2\}\le C$.
        \item[$(b)$] With probability at least $1-Cn^{-1/4}$, we have $|\|\bff^*\|_2^2/n-\|f^*\|_{L^2}^2|\le n^{-3/8}$.
        \item[$(c)$]With probability at least $1-Cn^{-1/4}$, we have $|\|\bff_{\sNL}^*\|_2^2/n-\|f_{\sNL}^*\|_{L^2}^2|\le n^{-3/8}$.
        \end{enumerate}
      \end{lemma}
      \begin{proof}
        By Jensen's inequality we have $\E\{f^*(\bx)^2\}\le C$. By orthogonality of $f^*_{\sNL}$ to linear and constant functions, we
        also have $\E\{f^*(\bx)^2\}=b_0^2+\E\{\<\bbeta_0,\bx\>^2\}+\E\{f^*_{\sNL}(\bx)^2\} = b_0^2+\|\bSigma^{1/2}\bbeta_0\|_2^2+\E\{f^*_{\sNL}(\bx)^2\}$,
        which proves claim $(a)$.

        To prove $(b)$, simply call $Z=\|\bff^*\|_2^2/n-\|f^*\|_{L^2}^2$, and note that $\E\{Z^2\}= (\E\{f^*(\bx)^4\}-\E\{f^*(\bx)^2\}^2)/n\le C/n$.
        The claim follows by Chebyshev inequality.

        Finally, $(c)$ follows by the same argument as for claim $(b)$, once we bound $\|f^*_{\sNL}\|_{L^4}$. In order to show this, notice
        that, by triangle inequality, $\|f^*_{\sNL}\|_{L^4}\le \|f^*\|_{L^4}+\|f_0\|_{L^4}+\|f_1\|_{L^4}$, where $f_0(\bx) = b_0$, $f_1(\bx)=\<\bbeta_0,\bx\>$.
        Since $\bx=\bSigma\bz$, with $\bz$ $C$-sub-Gaussian, $\|f^*_{\sNL}\|_{L^4}\le \|f^*\|_{L^4}+b_0+C\|\bSigma^{1/2}\bbeta_0\|_2\le C$.
      \end{proof}
      
      \begin{lemma}\label{eq:Bias0}
        Under the assumptions of Theorem \ref{thm:npropd}, let $\bM_0$, $\bv_0$ be defined as in the statement of Lemma
        \ref{lemma:ApproxMV}. Then, with probability at least $1-Cn^{-1/4}$, we have
      \begin{align}
        \big|\biassquaredemp - \biassquaredemp_0\big|&\le \frac{C\log d}{\sqrt{d}}\, ,\\
        \biassquaredemp_0&:= \<\bff^*,\bK^{-1}\bM_0\bK^{-1}\bff^*\>-2\<\bv_0,\bK^{-1}\bff^*\>+\|f^*\|_{L^2}^2\
      \end{align}
    \end{lemma}
    \begin{proof}
      We have
      \begin{align}
        \big|\biassquaredemp - \biassquaredemp_0\big|& \le \big|  \<\bff^*,\bK^{-1}(\bM-\bM_0)\bK^{-1}\bff^*\>\big|+
                                                       2\big| \<\bv-\bv_0,\bK^{-1}\bff^*\>\big|\\
                                                     &\le \|\bM-\bM_0\|_F \|\bK^{-1}\bff^*\|^2_2+ 2 \|\bv-\bv_0\|_2 \|\bK^{-1}\bff^*\|_2\\
                                                     &\le \|\bM-\bM_0\|_F \|\bK^{-1}\|^2_{\op}
                                                       \| \bff^*\|^2_2+ 2 \|\bv-\bv_0\|_2 \|\bK^{-1}\|_{\op}\|\bff^*\|_2\\
                                                     &\le C\frac{\log d}{d^{3/2}} \times n + C\frac{\log d}{d} \times \sqrt{n} \le \frac{C\log d}{\sqrt{d}}\, .
      \end{align}
      Here, in the last line, we used   Lemmas \ref{lemma:ApproxK},    \ref{lemma:ApproxMV} and the fact that $\|\bff^*\|^2\le Cn$ by Lemma \ref{lemma:Fnorm}.
    \end{proof}

    In view of the last lemma, it is sufficient to work with    $\biassquaredemp_0$. We decompose it as
    \begin{align}
      \biassquaredemp_0&=   \biassquaredemp_{\sL}+ \biassquaredemp_{\sNL}+ \biassquaredemp_{\smix}+\|f^*_{\sNL}\|_{L^2}^2\, ,\label{eq:Bias0Decomp}\\
       \biassquaredemp_{\sL}&:= \<\bff_{\sL}^*,\bK^{-1}\bM_0\bK^{-1}\bff_{\sL}^*\>-2\<\bv_0,\bK^{-1}\bff_{\sL}^*\>+\|f^*_{\sL}\|_{L^2}^2\, ,\\
       \biassquaredemp_{\sNL}&:= \<\bff_{\sNL}^*,\bK^{-1}\bM_0\bK^{-1}\bff_{\sNL}^*\>\, ,\\
       \biassquaredemp_{\smix}&:= 2\<\bff_{\sL}^*,\bK^{-1}\bM_0\bK^{-1}\bff_{\sNL}^*\>-2\<\bv_0,\bK^{-1}\bff_{\sNL}^*\>\, .
    \end{align}
    We next show that the contribution of the constant term in $f^*_{\sL}(\bx)$ and $\bM_0$ is negligible.
    \begin{lemma}\label{lemma:DropConstant}
           Under the assumptions of Theorem \ref{thm:npropd}, let
           $\bM_0$, $\bB$, $\bv_0$ be defined as in the statement of
           Lemma~\ref{lemma:ApproxMV}. Further define
           \begin{align}
             R_{\sL}&:= \<\bX\bbeta_0,\bK^{-1}\bB\bK^{-1}\bX\bbeta_0\>-\frac{2h'(0)}{d}\<\bX\bSigma\bbeta_0,\bK^{-1}\bX\bbeta_0\>+\<\bbeta_0,\bSigma\bbeta_0\>\, ,\label{eq:RLdef}\\
       R_{\sNL}&:= \<\bff_{\sNL}^*,\bK^{-1}\bB\bK^{-1}\bff_{\sNL}^*\>\, , \label{eq:RNLdef}\\
       R_{\smix}&:= 2\<\bX\bbeta_0,\bK^{-1}\bB\bK^{-1}\bff_{\sNL}^*\>-\frac{2h'(0)}{d}\<\bX\bSigma\bbeta_0,\bK^{-1}\bff_{\sNL}^*\>\, .
           \end{align}
           Then, with very high probability we have
        \begin{align}
         \big| \biassquaredemp_{\sL}-R_{\sL}\big|\le \frac{C}{n}\, ,\\
          \big| \biassquaredemp_{\sNL}-R_{\sNL}\big|\le \frac{C}{n}\, ,\\
         \big|   \biassquaredemp_{\smix}-R_{\smix}\big|\le \frac{C}{n}\, .
       \end{align}   
    \end{lemma}
    \begin{proof}
    The proof of this lemma is very similar to the one of Lemma \ref{lemma:AK}, and we omit it. 
    \end{proof}

    \begin{lemma}\label{lemma:LastLinear}
      Under the assumptions of Theorem \ref{thm:npropd}, let $\cuB(\bSigma,\bbeta_0)$ be defined as in
      Eq.~\eqref{eq:BiasProportional-2}, and $R_{\sL}$ be defined as in the statement of Lemma
      \ref{lemma:DropConstant}. Let $a\in (0,1/2)$. Then we have, with very high probability
           \begin{align}
            \big| R_{\sL}-\cuB(\bSigma,\bbeta_0)\big| \le C\, n^{-a}\,.
           \end{align}
         \end{lemma}
         \begin{proof}
           Recall the definition of $\bK_1$ in Eq.~\eqref{eq:K1def}. and define $\tR_{\sL}$ as $R_{\sL}$
           (cf. Eq.~\eqref{eq:RLdef}) except with $\bB$ replaced by $\bB_0$ defined in Eq.~\eqref{eq:B0Def}, and
           $\bK$ replaced by $\bK_1$ defined in Eq.~\eqref{eq:K1def}. Namely:
           \begin{align}
             \tR_{\sL}&:= \<\bX\bbeta_0,\bK_1^{-1}\bB_0\bK_1^{-1}\bX\bbeta_0\>-\frac{2h'(0)}{d}\<\bX\bSigma\bbeta_0,\bK_1^{-1}\bX\bbeta_0\>+\<\bbeta_0,\bSigma\bbeta_0\>\, .
           \end{align}
           Letting $\bu= \bX\bbeta_0=\bZ\bSigma^{1/2}\bbeta_0$, note that $\|\bu\|_2\le \|\bZ\|_{\op}\|\bSigma^{1/2}\bbeta_0\|_2\le C\sqrt{n}$ with very high probability (using Lemma \ref{lemma:Fnorm}). We then have
           \begin{align*}
             \big|R_{\sL}-\tR_{\sL}\big|& \le \big|\<\bu,\bK^{-1}\bB\bK^{-1}\bu\>-\<\bu,\bK^{-1}\bB_0\bK^{-1}\bu\>\big|\\
                                       &+   \big|\<\bu,\bK^{-1}\bB_0\bK^{-1}\bu\>-\<\bu,\bK_1^{-1}\bB_0\bK_1^{-1}\bu\>\big|
                                          +\frac{C}{d}\big|\<\bX\bSigma\bbeta_0,\bK^{-1}\bu\>-\<\bX\bSigma\bbeta_0,\bK_1^{-1}\bu\>\big|\\
                                        &=: E_1+E_2+E_3\, .
           \end{align*}
           We bound each of the three terms with very high
           probability:
           \begin{align}
             E_1  &\le \|\bB-\bB_0\|_{\op}\cdot \|\bK^{-1}\|_{\op}^2\cdot\|\bu\|_2^2\le \frac{C}{d^{3/2}}\times C\times Cn\le \frac{C}{n^{1/2}}\, ,\label{eq:E1}\\
             E_2 & \le \big( \|\bB_0\bK^{-1}\bu\|_2+\|\bB_0\bK_1^{-1}\bu\|_2\big)\|\bu\|_2 \|\bK^{-1}-\bK^{-1}_1\|_{\op}\nonumber\\
                  &\le \|\bB_0\|_{\op}\big( \|\bK^{-1}\|_{\op}+\|\bK_1^{-1}\|_{\op}\big)\|\bu\|^2_2 \|\bK^{-1}-\bK^{-1}_1\|_{\op}\label{eq:E2}\\
                  & \le \frac{C}{d}\times C\times Cn\times n^{-c_0} \le C\, n^{-c_0}\, ,\nonumber\\
             E_3 & \le \frac{C}{d}\|\bX\|_{\op}\|\bSigma\bbeta_0\|_2 \|\bu\|_2 \|\bK^{-1}-\bK^{-1}_1\|_{\op}\label{eq:E3}\\
                  & \le \frac{C}{d}\times C\sqrt{n}\times C\times C\sqrt{n} \times Cn^{-c_0} \le Cn^{-c_0}\, .\nonumber
           \end{align}
           Here in Eq.~\eqref{eq:E1} we used Lemma \ref{lemma:ApproxK} and Lemma \ref{lemma:BB0}; in Eq.~\eqref{eq:E2}  Lemma \ref{lemma:ApproxK}
           and the fact that $\|\bB_0\|_{\op}\le C/d$; in Eq.~\eqref{eq:E3}, Lemma \ref{lemma:ApproxK}  and $\|\bX\|_{\op}\le C\sqrt{d}$. 
           Hence we conclude that
           \begin{align}
             \big|R_{\sL}-\tR_{\sL}\big|& \le Cn^{-c_0}\, .\label{eq:RtR}
           \end{align}

           Finally define $\ttR_{\sL}$ as $\tR_{\sL}$, with $\bK_1$ replaced by  $\bK_0= \beta \frac{\bX\bX^\tr}{d} + \beta \gamma \bI_n$.
           \begin{align*}
             \big|\tR_{\sL}-\ttR_{\sL}\big|& \le
                                             \big|\<\bu,(\bK_1^{-1}+\bK_0^{-1})\bB_0(\bK_1^{-1}-\bK_0^{-1})\bu\>\big|
                                             +\frac{C}{d}\big|\<\bX\bSigma\bbeta_0,(\bK^{-1}_1-\bK_0^{-1})\bu\>\big|\\
                                           & =: G_1+G_2\, .
           \end{align*}
           By the Sherman-Morrison formula, for any two vectors $\bw_1,\bw_2\in\reals^n$, we have
           \begin{align}
             \big|\<\bw_1,(\bK_1^{-1}-\bK_0^{-1})\bw_2\>\big| &= \alpha
                                                                \frac{\big|\<\bfone,\bK_0^{-1}\bw_1\>\<\bfone,\bK_0^{-1}\bw_2\>\big|}{1+\alpha\<\bfone,\bK_0^{-1}\bfone\>}\\
             &\le \frac{C}{d}\big|\<\bfone,\bK_0^{-1}\bw_1\>\big|\cdot\big|\<\bfone,\bK_0^{-1}\bw_2\>\big|\, .
           \end{align}
           Further notice that
           \begin{align*}
             |\<\bu,\bK^{-1}_0,\bfone\>| = |\<\bbeta_0,\bX^{\tr}(\beta\bX\bX^{\tr}/d+\beta\gamma\bI_n)^{-1}\bfone\>|\le C\sqrt{d\log d}\,,
           \end{align*}
           where the last inequality holds with very high probability by \cite[Theorem 3.16]{knowles2017anisotropic} (cf. also Lemma 4.4 in the same paper).
           We therefore have
           \begin{align}
             G_1&\le\frac{C}{d} \big|\<\bu, (\bK_1^{-1}+\bK_0^{-1})\bB_0\bK_0^{-1}\bfone\>\big|\cdot
                  \big|\<\bu,\bK_0^{-1}\bfone\>\big| \\
                & \le \frac{C}{d}\|\bB_0\|_{\op}\|\bu\|_2\|\bone\|_2  \big|\<\bu,\bK_0^{-1}\bfone\>\big| \\
                & \le \frac{C}{d}\times \frac{1}{d} \times \sqrt{d}\times  \sqrt{d}\times \sqrt{d\log d}\le C\sqrt{\frac{\log d}{d}}\, .
           \end{align}
           Analogously
           \begin{align}
             G_2& \le \frac{C}{d^2} \big|\<\bX\bSigma\bbeta_0,\bK_0^{-1}\bfone\>\big|  \cdot
                  \big|\<\bu,\bK_0^{-1}\bfone\>\big|\\
                & \le \frac{C}{d^2}\|\bX\|_{\op}\|\bSigma\bbeta_0\|_2\|\bK_{0}^{-1}\|_{\op}\|\bfone\|_2
                  \big|\<\bu,\bK_0^{-1}\bfone\>\big|\\
             &\le \frac{C}{d^2}\times C\sqrt{d}\times C\times C\times C\sqrt{n} \times C\sqrt{d\log d} \le C\sqrt{\frac{\log d}{d}}\, .
           \end{align}
           Summarizing
           \begin{align}
             \big|\tR_{\sL}-\ttR_{\sL}\big|& \le C\sqrt{\frac{\log d}{d}}\, .\label{eq:tRttR}
           \end{align}

           We are left with the task of estimating $\ttR_{\sL}$ which we rewrite explicitly as
           \begin{align}
             \ttR_{\sL} =\gamma^2\big\|\bSigma^{1/2}(\bX\bX^{\sT}+\gamma\id_n)^{-1}\bbeta_0\big\|_2^2\, .
           \end{align}
           We recognize in this the bias of ridge regression with respect to the linear features $\bx_i$, when the responses are also linear
           $\<\bbeta_0,\bx_i\>$. Using the results of \cite{hastie2019surprises}, we obtain that, for any $a\in (0,1/2)$, the following holds
           with very high probability.
           \begin{align}
            \big| \ttR_{\sL}-\cuB(\bSigma,\bbeta_0)\big| \le C\,
            n^{-c_0}\,. \label{eq:ttR_B}
           \end{align}
           The proof is completed by  using Eqs.~\eqref{eq:RtR}, \eqref{eq:tRttR}, \eqref{eq:ttR_B}.
         \end{proof}

         We next consider the nonlinear term  $R_{\sNL}$, cf.~Eq.~\eqref{eq:RNLdef}.
         \begin{lemma}\label{lemma:LastNL}
           Under the assumptions of Theorem \ref{thm:npropd}, let $\cuV(\bSigma)$ be defined as in
      Eq.~\eqref{eq:VarianceProportional-2}, and $R_{\sNL}$ be defined as in the statement of Lemma
      \ref{lemma:DropConstant}.  Then there exists $c_0>0$ such that, with probability at least $1-Cn^{-1/4}$,
           \begin{align}
             \big| R_{\sNL}-\cuV(\bSigma)\|\proj_{>1}f^*\|_{L^2}^2\big| \le C\, n^{-c_0}\, .
           \end{align}
         \end{lemma}
         \begin{proof}
           Define
           \begin{align}
             \ttR_{\sNL}&:= \<\bff_{\sNL}^*,\bK_0^{-1}\bB_0\bK_0^{-1}\bff_{\sNL}^*\>\\
                        &= \frac{1}{d^2}\<\bff_{\sNL}^*,(\bX\bX^{\tr}/d+\gamma \id_n)^{-1}\bX\bSigma\bX^{\tr}
                          (\bX\bX^{\tr}/d+\gamma \id_n)^{-1}\bff_{\sNL}^*\>\\
             &= \frac{1}{d^2}\<\bff_{\sNL}^*,(\bZ\bSigma\bZ^{\tr}/d+\gamma \id_n)^{-1}\bZ\bSigma^2\bZ^{\tr}
                          (\bZ\bSigma\bZ^{\tr}/d+\gamma \id_n)^{-1}\bff_{\sNL}^*\>\, .
           \end{align}
           By the same argument as in the proof of Lemma \ref{lemma:LastLinear}, we have, with very high probability,
           \begin{align}
             \big|R_{\sNL}-\ttR_{\sNL}\big|&\le C\sqrt{\frac{\log d}{d}}\, .\label{eq:RNL-firstBound}
           \end{align}

           We next use the following identity, which holds for any two symmetric matrices $\bA$, $\bM$, and any $t\neq 0$,
           \begin{align}
             \bA^{-1}\bM\bA^{-1} = \frac{1}{t}\big[\bA^{-1}-(\bA+t\bM)^{-1}\big]+
             t \bA^{-1}\bM\bA^{-1}\bM(\bA+t\bM)^{-1}\, .
           \end{align}
           Therefore, for any matrix $\bU$ and any $t>0$, we have
           \begin{align}
             \big|\<\bA^{-1}\bM\bA^{-1},\bU\>\big| \le  \frac{1}{t}\big|\<\bA^{-1},\bU\>\big|
             + \frac{1}{t} \big|\<(\bA+t\bM)^{-1},\bU\>\big|+
             t \|\bA^{-1}\|_{\op}^2\|\bM\|_{\op}^2\|(\bA+t\bM)^{-1}\|_{\op} \|\bU\|_*\, .\label{eq:ApproxDerivative}
           \end{align}
           We apply this inequality to $\bA =\bZ\bSigma\bZ^{\tr}/d+\gamma \id_n$,
           $\bM =\bZ\bSigma^2\bZ^{\tr}/d$ and $U_{ij} = f^*_{\sNL}(\bx_i)  f^*_{\sNL}(\bx_i) \bfone_{i\neq j}$.
           Note that $\|\bA^{-1}\|_{\op}, \|\bM\|_{\op},\|(\bA+t\bM)^{-1}\|_{\op} \le C$. Further
           $\|\bU\|_*\le 2\|\bff^*_{\sNL}\|_2^2\le Cn$ with probability at least $1-Cn^{-1/4}$ by Lemma \ref{lemma:Fnorm}.
           Finally for any $t\in (0,1)$, by Theorem \ref{thm:ExpResolvent}, the following hold
           with probability at least $1-Cd^{-1/4}$:
           \begin{align}
             \frac{1}{d}\big|\<\bA^{-1},\bU\>\big|\le C\, d^{-1/8}\, ,\;\;\;\;\;\;
             \frac{1}{d} \big|\<(\bA+t\bM)^{-1},\bU\>\big|\le C\, d^{-1/8}\, .
           \end{align}
           Therefore, applying Eq.~\eqref{eq:ApproxDerivative} we obtain
           \begin{align}
             \frac{1}{d}\big|\<\bA^{-1}\bM\bA^{-1},\bU\>\big|\le \frac{1}{t} C\, d^{-1/8}+ Ct\le Cd^{-1/16}\, ,
           \end{align}
           where in the last step we selected $t=d^{-1/16}$. Recalling the definitions of $\bA,\bM,\bU$, we have proved:
           \begin{align}
             \left|\ttR_{\sNL}-\frac{1}{d^2}\sum_{i=1}^n[\bA^{-1}\bM\bA^{-1}]_{ii}f_{\sNL}^*(\bx_i)^2\right|\le C d^{-1/16}\, .
           \end{align}
           We are therefore left with the task of controlling the diagonal terms. Using the results of \cite{knowles2017anisotropic}, we get
           \begin{align}
            \max_{i\le n}\left|[\bA^{-1}\bM\bA^{-1}]_{ii}-\frac{1}{n}\trace(\bA^{-1}\bM\bA^{-1})\right|\le Cn^{-1/8}\,.
           \end{align}
           Further $|\|\bff^*_{\sNL}\|_2^2/n-\|f^*_{\sNL}\|_{L^2}^2|\le Cn^{-1/2}$ with probability at least $1-Cn^{-1/4}$ by Lemma
           \ref{lemma:Fnorm}. Therefore, with probability at least
             $1-Cd^{-1/4}$,
           \begin{align}
             &\left|\ttR_{\sNL}-V_{\mbox{\tiny \rm RR}}\|f_{\sNL}^*\|_{L^2}^2\right|\le C d^{-1/16}\, , \\
             &V_{\mbox{\tiny \rm RR}}:= \frac{1}{d^2}\big\|\bSigma^{1/2}\bX^{\tr}(\bX\bX^{\tr}/d+\gamma \id_n)^{-1}\big\|_F^2 \, .
           \end{align}
           We finally recognize that the term $V_{\mbox{\tiny \rm RR}}$ is just the variance of ridge regression with respect to the linear features $\bx_i$, and using \cite{hastie2019surprises}, we obtain
           \begin{align}
             \left|\ttR_{\sNL}-\cuV(\bSigma)\|f_{\sNL}^*\|_{L^2}^2\right|\le C d^{-1/16}\, .
           \end{align}
           The proof of the lemma is concluded by using the last equation together with Eq.~\eqref{eq:RNL-firstBound}.
         \end{proof}

         \begin{lemma}\label{lemma:LastMIX}
           Under the assumptions of Theorem \ref{thm:npropd},
           $R_{\smix}$ be defined as in the statement of
           Lemma~\ref{lemma:DropConstant}. Then we have, with
           probability at least $1-Cd^{-1/4}$,
           \begin{align}
             \big| R_{\smix}\big| \le C\, n^{-1/16}\, .
           \end{align}
         \end{lemma}
         \begin{proof}
           The proof of this lemma is analogous to the one of Lemma \ref{lemma:LastNL} and we omit it. 
         \end{proof}

         We are now in a position to prove Theorem~\ref{thm:npropd}.

         \begin{proof}[Proof of Theorem \ref{thm:npropd}: Bias term]
           Using Lemma \ref{eq:Bias0}, Eq.~\eqref{eq:Bias0Decomp} and Lemma \ref{lemma:DropConstant}, we obtain
           that, with very high probability,
           \begin{align}
             \big| \biassquaredemp-(R_{\sL}+R_{\sNL}+R_{\smix}+\|f^*_{\sNL}\|_{L^2}^2) \big|\le C\sqrt{\frac{\log n}{n}}\, .
           \end{align}
           Hence the proof is completed by using Lemmas \ref{lemma:LastLinear}, \ref{lemma:LastNL}, \ref{lemma:LastMIX}.
         \end{proof}
         
         \subsubsection{Consequences: Proof of Corollary 4.14}

         We denote by 
         $\lambda_1\ge \dots \ge\lambda_d$ the eigenvalues of $\bSigma$ in   decreasing order.
         
First  note that the left hand side of Eq.~\eqref{eq:FixedPointProportional-2} is strictly increasing in $\lambda_*$, while the right
hand side is strictly decreasing. By considering the limits as $\lambda_*\to 0$ and $\lambda_*\to\infty$, it is easy to
see that this equation admits indeed a unique solution.

Next denoting by $F(x) :=\Trace\Big( \bSigma(\bSigma+x\id)^{-1}\Big)$
the function appearing on the right hand side of Eq.~\eqref{eq:FixedPointProportional-2}, we have, for $x\ge c_*\lambda_{k+1}$,
\begin{align}
  F(x) & = \sum_{i=1}^{d}\frac{\lambda_i}{x+\lambda_i}\ge  \sum_{i=k+1}^{d}\frac{\lambda_i}{x+\lambda_i}\\
       & \ge \frac{c_*}{(1+c_*)x}  \sum_{i=k+1}^{d}\lambda_i =:\uF(x)\, .
\end{align}
Let $\ulambda_*$ be the unique non-negative solution  of $n(1-(\gamma/\ulambda_*))=\uF(\ulambda_*)$. Then, the above inequality
implies that whenever $\ulambda_*\ge c_*\lambda_{k+1}$ we have $\lambda_*\ge \ulambda_*$.
Solving explicitly for $\ulambda_*$, we get
\begin{align}
  \frac{(1+c_*)\gamma}{c_*\lambda_{k+1}}+\frac{r_k(\bSigma)}{n}\ge  (1+c_*)\;\;\Rightarrow\;\;
  \lambda_*\ge \gamma+\frac{c_*}{1+c_*}\frac{1}{n}\sum_{i=k+1}^{d}\lambda_i\, .
\end{align}
Next,  we upper bound
\begin{align}
  \Trace\big( \bSigma^2(\bSigma+\lambda_*\id)^{-2}\big)& = \sum_{i=1}^{d}\frac{\lambda_i^2}{(\lambda_i+\lambda_*)^2}\\
  &\le k+\frac{1}{\lambda_*^2}\sum_{i=k+1}^d \lambda_i^2\\
                                                       &\le k + (1+c_*^{-1})^2n^2 \frac{\sum_{i=k+1}^d \lambda_i^2}{(n\gamma/c_*+\sum_{i=k+1}^{d}\lambda_i)^2}\, .\label{eq:TrS2}
\end{align}
If we assume that the right-hand side is less than $1/2$, using Theorem \ref{thm:npropd}, we
obtain that, with high probability,
\begin{align}
  \frac{1}{\sigma_{\xi}^2}\varianceemp\le
  k + (1+c_*^{-1})^2n^2 \frac{\sum_{i=k+1}^d \lambda_i^2}{(n\gamma/c_*+\sum_{i=k+1}^{d}\lambda_i)^2}+n^{-c_0}\, .
\end{align}

Next, considering again Eq.~\eqref{eq:FixedPointProportional-2} and upper bounding the right-hand side, we get
\begin{align}
  n \Big(1-\frac{\gamma}{\lambda_*}\Big) \le k +\frac{1}{\lambda_*} \sum_{i=k+1}^d\lambda_i\, .
\end{align}
Hence, using the assumption that the right hand side of Eq.~\eqref{eq:TrS2} is upper bounded by $1/2$, which implies $k\le n/2$,
we get
\begin{align}
  \lambda_*\le 2\gamma+\frac{2}{n}\sum_{i=k+1}^d\lambda_i\, .
\end{align}
Next consider the formula for the bias term, Eq.~\eqref{eq:BiasProportional-2}. Denoting by $(\beta_{0,i})_{i\le p}$ the
coordinates of $\bbeta_0$ in the basis of the eigenvectors of $\bSigma$, we get
\begin{align}
  \lambda_*^2\<\bbeta_0,(\bSigma+\lambda_*\id)^{-2}\bSigma\bbeta_0\> &= \sum_{i=1}^{d}
                                                                       \frac{\lambda_*^2\lambda_i\beta_{0,i}^2}{(\lambda_i+\lambda_*)^2}\\
                                                                     &\le \lambda_*^2\sum_{i=1}^{k}\lambda_i^{-1}\beta_{0,i}^2+ \sum_{i=1}^{d}\lambda_i\beta_{0,i}^2\\
  &\le 4\Big(\gamma+\frac{1}{n}\sum_{i=k+1}^d\lambda_i\Big)^2 \|\bbeta_{0,\le k}\|_{\bSigma^{-1}}^2+\|\bbeta_{0,>k}\|_{\bSigma}^2\, .
\end{align}
Together with Theorem \ref{thm:npropd}, this implies the desired bound on the bias.

\section{Optimization in the linear regime}

{	
\renewcommand{\thetheorem}{\ref{thm:Linearization}}

\begin{theorem} 
  Assume
  \begin{align}
    \Lip(\bD f_n)\, \|\by-f_n(\btheta_0)\|_{2}< \frac{1}{4} \sigma^2_{\min}(\bD f_n(\btheta_0))\, .
    \label{eq:ConditionLinearization-2}
  \end{align}
  Further define $$\sigma_{\max}:=\sigma_{\max}(\bD f_n(\btheta_0)), \sigma_{\min}:=\sigma_{\min}(\bD  f_n(\btheta_0)).$$
  Then the following hold for all $t>0$:
  \begin{enumerate}
  \item The empirical risk decreases exponentially fast to $0$, with  rate $\lambda_0 = \sigma^2_{\min}/(2n)$:
    \begin{align}
      \hRisk(\btheta_t)\le \hRisk(\btheta_0) \, e^{-\lambda_0 t}\, .
      \label{eq:R-Linearization-2}
    \end{align}
  \item The parameters stay close to the initialization and are closely tracked by those of the linearized flow.
    Specifically, letting $L_n:=\Lip(\bD f_n)$,
  \begin{align}
    \|\btheta_t-\btheta_0\|_2& \le \frac{2}{\sigma_{\min}}\, \|\by-f_n(\btheta_0)\|_2\, ,\label{eq:CloseToInit-2}\\
    \|\btheta _t-\obtheta_t\|_2 &\le
      \Big\{\frac{32\sigma_{\max}}{\sigma^2_{\min}}
      \|\by-f_n(\btheta_0)\|_{2}+ \frac{16L_n}{\sigma^3_{\min}}
      \|\by-f_n(\btheta_0)\|^2_{2}\Big\} \notag\\*
    & \qquad\qquad {}
      \wedge
      \frac{180L_n\sigma_{\max}^2}{\sigma^5_{\min}}
      \|\by-f_n(\btheta_0)\|^2_{2}
\,.                                  \label{eq:Coupling-2}
  \end{align}
\item The models constructed by gradient flow and by the linearized flow are similar on test data.
  Specifically, writing $f^{\slin}(\btheta) = f(\btheta_0)+\bD f(\btheta_0) (\btheta-\btheta_0)$, we have  
  \begin{align}
      \lefteqn{\|f(\btheta_t)
      -f^{\slin}(\obtheta_t)\|_{L^2(\P)}}
      & \notag\\*
      & \le \Big\{4\, \Lip(\bD f)\frac{1}{\sigma_{\min}^2}+
    180\|\bD
    f(\btheta_0)\|_{\op}\frac{L_n\sigma_{\max}^2}{\sigma_{\min}^5}\Big\}\|\by-f_n(\btheta_0)\|_2^2\, .  \label{eq:ModelDeviationLinearized-2}
  \end{align}
\end{enumerate}
\end{theorem}
	\addtocounter{theorem}{-1}	
	}							
\begin{proof}
  Throughout the proof we let $L_n:=\Lip(\bD f_n)$, and we
  use $\dot{\ba}_t$ to denote the derivative of quantity $\ba_t$ with respect to time.

  Let $\by_t = f_n(\btheta_t)$. By the gradient flow equation,
  \begin{align}
    \dot{\by}_t = \bD f_n(\btheta_t)\, \dot{\btheta}_t =
    -\frac{1}{n}\bD f_n(\btheta_t)\bD f_n(\btheta_t)^{\sT}(\by_t-\by)\, .
  \end{align}
  Defining the empirical kernel at time $t$, $\bK_t:= \bD f_n(\btheta_t)\bD f_n(\btheta_t)^{\sT}$,
  we thus have
  \begin{align}
    \dot{\by}_t &=   -\frac{1}{n}\bK_t(\by_t-\by)\, ,\label{eq:YtODE}\\
    \frac{\de\phantom{t}}{\de t}\|\by_t-\by\|_2^2 &= -\frac{2}{n} \<\by_t-\by,\bK_t(\by_t-\by)\>\, .
 \end{align}
 Letting $r_*:= \sigma_{\min}/(2L_n)$ and $t_*:=\inf\{t:\; \|\btheta_t-\btheta_0\|_2>r_*\}$,
 we have $\lambda_{\min}(\bK_t)\ge (\sigma_{\min}/2)^2$ for all $t\le t_*$, whence
 \begin{align}
   t\le t_* \;\; \Rightarrow\;\;  \|\by_t-\by\|_2^2 \le \|\by_0-\by\|_2^2e^{-\lambda_0 t}\, ,\label{eq:ResidualNorm}
 \end{align}
 with $\lambda_0 =\sigma_{\min}^2/(2n)$.

  Note that, for any $t\le t_*$, $\sigma_{\min}(\bD f_n(\btheta_t))\ge \sigma_{\min}/2$.
 Therefore, by the gradient flow equations, for any $t\le t_*$,
 \begin{align}
   \|\dot\btheta_t\|_2&= \frac{1}{n}\big\|\bD f_n(\btheta_t)^{\sT}(\by_t-\by)\big\|_2\, ,\\
   \frac{\de\phantom{t}}{\de t} \|\by_t-\by\|_2& =- \frac{1}{n}\cdot \frac{\|\bD f_n(\btheta_t)^{\sT}(\by_t-\by)\big\|_2^2}{\|\by_y-\by\|_2}\\
   &\le -\frac{\sigma_{\min}}{2n} \|\bD f_n(\btheta_t)^{\sT}(\by_t-\by)\big\|_2\, .
 \end{align}
Therefore, by Cauchy-Schwartz,
 \begin{align}\label{eq:CS-DistanceFromInit}
   \frac{\de\phantom{t}}{\de t}\Big(\|\by_t-\by\|_2+\frac{\sigma_{\min}}{2}\|\btheta_t-\btheta_0\|_2\Big)\le
   \frac{\de\phantom{t}}{\de t} \|\by_t-\by\|_2 +\frac{\sigma_{\min}}{2}\|\dot\btheta_t\|_2\le 0\, .
 \end{align}
 This implies, for all $t\le t_*$, 
 \begin{align}
 \|\btheta_t-\btheta_0\|_2\le \frac{2}{\sigma_{\min}} \|\by-\by_0\|_2\, . \label{eq:BoundDistanceFromInit}
 \end{align}
 Assume by contradiction $t_*<\infty$. The last equation together with the assumption \eqref{eq:ConditionLinearization-2}
 implies $\|\btheta_{t_*}-\btheta_0\|_2<r_*$, which contradicts the definition of $t_*$. We conclude that $t_*=\infty$,
 and Eq.~\eqref{eq:R-Linearization-2} follows from Eq.~\eqref{eq:ResidualNorm}.

 Equation \eqref{eq:CloseToInit-2} follows from Eq.~\eqref{eq:BoundDistanceFromInit}.

 In order to prove Eq.~\eqref{eq:Coupling-2}, let $\oby_t:= f_n(\btheta_0)+\bD f_n(\btheta_0)(\obtheta_t-\btheta_0)$.
 Note that this satisfies an equation similar to \eqref{eq:YtODE}, namely
  \begin{align}
    \dot{\oby}_t &=   -\frac{1}{n}\bK_0(\oby_t-\by)\,\label{eq:BYtODE} .
  \end{align}
  Define the difference $\br_t:= \by_t-\oby_t$. We then have $\dot{\br}_t = -(\bK_t/n)\br_t-((\bK_t-\bK_0)/n)(\oby_t-\by)$,
  whence
  \begin{align}
    \frac{\de\phantom{t}}{\de t}\|\br_t\|_2^2&=-\frac{2}{n}\<\br_t,\bK_t\br_t\>-\frac{2}{n}\<\br_t,(\bK_t-\bK_0)(\oby_t-\by)\>\\
    &\le -\frac{2}{n}\lambda_{\min}(\bK_t)\|\br_t\|^2_2+\frac{2}{n}\|\br_t\|_2\big\|\bK_t-\bK_0\big\|_{\op}\|\oby_t-\by\|_2\, .
  \end{align}
  Using $2\lambda_{\min}(\bK_t)/n\ge \lambda_0$ and $\|\oby_t-\by_t\|_2\le \|\by_0-\by\|_2e^{-\lambda_0t/2}$, we get
  \begin{align}
    \frac{\de\phantom{t}}{\de t}\|\br_t\|_2&=-\frac{\lambda_0}{2}\|\br_t\|_2+\frac{1}{n}\big\|\bK_t-\bK_0\big\|_{\op}\|\by_0-\by\|_2
                                             \, e^{-\lambda_0t/2}\, .
  \end{align}
  Note that
  \begin{align}
    \big\| \bK_t-\bK_0\big\|_{\op} &= \big\|\bD f_n(\btheta_t) \bD f_n(\btheta_t)^{\sT}-\bD f_n(\btheta_0) \bD f_n(\btheta_0)^{\sT}\big\|_{\op}\\
                                   & \le 2 \big\|\bD f_n(\btheta_0)\big\|_{\op}\big\|\bD f_n(\btheta_t)-\bD f_n(\btheta_0)\big\|_{\op} +\big\|\bD f_n(\btheta_t)-\bD f_n(\btheta_0)\big\|_{\op}^2\\
    &\le 2\sigma_{\max}L_n\|\btheta_t-\btheta_0\|_{\op}+L_n^2\|\btheta_t-\btheta_0\|^2_{\op}\\
                                     & \le \frac{5}{2}\sigma_{\max}L_n\|\btheta_t-\btheta_0\|_{\op}\, .\label{eq:BoundKdiff}
  \end{align}
  (In the last inequality, we used the fact that $L_n\|\btheta_t-\btheta_0\|_{\op}\le \sigma_{\min}/2$
  by definition of $r_*$.)
  Applying Gr\"onwall's inequality, and using $\br_0=0$, we obtain
  \begin{align}
    \|\br_t\|_2&\le e^{-\lambda_0t/2} \|\by_0-\by\|_2\int_0^t \frac{1}{n}\big\|\bK_s-\bK_0\big\|_{\op}\de s\\
               &\le e^{-\lambda_0t/2} t\|\by_0-\by\|_2\sup_{s\in[0,t]}\frac{1}{n}\big\|\bK_s-\bK_0\big\|_{\op}\\
               &\le e^{-\lambda_0t/4}\frac{2}{\lambda_0} \|\by_0-\by\|_2\sup_{s\ge 0}\frac{1}{n}\big\|\bK_s-\bK_0\big\|_{\op}\\
               & \stackrel{(a)}{\le}  e^{-\lambda_0t/4}\frac{2}{\lambda_0} \|\by_0-\by\|_2\frac{5}{2n}L_n\sigma_{\max}\sup_{s\ge 0}\|\btheta_s-\btheta_0\|_2\\
    & \stackrel{(b)}{\le} e^{-\lambda_0t/4}\frac{2}{\lambda_0} \|\by_0-\by\|_2\frac{5}{2n}L_n\sigma_{\max}\cdot \frac{2}{\sigma_{\min}}\|\by_0-\by\|_2\\
               & \le 20\, e^{-\lambda_0t/4}\frac{\sigma_{\max}}{\sigma_{\min}^3}L_n\|\by-\by_0\|_2^2\, .
  \end{align}
  Here in $(a)$ we used Eq.~\eqref{eq:BoundKdiff} and in $(b)$ Eq.~\eqref{eq:BoundDistanceFromInit}.
  Further using $\|\br_t\|_2\le \|\by_t-\by\|_2+\|\oby_t-\by\|_2\le 2\|\by_0-\by\|\exp(-\lambda_0t/2)$, we get
  \begin{align}
    \|\by_t-\oby_t\|_2&\le 2e^{-\lambda_0t/4}\|\by-\by_0\|_2
                        \Big\{1\wedge \frac{10\sigma_{\max}}{\sigma_{\min}^3}L_n\|\by-\by_0\|_2\Big\}\, .\label{eq:YdiffBound}
  \end{align}
 Recall the gradient flow equations for $\btheta_t$ and $\obtheta_t$:
 \begin{align}
   \dot{\btheta}_t & = \frac{1}{n}\bD f_n(\btheta_t)^{\sT}(\by-\by_t)\, ,\\
   \dot{\obtheta}_t & = \frac{1}{n}\bD f_n(\btheta_0)^{\sT}(\by-\oby_t)\, .
 \end{align}
 Taking the difference of these equations, we get
 \begin{align}
   \frac{\de\phantom{t}}{\de t}\|\btheta_t-\obtheta_t\|_2 &\le \frac{1}{n}\big\|\bD f_n(\btheta_t)-\bD f_n(\btheta_0)\big\|_{\op}
   \|\by_t-\by\|_2+ \frac{1}{n}\big\|\bD f_n(\btheta_0)\big\|_{\op} \|\by_t-\oby_t\|_2\\
                                                          &\le \frac{L_n}{n}\|\btheta_t-\btheta_0\|_2 \|\by_t-\by\|_2+ \frac{\sigma_{\max}}{n}\|\by_t-\oby_t\|_2\\
                                                          &\stackrel{(a)}{\le} \frac{L_n}{n}\cdot \frac{2}{\sigma_{\min}}\, \|\by-\by_0\|_2^2e^{-\lambda_0 t/2}
                                       +\frac{\sigma_{\max}}{n}\cdot 2e^{-\lambda_0t/4}\|\by-\by_0\|_2
                        \Big\{1\wedge \frac{10\sigma_{\max}}{\sigma_{\min}^3}L_n\|\by-\by_0\|_2\Big\}            
 \end{align}
 where in $(a)$ we used Eqs.~\eqref{eq:CloseToInit-2}, \eqref{eq:ResidualNorm} and \eqref{eq:YdiffBound}.
 Integrating the last expression (thanks to $\obtheta_0=\btheta_0$), we get
 \begin{align}
   \|\btheta_t-\obtheta_t\|_2 &\le\frac{8L_n}{\sigma_{\min}^3}\|\by-\by_0\|_2^2+
                                \Big\{\frac{16\sigma_{\max}}{\sigma_{\min}^2}\|\by-\by_0\|_2\wedge\frac{160\sigma^2_{\max}}{\sigma_{\min}^5}       L_n\|\by-\by_0\|_2^2\Big\}\, .
 \end{align}
 Simplifying, we get Eq.~\eqref{eq:Coupling-2}.

 Finally, to prove Eq.~\eqref{eq:ModelDeviationLinearized-2}, write
 \begin{align}
   \|f(\btheta_t)-f_{\slin}(\obtheta_t)\|_{L^2}\le    \underbrace{\|f(\btheta_t)-f_{\slin}(\btheta_t)\|_{L^2}}_{E_1}   +
   \underbrace{\|f_{\slin}(\btheta_t)-f_{\slin}(\obtheta_t)\|_{L^2}}_{E_2}\, .
 \end{align}
 By writing $f(\btheta_t)-f_{\slin}(\btheta_t)= \int_{0}^t \frac{\de\phantom{s}}{\de s} [f(\btheta_s)-f_{\slin}(\btheta_s)]\de s$, we get
 \begin{align}
   E_1&= \left\| \int_0^t [\bD f(\btheta_s)-\bD f(\btheta_0)] \dot{\btheta}_s\de s\right\|_{L^2}\\
      &\le \Lip(\bD f)\sup_{s\ge 0}\|\btheta_s-\btheta_0\|_2\int_0^t\|\dot{\btheta}_s\|_2\de s\\
   &\le \Lip(\bD f)\cdot\frac{4\|\by-\by_0\|_2^2}{\sigma_{\min}^2}\, .
 \end{align}
 In the last step we used Eq.~\eqref{eq:CloseToInit-2} and noted that the same argument to prove the latter indeed
 also bounds the integral $\int_0^t\|\dot{\btheta}_s\|_2\de s$ (see Eq.~\eqref{eq:CS-DistanceFromInit}).

 Finally, to bound term $E_2$, note that $f_{\slin}(\btheta_t)-f_{\slin}(\obtheta_t) =\bD f(\btheta_0)(\btheta_t-\obtheta_t)$,
 and using Eq.~\eqref{eq:Coupling-2}, we get
 \begin{align}
   E_2\le 180\|\bD f(\btheta_0)\|_{\op}\frac{L_n\sigma_{\max}^2}{\sigma_{\min}^5}\|\by-\by_0\|_2^2\, .
 \end{align}
 Equation \eqref{eq:ModelDeviationLinearized-2} follows by putting together the above bounds for $E_1$ and $E_2$.
\end{proof}

We next pass to the case of two-layers networks:
\begin{align}
  f(\bx;\btheta) := \frac{\alpha}{\sqrt{m}}\sum_{j=1}^m b_j\sigma(\<\bw_j,\bx\>),~~
  \btheta= (\bw_1,\dots,\bw_m)\, .\label{eq:Two-Layers-Net-2}
\end{align}
  {	
	 \renewcommand{\thetheorem}{\ref{lemma:TwoLayers}}
\begin{lemma}
  Under Assumption 5.2, 
  further assume
 $\{(y_i,\bx_i)\}_{i\le n}$ to be i.i.d. with $\bx_i \sim_{iid}\normal(0,\id_d)$,  and  $y_i$ $B^2$-sub-Gaussian.
 Then there exist constants $C_i$,
  depending uniquely on $\sigma$, such that the following hold with probability at least $1-2\exp\{-n/C_0\}$,
  provided $md\ge C_0 n\log n$, $n\le d^{\ell_0}$ (whenever not specified, these hold for both
  $\btheta_0\in\{\btheta_0^{(1)},\btheta_0^{(2)}\}$):
  \begin{align}
    \|\by-f_n(\btheta^{(1)}_0)\|_2 & \le C_1\big(B+\alpha)\sqrt{n}\, \label{eq:LemmaLin1-2}\\
                         \|\by-f_n(\btheta^{(2)}_0)\|_2 & \le C_1B\sqrt{n}\, , \label{eq:LemmaLin2-2}\\ 
    \sigma_{\min}(\bD f_n(\btheta_0)) &\ge C_2\alpha \sqrt{d}\, , \label{eq:LemmaLin3-2}\\
    \sigma_{\max}(\bD f_n(\btheta_0)) &\le C_3\alpha \big(\sqrt{n}+\sqrt{d}\big)\, , \label{eq:LemmaLin4-2}\\
    \Lip(\bD f_n) & \le C_4\alpha\sqrt{\frac{d}{m}}\big(\sqrt{n}+\sqrt{d}\big)\, . \label{eq:LemmaLin5-2}
  \end{align}
  Further
  \begin{align}
    \|\bD f(\btheta_0)\|_{\op}& \le C_1'\alpha\, , \label{eq:LemmaLin6-2}\\
    \Lip(\bD f) & \le    C_4' \alpha\sqrt{\frac{d}{m}}\, . \label{eq:LemmaLin7-2}
    \end{align}
  \end{lemma}
  	\addtocounter{theorem}{-1} 
  }
  \begin{proof}
    Since the $y_i$ are $B^2$ sub-Gaussian, we have $\|\by\|_2\le C_1B\sqrt{n}$ with the stated probability.
    Equation~\eqref{eq:LemmaLin2-2} follows since by construction $f_n(\btheta^{(2)}_0)=0$.

    For Eq.~\eqref{eq:LemmaLin1-2} we claim that $\|f_n(\btheta_0^{(1)})\|_2\le C_1\alpha\sqrt{n}$ with the claimed probability.
    To show this, it is sufficient of  course to consider $\alpha=1$. Let $F(\bX,\bW):= \|f_n(\btheta_0^{(1)})\|_2$, where
    $\bX\in\reals^{n\times d}$ contains as rows the vectors $\bx_i$, and $\bW$ the vectors $\bw_i$.
    We also write
    $\btheta_0^{(1)}=\btheta_0$ for simplicity. We
    have
    \begin{align}
      \E\{F(\bX,\bW)\}^2&\le \E\{\|f_n(\btheta_0)\|_2^2\} = n \E\{f(\bx_1;\btheta_0)^2\}\\
                        & = n\Var\{\sigma(\<\bw_1,\bx_1\>)\} \le Cn\, .
    \end{align}
    Next, proceeding as in the proof of \cite[Lemma 7]{oymak2020towards}
    (letting $\bb = (b_j)_{j\le m}$)
    \begin{align*}
      \big|F(\bX,\bW_1)-F(\bX,\bW_2)\big|&\le \frac{1}{\sqrt{m}}\big\|\sigma(\bX\bW_1^{\sT})\bb-\sigma(\bX\bW_2^{\sT})\bb\big\|_2\\
                                         &\le \big\|\sigma(\bX\bW_1^{\sT})-\sigma(\bX\bW_2^{\sT})\big\|_F\\
                                         &\le C \big\|\bX\bW_1^{\sT}-\bX\bW_2^{\sT}\big\|_F\\
                                         &\le C \|\bX\|_{\op}\big\|\bW_1^{\sT}-\bW_2^{\sT}\big\|_F\, .
    \end{align*}
    We  have $\|\bX\|_{\op}\le 2(\sqrt{n}+\sqrt{d})$ with the probability at least $1-2\exp\{-(n\vee d)/C)$ \cite{v-hdp-18}.
  On this event, $F(\bX,\, \cdot\,)$ is $2(\sqrt{n}+\sqrt{d})$-Lipschitz with respect to $\bW$. Recall that the uniform measure on
  the sphere of radius $\sqrt{d}$ satisfies a log-Sobolev inequality with $\Theta(1)$ constant, \cite[Chapter 5]{ledoux2001concentration},
  that the log-Sobolev constant for a product measure is the same as the worst constant of each of the terms.  
  We  then have
  \begin{align}
    \P\big(F(\bX,\bW)\ge \E F(\bX,\bW)+t) \le e^{-dt^2/C(n+d)}+2e^{-(n\vee d)/C} \, .
    \end{align}
    Taking $t=C_1\sqrt{n}$ for a sufficiently large constant $C_1$ implies that the right-hand
    side is at most $2\exp(-(n\vee d)/C)$, which proves the claim.

    Notice that all the following inequalities are homogeneous in
    $\alpha>0$. Hence, we will assume---without loss of
    generality---that $\alpha=1$.
    Equation    \eqref{eq:LemmaLin3-2} follows from \cite[Lemma 4]{oymak2020towards}. Indeed this lemma implies
    \begin{align}
      m\ge \frac{C(n+d)\log n}{d\lambda_{\min}(\bK)} \;\;\;\Rightarrow\;\;\;     \sigma_{\min}(\bD f_n(\btheta_0))\ge
      c_0\sqrt{d\lambda_{\min}(\bK)}\, ,
    \end{align}
    where $\bK$ is the empirical NT kernel
    \begin{align}
      \bK = \frac{1}{d}\E\big\{\bD f_n(\btheta_0)\bD f_n(\btheta_0)\big\} = \big(K_{\NT}(\bx_i,\bx_j)\big)_{i,j\le n}\, .
    \end{align}
    Under Assumption 5.2 
    (in particular $\sigma'$ having non-vanishing Hermite coefficients $\mu_{\ell}(\sigma)$
    for all $\ell\le \ell_0$), and $n\le d^{\ell_0}$, we have $\lambda_{\min}(\bK)\ge c_0$ with the stated probability,
    see for instance \cite{el2010spectrum}. This implies the claim.

    For Eq.~\eqref{eq:LemmaLin4-2}, note that, for any vector $\bv\in \reals^n$, $\|\bv\|_2=1$
    we have
    \begin{align}
      \big\|\bD f_n(\btheta_0)^{\sT}\bv\big\|_2^2&= \frac{1}{m}\sum_{i,j\le n}\sum_{\ell=1}^m
      v_i\sigma'(\<\bw_{\ell},\bx_i\>)   v_j\sigma'(\<\bw_{\ell},\bx_j\>)\<\bx_i,\bx_j\>\\
                                                 &=
                                                 \<\bM,\bX,\bX^{\sT}\>\,,\\
      M_{ij}&:=\frac{1}{m}\sum_{\ell=1}^m
      v_i\sigma'(\<\bw_{\ell},\bx_i\>)   v_j\sigma'(\<\bw_{\ell},\bx_j\>)\, .
    \end{align}
    Since $\bM\succeq 0$, we have
    \begin{align}
      \big\|\bD f_n(\btheta_0)^{\sT}\bv\big\|_2^2& \le \Trace(\bM)\|\bX\|_{\op}^2\\
                                                 &  =\frac{1}{m}\sum_{\ell=1}^m\sum_{i=1}^n  v^2_i\sigma'(\<\bw_{\ell},\bx_i\>)^2  \,\cdot\, \|\bX\|_{\op}^2\\
                                                 &\le B^2\|\bv\|^2_2\|\bX\|_{\op}^2\, .
    \end{align}
    Hence $\sigma_{\max}(\bD f_n(\btheta_0))\le B\|\bX\|_{\op}$ and
    the claim follows from standard estimates of
    operator norms of random matrices with independent entries.

    Equation \eqref{eq:LemmaLin5-2} follows from  \cite[Lemma 5]{oymak2020towards}, which yields
    (after adapting to the different normalization of the $\bx_i$, and using the fact that $\max_{i\le n}\|\bx_i\|_2\le C\sqrt{d}$
    with probability at least $1-2\exp(-d/C)$):
    \begin{align*}
      \big\|\bD f_n(\btheta_1)-\bD f_n(\btheta_2)\big\|_{\op}\le C\sqrt{\frac{d}{m}} \|\bX\|_{\op} \|\btheta_1-\btheta_2\|_2\, .
      \end{align*}
      (Here $\|\btheta_1-\btheta_2\|_2=\|\bW_1-\bW_2\|_F$, where $\bW_i\in\reals^{m\times d}$ is the matrix
      whose rows are the weight vectors.) The claim follows once more by using  $\|\bX\|_{\op}\le 2(\sqrt{n}+\sqrt{d})$ with
      probability at least $1-2\exp\{-(n\vee d)/C)$.

      In order to prove Eq.~\eqref{eq:LemmaLin6-2}, note that, for $h\in L^2(\reals^d,\P)$,
      \begin{align}
        \big\|\bD f(\btheta_0)^* h\big\|_2&= \E\{Q_h(\bx_1,\bx_2)\,P(\bx_1,\bx_2)\}\, ,\\
        Q_h(\bx_1,\bx_2) & :=\frac{1}{m}\sum_{\ell=1}^m\sigma'(\<\bw_{\ell},\bx_1\>) h(\bx_1)
                         \sigma'(\<\bw_{\ell},\bx_2\>) h(\bx_2)\, ,\\
        P(\bx_1,\bx_2)& := \<\bx_1,\bx_2\>\, .
      \end{align}
      Here expectation is with respect to independent random vectors $\bx_1,\bx_2\sim\P$. Denote
      by $\bQ_h$ and $\bP$ the integral operators in  $L^2(\reals^d,\P)$ with kernels $Q_h$ and $P$.
      It is easy to see that $\bP$ is the projector onto the subspace of linear functions, and $\bQ_h$
      is positive semidefinite. Therefore
      \begin{align}
        \big\|\bD f(\btheta_0)^* h\big\|_2&\le \Trace(\bQ_h) =
        \frac{1}{m}\sum_{\ell=1}^m\E\big\{\sigma'(\<\bw_{\ell},\bx_1\>)^2 h(\bx_1)^2\big\}\\
        &\le B^2\|h\|_{L^2}^2\, .
      \end{align}
      This implies $\|\bD f(\btheta_0)\|_{\op}\le B$.

      In order to prove Eq.~\eqref{eq:LemmaLin7-2}, define
      $\Delta_{\ell}(\bx):=\sigma'(\<\bw_{1,\ell},\bx\>)-\sigma'(\<\bw_{2,\ell},\bx\>)$. 
 Let $h\in L^2(\reals^d,\P)$ and note that
      \begin{align}
        \big\|\bD f(\btheta_0)^*h\big\|_2^2
        &= \frac{1}{m}\sum_{\ell=1}^m  \Big\|\E\big\{\bx h(\bx) \Delta_{\ell}(\bx)\big\}\Big\|_2^2\\
        &\le \frac{1}{m}\sum_{\ell=1}^m  \E\big\{\|\bx\|\,  |h(\bx) \Delta_{\ell}(\bx)|\big\}^2\\
        &\le  \frac{1}{m}\sum_{\ell=1}^m  \E\big\{\|\bx\|^2\,  \Delta_{\ell}(\bx)^2\big\} \, \|h\|_{L^2}\, .
      \end{align}
      Note that $|\Delta_{\ell}(\bx)|\le B\, |\<\bw_{1,\ell}-\bw_{2,\ell},\bx\>|$. Using this and the last expression above, we get
      \begin{align}
        \big\|\bD f(\btheta_0)\big\|_{\op}^2&\le  \frac{B^2}{m}\sum_{\ell=1}^m  \E\big\{\|\bx\|^2\,  \<\bx, \bw_{1,\ell}-\bw_{2,\ell}\>^2\big\}\\
        &\le \frac{B^2}{m}(d+2) \sum_{\ell=1}^m  \|\bw_{1,\ell}-\bw_{2,\ell}\|_2^2 = \frac{B^2}{m}(d+2) \|\bW_1-\bW_2\|_F^2\, ,
      \end{align}
      where the second inequality follows from the Gaussian identity $\E\{\|\bx\|^2\bx\bx^{\sT}\} = (d+2)\id_d$.
      This proves Eq.~\eqref{eq:LemmaLin7-2}.
      \end{proof}

	  {	
	  \renewcommand{\thetheorem}{\ref{thm:Two-Layers-Linear}}
      \begin{theorem} 
         Consider the two layer neural network of \eqref{eq:Two-Layers-Net-2}
  under the assumptions of Lemma~\ref{lemma:TwoLayers}.
  Further let $\oalpha :=\alpha/(1+\alpha)$ for initialization $\btheta_0=\btheta_0^{(1)}$ and
  $\oalpha :=\alpha$ for $\btheta_0=\btheta_0^{(2)}$.
  Then there exist constants $C_i$,
  depending uniquely on $\sigma$, such that
 if $md\ge C_0 n\log n$, $d\le n\le d^{\ell_0}$ and
  \begin{align}
    \oalpha\ge C_0 \sqrt{\frac{n^2}{md}}\,
    ,\label{eq:ConditionLinearization2-2}
  \end{align}
  then,
  with probability at least $1-2\exp\{-n/C_0\}$,
  the following hold for all $t\ge 0$.
  \begin{enumerate}
  \item Gradient flow converges exponentially fast to a global
  minimizer. Specifically, letting
    $\lambda_* = C_1\alpha^2d/n$, we have
    \begin{align}
      \hRisk(\btheta_t)\le \hRisk(\btheta_0) \, e^{-\lambda_* t}\, .\label{eq:TrainingConvergence-2}
    \end{align}
  \item The model constructed by gradient flow and linearized flow are similar on test data,
    namely
    \begin{align}
      \|f(\btheta_t) -f_{\slin}(\obtheta_t)\|_{L^2(\P)}\le C_1\left\{\frac{\alpha}{\oalpha^2}\sqrt{\frac{n^2}{md}} +\frac{1}{\oalpha^2}
      \sqrt{\frac{n^5}{md^4}}\right\}\, .\label{eq:ModelDistanceLin-2}
      \end{align}
    \end{enumerate}
  \end{theorem}
  	\addtocounter{theorem}{-1} 
  }
\begin{proof}
  Throughout the proof, we use $C$ to denote constants depending only on $\sigma$, that might change from line to line.
  Using Lemma \ref{lemma:TwoLayers}, the condition \eqref{eq:ConditionLinearization-2} reads
  \begin{align}
    \alpha\sqrt{\frac{dn}{m}}\cdot\frac{\alpha}{\oalpha}\sqrt{n}\le C \big(\alpha\sqrt{d}\big)^2\, .
  \end{align}
  which is equivalent to Eq.~\eqref{eq:ConditionLinearization2-2}. We can therefore apply Theorem \ref{thm:Linearization}.

  Equation \eqref{eq:TrainingConvergence-2} follows from Theorem \ref{thm:Linearization}, point 1,
  using the lower bound on $\sigma_{\min}$ given in Eq.~\eqref{eq:LemmaLin3-2}.

  Equation \eqref{eq:ModelDistanceLin-2} follows from Theorem \ref{thm:Linearization}, point 3, using the
   estimates in Lemma \ref{lemma:TwoLayers}.
 \end{proof}